\documentclass{amsart}

\numberwithin{equation}{section}

\usepackage{my-preamble}

\begin{document}

\title{$n$-abelian and $n$-exact categories}

\date{\today}


\author[G.~Jasso]{Gustavo Jasso} \email{gjasso@math.uni-bonn.de}
\urladdr{https://gustavo.jasso.info}

\address{Graduate School of Mathematics, Nagoya University. Furo-cho,
  Chikusa-ku. 464-8602 Nagoya, Japan.}

\curraddr{Mathematik Zentrum, Universit\"at Bonn, Endenicher Allee 60,
  53115 Bonn, Germany}


\thanks{The author wishes to thank Erik Darp\"o, Laurent Demonet,
  Martin Herschend, Martin Kalck, Julian K\"ulshammer, Boris Lerner,
  Yann Palu and Pierre-Guy Plamondon for motivating conversations
  regarding the contents of this article. This acknowledgment is
  extended to Prof. Osamu Iyama for his encouragement, helpful and
  interesting discussions, his comments on previous versions of this
  article, and especially for his generosity in sharing his ideas
  regarding a first definition of $n$-abelian category. Finally, the
  author wishes to express his sincere thanks to the anonymous referee
  for her/his detailed comments on a previous version of this article;
  in particular, for pointing out an error in earlier formulations
  of \thref{simple-embedding-thm,embedding-thm}.}

\begin{abstract}
  We introduce $n$-abelian and $n$-exact categories, these are analogs
  of abelian and exact categories from the point of view of higher
  homological algebra.  We show that $n$-cluster-tilting subcategories
  of abelian (resp. exact) categories are $n$-abelian
  (resp. $n$-exact). These results allow to construct several examples
  of $n$-abelian and $n$-exact categories. Conversely, we prove that $n$-abelian
  categories satisfying certain mild assumptions can be realized as
  $n$-cluster-tilting subcategories of abelian categories. In analogy
  with a classical result of Happel, we show that the stable category
  of a Frobenius $n$-exact category has a natural $(n+2)$-angulated
  structure in the sense of Gei\ss-Keller-Oppermann. We give several
  examples of $n$-abelian and $n$-exact categories which have appeared
  in representation theory, commutative algebra, commutative and
  non-commutative algebraic geometry.  
\end{abstract}


\maketitle

\tableofcontents

\section{Introduction}
\label{sec:introduction}

Let $n$ be a positive integer. In this article we introduce
$n$-abelian and $n$-exact categories, these are higher analogs of
abelian and exact categories from the viewpoint of higher homological
algebra. Throughout we use the comparative adjective ``higher'' in
relation to the length of exact sequences and \emph{not} in the sense
of higher category theory.

Abelian categories were introduced by Grothendieck in
\cite{grothendieck_sur_1957} to axiomatize the properties of the
category of modules over a ring and of the category of sheaves over a
scheme. It is often the case that interesting additive categories are
not abelian but still have good homological properties with respect
to a restricted class of short exact sequences.  Exact categories
were introduced by Quillen in \cite{quillen_higher_1973} from this
perspective to axiomatize extension-closed subcategories of abelian
categories.

Derived categories play an important role in the study of the
homological properties of abelian and exact categories. Their
properties are captured by the notion of triangulated categories,
introduced by Grothendieck-Verdier in
\cite{verdier_categories_1996}. By a result of Happel, the stable
category of a Frobenius exact category has a natural structure of a
triangulated category, see
\cite[Thm. I.2.6]{happel_triangulated_1988}. Triangulated categories
arising in this way have been called \emph{algebraic} by Keller in
\cite{keller_differential_2006}. Algebraic triangulated categories
have a natural $dg$-enhancement in the sense of Bondal-Kapranov
\cite{bondal_framed_1990}, thus are often considered as a more
reasonable class than that of general triangulated categories.

Recently, a new class of additive categories appeared in
representation theory. The 2-cluster-tilting subcategories were
introduced by Buan-Marsh-Reiten-Reineke-Todorov in
\cite{buan_tilting_2006} as the key concept involved in the additive
categorification of the mutation combinatorics of Fomin-Zelevinsky's
cluster algebras \cite{fomin_cluster_2001} via 2-Calabi-Yau
triangulated categories. It was then observed by Iyama-Yoshino
\cite{iyama_mutation_2008} that the notion of mutation can be extended
to the class of $n$-cluster-tilting subcategories of triangulated
categories.

From a different perspective, $n$-cluster-tilting subcategories of
certain exact categories were introduced by Iyama in
\cite{iyama_higher-dimensional_2007} and further investigated in
\cite{iyama_cluster_2011,iyama_auslander_2007} from the viewpoint of
higher Auslander-Reiten theory. In this theory, the notion of
$n$-almost-split sequence, which are certain exact sequences with
$n+2$ terms, plays an important role.

With motivation coming from these examples in representation theory,
the class of $(n+2)$-angulated
categories was introduced by Gei\ss-Keller-Oppermann as categories
``naturally inhabited by the shadows of exact sequences with $n+2$
terms'', to paraphrase the authors.  We note that the case $n=1$
corresponds to triangulated categories.  Their main source of examples
of $(n+2)$-angulated categories are $n$-cluster-tilting subcategories
of triangulated categories which are closed under the $n$-th power of
the shift functor \cite[Thm. 1]{geiss_n-angulated_2013}. The
properties of $(n+2)$-angulated categories have been investigated by
Bergh-Thaule in
\cite{bergh_axioms_2013,bergh_higher_2013,bergh_grothendieck_2014}.

The aim of this article is to introduce $n$-abelian categories which
are categories inhabited by certain exact sequences with $n+2$ terms,
called $n$-exact sequences. The case $n=1$ corresponds to the
classical concepts of abelian categories. We do so by modifying the
axioms of abelian categories in a suitable manner. We prove several
basic properties of $n$-abelian categories, including the existence of
$n$-pushout (resp. $n$-pullback) diagrams which are analogs of
classical pushout (resp. pullback) diagrams, see \th\ref{n-abelian-cats-have-n-pushout-diagrams}.

An important source of examples of $n$-abelian categories are
$n$-cluster-tilting subcategories. This is made precise by the
following theorem.

\begin{thmu}[see
  \th\ref{recognition-thm-abelian} for details]
  Let $\M$ be an $n$-cluster-tilting subcategory of an abelian
  category. Then $\M$ is an $n$-abelian category.
\end{thmu}

We introduce the notion of projective object in an $n$-abelian
category, and study their properties. Remarkably, projective objects
satisfy the following strong property which is obvious in the case of
abelian categories.

\begin{thmu}
  [see \th\ref{projectives-are-strongly-projective} for details] Let $\M$ be an
  $n$-abelian category and $P\in\M$ a projective object. Then, for
  every morphism $f\colon L\to M$ and every weak cokernel $g\colon
  M\to N$ of $f$, the following sequence is exact:
  \[
    \begin{tikzcd}
      \M(P,L)\rar{?\cdot f}&\M(P,M)\rar{?\cdot g}&\M(P,N).
    \end{tikzcd}  
  \]
\end{thmu}

Using this result, we show that certain $n$-abelian categories can be
realized as $n$-cluster-tilting subcategories of abelian
categories. More precisely, we prove the following theorem.

\begin{thmu}
  [see \th\ref{simple-embedding-thm} for details] Let $\M$ be a small projectively
  generated $n$-abelian category, and $\P$ the category of
  projective objects in $\M$. If $\mod\P$ is injectively cogenerated,
  then $\M$ is equivalent to an $n$-cluster-tilting subcategory of $\mod\P$.
\end{thmu}

After introducing $n$-abelian categories, it is natural to introduce
$n$-exact categories as higher analogs of exact categories. For this,
we modify Keller-Quillen's axioms of exact categories.
We prove that the class of
$n$-exact categories contains that of $n$-abelian categories, see \th\ref{n-abelian-cats-are-n-exact}.
Similarly to the case of $n$-abelian categories, we prove the
following theorem.

\begin{thmu}[see
  \th\ref{recognition-thm-exact}
  for details]
  Let $\M$ be an $n$-cluster-tilting subcategory of an exact
  category. Then $\M$ is an $n$-exact category.
\end{thmu}

We also introduce Frobenius $n$-exact
categories. These are $n$-exact categories with enough projectives and
enough injectives, and such that these two classes of objects
coincide. Frobenius $n$-exact categories are related to $(n+2)$-angulated categories as shown
by the following theorem.

\begin{thmu}[see \th\ref{n-happel-theorem} for details]
  Let $\M$ be a Frobenius $n$-exact category. Then, the stable
  category $\underline{\M}$ has a natural structure of an
  $(n+2)$-angulated category.
\end{thmu}

Finally, we prove the following result also in the direction of
Frobenius $n$-exact categories.

\begin{thmu}[see
  \th\ref{standard-construction-n-exact} for details]
  Let $\M$ be an $n$-cluster-tilting subcategory of a Frobenius exact
  category $\E$, and suppose that $\M$ is closed under taking $n$-th
  cosyzygies. Then, $\M$ is a Frobenius $n$-exact category.
\end{thmu}

This theorem is closely related to the results of
Gei\ss-Keller-Oppermann. The relation
between both approaches to construct $(n+2)$-angulated categories is
explained in \th\ref{standard-construction-n-exact}.

\medskip

Now we explain the notion of 2-exact category with concrete examples. The
first example is due to Herschend-Iyama-Minamoto-Oppermann
\cite{herschend_representation_2014} (see also
\th\ref{n-exact-cat-U} and the example after it).  
Let $K$ be an algebraically closed field, $\coh\PP_K^2$ the
category of coherent sheaves over the projective plane over $K$, and
denote the category of vector bundles over $\PP_K^2$ by
$\vect\PP_K^2$. Note that $\vect\PP_K^2$ is closed under extensions in
$\coh\PP_K^2$, and hence is an exact category. Then, the category 
\[
\U:=\add\setP{\OO(i)}{i\in\ZZ}
\]
of finite direct sums of degree shifts of the structure sheaf on
$\PP_K^2$ is a 2-cluster-tilting subcategory of $\vect\PP_K^2$. In
view of the previous theorem, the category $\U$ is a 2-exact category.
An interesting consequence of the 2-cluster-tilting property is that
for every exact sequence
\begin{equation}
  \label{eq:star}
  \begin{tikzcd}
    0\rar&A\rar&B\rar&C\rar&D\rar&0
  \end{tikzcd}
\end{equation}
with terms in $\U$ the sequences of functors
\[
\begin{tikzcd}[column sep=small]
  0\rar&\Hom(-,A)|_{\U}\rar&\Hom(-,B)|_{\U}\rar&\Hom(-,C)|_{\U}\rar&\Hom(-,D)|_{\U},
\end{tikzcd}
\]
\[
\begin{tikzcd}[column sep=small]
  0\rar&\Hom(D,-)|_{\U}\rar&\Hom(C,-)|_{\U}\rar&\Hom(B,-)|_{\U}\rar&\Hom(A,-)|_{\U}
\end{tikzcd}
\]
are exact. In general, we call a sequence of the form \eqref{eq:star} satisfying these properties a \emph{2-exact
  sequence}.  In this case, the Koszul complexes
\[
\begin{tikzcd}
  0\rar&\OO(i-3)\rar&\OO(i-2)^3\rar&\OO(i-1)^3\rar&\OO(i)\rar&0\quad(i\in\ZZ)
\end{tikzcd}
\]
gives a special class of 2-exact sequences called 2-almost-split
sequences in higher Auslander-Reiten
theory.

Let us provide the reader with another example of a 2-exact
category, following Iyama
\cite[Sec. 2.5]{iyama_higher-dimensional_2007}. 
Let $K$ be an
algebraically closed field and $S:=K\llbracket x_0,x_1,x_2\rrbracket$
be the ring of power series three commuting variables. Also, let $G$
be a finite subgroup of $\SL_{3}(K)$ and $R:=S^G$ the associated
invariant subring of $S$. Finally, we denote the category of
Cohen-Macaulay $R$-modules by $\CM R$, see Section
\ref{sec:isolated-singularities} for details and definitions. Note
that $\CM R$ is a Frobenius exact category. Then, the category
\[
\S:=\add S=\setP{M\in\CM R}{M\text{ is a direct summand of }S^m\text{
    for some }m}
\]
is a 2-cluster-tilting subcategory of $\CM R$, \ie we have
\[
\S=\setP{M\in\CM R}{\Ext_R^1(\S,M)=0}=\setP{M\in\CM
  R}{\Ext_R^1(M,\S)=0}.
\]
An important example of a 2-exact sequence in this case is given by
the Koszul complex of $S$:
\[
\begin{tikzcd}
  K(S)\colon 0\rar&S\rar&S^{3}\rar&S^{3}\rar&S\rar&K\rar&0.
\end{tikzcd}
\]
As a complex of $R$-modules, $K(S)$ is the direct sum of finitely many
2-almost-split sequences and a 2-fundamental sequence in the sense of
\cite[Sec. 3]{iyama_higher-dimensional_2007}.

Let us mention other examples of $n$-cluster-tilting subcategories,
which give us examples of $n$-abelian and $n$-exact categories.

Finite dimensional algebras of finite
global dimension whose module category contains an $n$-cluster-tilting
subcategory are one of the central objects of study of higher
Auslander-Reiten theory. A distinguished class of such algebras, the
so-called $n$-representation-finite algebras, were introduced by
Iyama-Oppermann in \cite{iyama_n-representation-finite_2011} and have
been studied in greater detail by Herschend-Iyama in the case $n=2$,
see \cite{herschend_selfinjective_2011}.

In a parallel direction, 2-cluster-tilting subcategories of the module
category of a preprojective algebra of Dynkin type, which has infinite
global dimension, are central in Gei\ss-Leclerc-Schr\"oer's
categorification of cluster algebras arising in Lie theory, see
\cite{geiss_preprojective_2008} and the references therein.

Further examples of $n$-cluster-tilting subcategories of abelian and
exact categories have been constructed by Amiot-Iyama-Reiten in the
category of Cohen-Macaulay modules over an isolated singularity
\cite{amiot_stable_2011}.

\medskip

Finally, let us give a brief description of the contents of this
article. In Section \ref{sec:preliminary_concepts} we introduce the
basic concepts behind the definitions of $n$-abelian and $n$-exact
categories: $n$-cokernels, $n$-kernels, $n$-exact sequences, and
$n$-pushout and $n$-pullback diagrams (the reader will forgive the
author for his lack of inventiveness in naming these concepts). The
class of $n$-abelian categories is introduced in Section
\ref{sec:n-abelian-categories}, where we also give a characterization
of semisimple categories in terms of $n$-abelian categories. In
\th\ref{recognition-thm-abelian,simple-embedding-thm} we explore the
connection between $n$-abelian categories and $n$-cluster-tilting
subcategories of abelian categories. Later, in Section
\ref{sec:n-exact-categories} we introduce $n$-exact categories and
establish a connection with $n$-cluster-tilting subcategories of exact
categories in \th\ref{recognition-thm-exact}. Frobenius $n$-exact
categories and their main properties are introduced in Section
\ref{sec:frobeinus-n-exact-categorties}. At last, in Section
\ref{sec:examples} we provide several examples to illustrate our
results.

\section{Preliminary concepts}
\label{sec:preliminary_concepts}

We begin by fixing our conventions and notation, and by reminding the
reader of basic concepts in homological algebra that we use freely in
the remainder.

\subsection{Conventions and notation}
\label{sec:conventions}

Throughout this article $n$ always denotes a fixed positive integer.
Let $\C$ be a category; all subcategories considered are supposed to
be full. If $A,B\in\C$, then we denote the set of morphisms $A\to B$
in $\C$ by $\C(A,B)$.  We denote the identity morphism of an object
$C\in\C$ by $1=1_C$.  We denote composition of morphisms by
concatenation: if $f\in\C(A,B)$ and $g\in\C(B,C)$, then
$fg\in\C(A,C)$. If $F\colon\C\to\D$ is a functor, then the
\emph{essential image of $F$} is the full subcategory of $\D$ given by
\[
  F\C:=\setP{D\in\D}{\exists C\in\C\text{ such that }FC\cong D}.
\]

A morphism $e\in\C(A,A)$ is \emph{idempotent} if $e^2=e$. We say that
$\C$ is \emph{idempotent complete} if for every idempotent
$e\in\C(A,A)$ there exist an object $B$ and morphisms $r\in\C(A,B)$
and $s\in\C(B,A)$ such that $rs=e$ and $sr=1_B$.

Let $\C$ be an additive category in the sense of
\cite[Sec. A.4.1]{weibel_introduction_1994}.  If $\X$ is a class of
objects in $\C$, then we denote by $\add\X$ the full subcategory whose
objects are direct summands of direct sums of objects in $\X$.

We denote the category of (cochain) complexes in $\C$ by
$\CC(\C)$. Also, we denote the full subcategory of $\CC(\C)$ given by
all complexes concentrated in non-negative (resp. non-positive)
cohomological degrees by $\CC^{\geq0}(\C)$ (resp. $\CC^{\leq0}(\C)$). For
convenience, we denote by $\CC^n(\C)$ the full subcategory of
$\CC(\C)$ given by all complexes
\[
\begin{tikzcd}
  X^0\rar{d^0}&X^1\rar{d^1}&\cdots\rar{d^{n-1}}&X^n\rar{d^n}&X^{n+1}
\end{tikzcd}
\]
which are concentrated in degrees $0,1,\dots,n+1$. A morphism of
complexes
\[
\begin{tikzcd}
  X\dar{f}&\cdots\rar{d_X^{-1}}&X^0\rar{d_X^0}\dar{f^0}&X^1\rar{d_X^1}\dar{f^1}&X^2\rar{d_X^2}\dar{f^2}&\cdots\\
  Y&\cdots\rar{d_Y^{-1}}&Y^0\rar{d_Y^0}&Y^1\rar{d_Y^1}&Y^2\rar{d_Y^2}&\cdots
\end{tikzcd}
\]
is \emph{null-homotopic} if for all $k\in\ZZ$ there exists a morphism
$h^k\colon X^k\to Y^{k-1}$ such that
\[
f^k=h^kd_Y^{k-1}+d_X^kh^{k+1}.
\]
In this case we say that $h=\tupleP{h^k}{k\in\ZZ}$ is a
\emph{null-homotopy}.  We say that two morphisms of complexes $f\colon
X\to Y$ and $g\colon X\to Y$ are \emph{homotopic} if their difference
is null-homotopic.  A \emph{homotopy} between $f$ and $g$ is a
null-homotopy $h$ of $f-g$ and we write $h\colon f\to g$.  It is
easily verified that being homotopic induces an equivalence relation
on $\CC(\C)(X,Y)$. The \emph{homotopy category of $\C$}, denoted by
$\KK(\C)$, is the category with the same objects as $\CC(\C)$ and in
which morphisms are given by morphisms of complexes modulo homotopy.
For further information on chain complexes and the homotopy category
we refer the reader to \cite[Ch. 1]{weibel_introduction_1994}.

We remind the reader of the notion of functorially finite subcategory
of an additive category.  Let $\C$ be an additive category and $\D$ a (full)
subcategory of $\C$.  We say that $\D$ is \emph{covariantly finite in
  $\C$} if for every $C\in\C$ there exists an object $D\in\D$ and a
morphism $f:C\to D$ such that, for all $D'\in\D$, the sequence of
abelian groups
\[
\begin{tikzcd}
  \C(D,D')\rar{f\cdot?}&\C(C,D')\rar&0
\end{tikzcd}
\]
is exact.  Such a morphism $f$ is called a \emph{left
  $\D$-approximation of $C$}.  The notions of \emph{contravariantly
  finite subcategory of $\C$} and \emph{right $\D$-approximation} are
defined dually.  A \emph{functorially finite subcategory of $\C$} is a
subcategory which is both covariantly and contravariantly finite in
$\C$.  For further information on functorially finite subcategories we
refer the reader to
\cite{auslander_almost_1981,auslander_applications_1991}.

\subsection{$n$-cokernels, $n$-kernels, and $n$-exact sequences}

Let $\C$ be an additive category and $f\colon A\to B$ a morphism in
$\C$.  A \emph{weak cokernel of $f$} is a morphism $g\colon B\to C$
such that for all $C'\in\C$ the sequence of abelian groups
\[
\begin{tikzcd}
  \C(C,C')\rar{g\cdot?}&\C(B,C')\rar{f\cdot?}&\C(A,C')
\end{tikzcd}
\]
is exact.  Equivalently, $g$ is a weak cokernel of $f$ if $fg=0$ and
for each morphism $h\colon B\to C'$ such that $fh=0$ there exists a
(not necessarily unique) morphism $p\colon C\to C'$ such that $h=gp$.
These properties are subsumed in the following commutative diagram:
\[
\begin{tikzcd}
  A\rar{f}\drar[swap]{0}&B\rar{g}\dar{\forall h}
  &C\dlar[dotted]{\exists p}\\&C'
\end{tikzcd}
\]
Clearly, a weak cokernel $g$ of $f$ is a cokernel of $f$ if and only
if $g$ is an epimorphism.  The concept of \emph{weak kernel} is
defined dually.

The following general result, together with its dual, plays a central
role in the sequel.

\begin{comparison-lemma}
  \th\label{comparison-lemma} Let $\C$ be an additive category and
  $X\in\CC^{\geq0}(\C)$ a complex such that for all $k\geq0$ the
  morphism $d_X^{k+1}$ is a weak cokernel of $d_X^k$.  If $f\colon
  X\to Y$ and $g\colon X\to Y$ are morphisms in $\CC^{\geq0}(\C)$ such
  that $f^0=g^0$, then there exists a homotopy $h\colon f\to g$ such
  that $h^1$ is the zero morphism.
\end{comparison-lemma}
\begin{proof}
  Let $u:=f-g$ and for all $k\leq 1$ let $h^k\colon X^k\to Y^{k-1}$ be
  the zero morphism. Note that $u^0=0$ by hypothesis. We proceed by
  induction on $k$. Let $k\geq1$ and suppose that for all $\ell\leq k$
  we have constructed a morphism
  \[
  h^{\ell}\colon X^{\ell}\to Y^{\ell-1}
  \]
  such that
  \[
  u^{\ell-1}=h^{\ell-1}d_Y^{\ell-2}+d_X^{\ell-1}h^{\ell}.
  \]
  Since $u$ is a morphism of complexes, we have
  \begin{align*}
    d_X^{k-1}(u^k-h^kd_Y^{k-1})=&d_X^{k-1}u^k+(h^{k-1}d_Y^{k-2}-u^{k-1})d_Y^{k-1}\\
    =&d_X^{k-1}u^k-u^{k-1}d_Y^{k-1}\\
    =&0.
  \end{align*}
  Hence, given that $d_X^k$ is a weak cokernel of $d_X^{k-1}$, there
  exists a morphism
  \[
  h^{k+1}\colon X^{k+1}\to Y^k
  \]
  such that $u^k-h^kd_Y^{k-1}=d_X^kh^{k+1}$ or, equivalently,
  \[
  u^k=h^kd_Y^{k-1}+d_X^kh^{k+1}.
  \]
  This finishes the construction of the required null-homotopy
  $h\colon f-g\to0$.
\end{proof}

The following terminology will prove convenient in the sequel.

\begin{definition}
  \th\label{def:n-cokernel} Let $\C$ be an additive category and
  $d^0\colon X^0\to X^1$ a morphism in $\C$.  An \emph{$n$-cokernel of
    $d^0$} is a sequence
  \[
  \begin{tikzcd}
    \tuple{d^1,\dots,d^n}\colon
    X^1\rar{d^1}&X^2\rar{d^2}&\cdots\rar{d^n}&X^{n+1}
  \end{tikzcd}
  \]
  such that for all $Y\in\C$ the induced sequence of abelian groups
  \[
  \begin{tikzcd}[column sep=normal]
    0\rar&\C(X^{n+1},Y)\rar{d^{n}\cdot ?}&\C(X^n,Y)\rar{d^{n-1}\cdot
      ?}&\cdots\rar{d^{1}\cdot ?}&\C(X^1,Y)\rar{d^{0}\cdot
      ?}&\C(X^0,Y)
  \end{tikzcd}
  \]
  is exact.  Equivalently, the sequence $\tuple{d^1,\dots,d^n}$ is an
  $n$-cokernel of $d^0$ if for all $1\leq k\leq n-1$ the morphism
  $d^k$ is a weak cokernel of $d^{k-1}$, and $d^n$ is moreover a
  cokernel of $d^{n-1}$.  The concept of \emph{$n$-kernel} of a
  morphism is defined dually.
\end{definition}

\begin{remark}
  If $n\geq2$, then $n$-cokernels are not unique in general. Indeed,
  for each object $C\in\C$ the sequence $0\to C\xto{1}C$ is a
  2-cokernel of the morphism $0\to 0$. This shortcoming can be
  resolved if one considers $n$-cokernels up to isomorphism in
  $\KK(\C)$, see \th\ref{weak-isos-induce-homotopy-equivalences}.
\end{remark}

As explained in the Introduction, $n$-exact sequences, defined below,
are the object of study of higher homological algebra. The
investigation of their properties is our main concern for the rest of
this article.

\begin{definition}
  \th\label{def:n-exact-sequence} Let $\C$ be an additive category. An
  \emph{$n$-exact sequence in $\C$} is a complex
  \[
  \begin{tikzcd}
    X^0\rar{d^0}&X^1\rar{d^1}&\cdots\rar{d^{n-1}}&X^n\rar{d^n}&X^{n+1}
  \end{tikzcd}
  \]
  in $\CC^n(\C)$ such that $\tuple{d^0,\dots,d^{n-1}}$ is an
  $n$-kernel of $d^n$, and $\tuple{d^1,\dots,d^n}$ is an $n$-cokernel
  of $d^0$.
\end{definition}

Let $\C$ be an additive category. We remind the reader that a complex
$X\in\CC(\C)$ is \emph{contractible} if the identity morphism of $X$
is null-homotopic or, equivalently, $X$ is isomorphic to the zero
complex in $\KK(\C)$.  As a first analogy with the classical theory,
let us show that the class of $n$-exact sequences is closed under
isomorphisms in $\KK(\C)$.

\begin{proposition}
  \th\label{n-exact-sequences-are-closed-under-homotopy-equivalence}
  Let $\C$ be an additive category and $X$ and $Y$ be complexes in
  $\CC^n(\C)$ which are isomorphic in $\KK(\C)$.  Then the following
  statements hold.
  \begin{enumerate}
  \item The complex $X$ is an $n$-exact sequence if and only if $Y$ is
    an $n$-exact sequence.
  \item Every contractible complex with $n+2$ terms is an $n$-exact
    sequence.
  \end{enumerate}
\end{proposition}
\begin{proof}
  Note that the second claim follows immediately from the first one
  since the zero complex in $\CC^{n}(\C)$ is clearly an $n$-exact
  sequence.  Suppose that $X$ is an $n$-exact sequence. By hypothesis,
  there exist morphisms of complexes
  \[
  \begin{tikzcd}
    Y\dar{f}&Y^0\rar\dar &Y^1\rar\dar&\cdots\rar&Y^n\rar\dar&Y^{n+1}\dar\\
    X\dar{g}&X^0\rar\dar &X^1\rar\dar&\cdots\rar&X^n\rar\dar&X^{n+1}\dar\\
    Y&Y^0\rar&Y^1\rar&\cdots\rar&Y^n\rar&Y^{n+1}
  \end{tikzcd}
  \]
  together with a homotopy $h\colon fg \to 1_{Y}$. Hence, for all
  $k\in\set{1,\dots,n}$ we have
  \begin{equation}
    \label{eq:in-proof-1-fg}
    1_{Y^k}=f^kg^k-h^kd_Y^{k-1}-d_Y^kh^{k+1}.
  \end{equation}
  In particular, we have
  \begin{equation}
    \label{eq:in-proof-epi-1-fgx}
    1_{Y^{n+1}}=f^{n+1}g^{n+1}-h^{n+1}d_Y^n.    
  \end{equation}
  
  We claim that for all $k\in\set{1,\dots,n}$ the morphism $d_Y^k$ is
  a weak cokernel of $d_Y^{k-1}$. Indeed, let $k\in\set{1,\dots,n}$
  and $u\colon Y^k\to C$ be a morphism such that $d_Y^{k-1}u=0$.  It
  follows that
  \[
  (d_X^{k-1}g^k)u=g^{k-1}(d_Y^{k-1}u)=0.
  \]
  Since $d_X^k$ is a weak cokernel of $d_X^{k-1}$ there exists a
  morphism $v\colon X^{k+1}\to C$ such that $g^ku=d_X^kv$.  Therefore,
  \begin{equation}
    \label{eq:in-proof-fgu}
    f^k(g^ku)=(f^kd_X^k)v=d_Y^kf^{k+1}v.
  \end{equation}
  By composing the identity \eqref{eq:in-proof-1-fg} on the right with
  $u$ and substituting the identity \eqref{eq:in-proof-fgu} we obtain
  \[
  u=(f^kg^k)u-(d_Y^kh^{k+1})u=d_Y^k(f^{k+1}v-h^{k+1}u).
  \]
  Therefore $u$ factors through $d_Y^k$.  This shows that $d_Y^k$ is a
  weak kernel of $d_Y^{k-1}$.

  We need to show that $d_Y^n$ is moreover a cokernel of $d_Y^{n-1}$.
  For this it is enough to show that $d_Y^n$ is an epimorphism for we
  already know that it is a weak cokernel of $d_Y^{n-1}$.  Let
  $w\colon Y^{n+1}\to C$ be a morphism such that $d_Y^nw=0$.  It
  follows that
  \[
  d_X^n(g^{n+1}w)=g^n(d_Y^nw)=0.
  \]
  Given that $d_X^n$ is an epimorphism we deduce that $g^{n+1}w=0$. By
  composing \eqref{eq:in-proof-epi-1-fgx} on the right with $w$, we
  obtain
  \[
  w=f^{n+1}(g^{n+1}w)-h^{n+1}(d_Y^nw)=0.
  \]
  Therefore $d_Y^n$ is an epimorphism.  This shows that
  $\tuple{d_Y^1,\dots,d_Y^n}$ is an $n$-cokernel of $d_Y^0$.  By
  duality, the sequence $\tuple{d_Y^0,\dots,d_Y^{n-1}}$ is an
  $n$-kernel of $d_Y^n$.  Hence $Y$ is an $n$-exact sequence. The
  converse implication is analogous.
\end{proof}

We have the following useful characterization of contractible
$n$-exact sequences.

\begin{proposition}
  \th\label{split-mono-implies-sequence-contracts} Let $\C$ be an
  additive category and $X$ a complex in $\CC^n(\C)$ such that
  $\tuple{d^1,\dots,d^n}$ is an $n$-cokernel of $d^0$.  Then, $d^0$ is
  a split monomorphism if and only if $X$ is a contractible $n$-exact
  sequence.
\end{proposition}
\begin{proof}
  Suppose that $d^0$ is a split monomorphism. Hence there exists a
  morphism $h^1\colon X^1\to X^0$ such that $d^0h^1=1_{X^0}$.  We
  shall extend $h^1$ to a null-homotopy of $1_{X}$.  Inductively, let
  $k\in\set{0,1,\dots,n}$ and suppose that for all $\ell\leq k$ we
  have constructed a morphism $h^\ell\colon X^\ell\to X^{\ell-1}$ such
  that
  \[
  1_{X^{\ell-1}}=h^{\ell-1}d^{\ell-2}+d^{\ell-1} h^\ell.
  \]
  Composing this identity, for $\ell=k$, on the left with $d^{k-1}$ we
  obtain
  \[
  d^{k-1}=(h^{k-1}d^{k-2}+d^{k-1}h^k)d^{k-1}=d^{k-1}(h^kd^{k-1}).
  \]
  Since $d^k$ is a weak cokernel of $d^{k-1}$, there exists a morphism
  $h^{k+1}\colon X^{k+1}\to X^{k}$ such that
  $d^kh^{k+1}=1_{X^k}-h^kd^{k-1}$ or, equivalently,
  \[
  1_{X^k}=h^kd^{k-1}+d^kh^{k+1}.
  \]
  This finishes the induction step.  It remains to show that
  $1_{X^{n+1}}=h^nd^n$.  For this, let $k=n$ and note that composing
  the previous equality on the right by $d^n$ yields
  \[
  d^n=(h^nd^{n-1}+d^nh^{n+1})d^n=d^n(h^{n+1}d^n).
  \]
  Since $d^n$ is an epimorphism, we have $1_{X^{n+1}}=h^nd^n$, which
  is what we needed to show. This shows that $X$ is a contractible
  complex, and so it is also an $n$-exact sequence.  The converse
  implication is obvious.
\end{proof}

The following result implies that $n$-cokernels and $n$-kernels are
unique up to isomorphism in $\KK(\C)$.

\begin{proposition}
  \th\label{weak-isos-induce-homotopy-equivalences} Let $\C$ be an
  additive category and $f\colon X\to Y$ a morphism of $n$-exact
  sequences in $\C$ such that $f^k$ and $f^{k+1}$ are isomorphisms for
  some $k\in\set{1,\dots,n}$. Then, $f$ induces an isomorphism in
  $\KK(\C)$.
\end{proposition}
\begin{proof}
  Using the factorization property of weak cokernels and weak kernels
  we can construct a morphism of $n$-exact sequences $g\colon Y\to X$
  where $g^k$ and $g^{k+1}$ are the inverses of $f^k$ and $f^{k+1}$
  respectively:
  \[
  \begin{tikzcd}[column sep=small]
    X\dar{f}&X^0\rar\dar&\cdots\rar&X^{k-1}\rar\dar&X^k\rar\dar{f^k}&X^{k+1}\rar\dar{f^{k+1}}&X^{k+2}\rar\dar&\cdots\rar&X^{n+1}\dar\\
    Y\dar{g}&Y^0\rar\dar[dotted]&\cdots\rar&Y^{k-1}\rar\dar[dotted]&Y^k\rar\dar{g^k}&Y^{k+1}\rar\dar{g^{k+1}}&Y^{k+2}\rar\dar[dotted]&\cdots\rar&Y^{n+1}\dar[dotted]\\
    X&X^0\rar&\cdots\rar&X^{k-1}\rar&X^k\rar&X^{k+1}\rar&X^{k+2}\rar&\cdots\rar&X^{n+1}
  \end{tikzcd}
  \]
  Then, the \th\ref{comparison-lemma} and its dual applied to diagrams
  \[
  \begin{tikzcd}[column sep=small]
    X^0\rar\dar&\cdots\rar&X^{k-1}\rar\dar&X^k\dar{f^{k+1}}\\
    Y^0\rar\dar[dotted]&\cdots\rar&Y^{k-1}\rar\dar[dotted]&Y^k\dar{g^{k+1}}\\
    X^0\rar&\cdots\rar&X^{k-1}\rar&X^k
  \end{tikzcd}
  \quad\text{and}\quad
  \begin{tikzcd}[column sep=small]
    X^{k+1}\rar\dar{f^{k+1}}&X^{k+2}\rar\dar&\cdots\rar&X^{n+1}\dar\\
    Y^{k+1}\rar\dar{g^{k+1}}&Y^{k+2}\rar\dar[dotted]&\cdots\rar&Y^{n+1}\dar[dotted]\\
    X^{k+1}\rar&X^{k+2}\rar&\cdots\rar&X^{n+1}
  \end{tikzcd}
  \]
  respectively imply that $f$ and $g$ induce mutually inverse
  isomorphisms in the homotopy category $\KK(\C)$.
\end{proof}

\begin{remark}
  The statement of \th\ref{weak-isos-induce-homotopy-equivalences} can
  be interpreted as saying that each morphism in an $n$-exact sequence
  determines the others ``up to homotopy''.  To prove that
  equivalences of $n$-exact sequences also induce isomorphisms in
  $\KK(\C)$ we need to impose a richer structure on the category $\C$,
  see \th\ref{X-closed-under-equivalences}.
\end{remark}

\begin{definition}
  Let $\C$ be an additive category. A \emph{morphism of $n$-exact
    sequences in $\C$} is a morphism of complexes
  \[
  \begin{tikzcd}
    X\dar{f}&X^0\rar\dar{f^0} &X^1\rar\dar{f^1}&\cdots\rar
    &X^n\rar\dar{f^n}&X^{n+1}\dar{f^{n+1}}\\
    Y&Y^0\rar&Y^1\rar&\cdots\rar&Y^n\rar&Y^{n+1}
  \end{tikzcd}
  \]
  in which each row is an $n$-exact sequence. We say that $f$ is an
  \emph{equivalence} if $f^0=1_{X^0}$ and $f^{n+1}=1_{X^{n+1}}$.
\end{definition}

\begin{remark}
  In \th\ref{X-closed-under-equivalences} (in which we consider
  morphisms of complexes up to homotopy) we show that, in the case of
  $n$-exact categories, equivalences of $n$-exact sequences are in
  fact an equivalence relation on the class of $n$-exact sequences.
\end{remark}

\subsection{$n$-pushout diagrams and $n$-pullback diagrams}

Let $\C$ be an additive category.  A pushout diagram of a pair of
morphisms
\[
\begin{tikzcd}
  X\rar{g}\dar{f}&Z\\
  Y
\end{tikzcd}
\]
in $\C$ can be identified with a cokernel of the morphism $[-g\
f]^\top\colon X\to Z\oplus Y$.  This motivates us to introduce the
following concept.

\begin{definition}
  \th\label{def:n-pushout-diagram} Let $\C$ be an additive category,
  $X$ a complex in $\CC^{n-1}(\C)$, and $f^0\colon{X^0}\to Y^0$ a
  morphism in $\C$.  An \emph{$n$-pushout diagram of $X$ along $f^0$}
  is a morphism of complexes
  \[
  \begin{tikzcd}
    X\dar{f}&X^0\rar\dar{f^0}&X^1\rar\dar&\cdots\rar&X^{n-1}\rar\dar&X^n\dar\\
    Y&Y^0\rar&Y_1\rar&\cdots\rar&Y^{n-1}\rar&Y^n
  \end{tikzcd}
  \]
  such that in the \emph{mapping cone} $C=C(f)$
  \[
  \begin{tikzcd}
    X^0\rar{d_C^{-1}}&X^1\oplus
    Y^0\rar{d_C^{0}}&\cdots\rar{d_C^{n-2}}&X^n\oplus
    Y^{n-1}\rar{d_C^{n-1}}&Y^n.
  \end{tikzcd}
  \]
  the sequence $\tuple{d_C^0,\dots,d_C^{n-1}}$ is an $n$-cokernel of
  $d_C^{-1}$, where we define
  \begin{equation}
    \label{differential_mapping_cone}
    \begin{tikzcd}
      d_C^k:=\begin{bmatrix}
        -d_X^{k+1}&0\\
        f^{k+1}&d_Y^{k}
      \end{bmatrix}\colon X^{k+1}\oplus Y^{k}\rar& X^{k+2}\oplus
      Y^{k+1}
    \end{tikzcd}
  \end{equation}
  for each $k\in\set{-1,0,1,\dots,n-1}$.
  In particular,
  \[
  d_C^{-1}=
  \begin{bmatrix}
    -d_X^0\\f^0
  \end{bmatrix}\quad\text{and}\quad d_C^{n-1}=
  \begin{bmatrix}
    f^n&d_Y^{n-1}
  \end{bmatrix}.
  \]
  Note that the fact that $C(f)$ is a complex encodes precisely that
  $X$ and $Y$ are complexes and that $f$ is a morphism of complexes.
  The concept of \emph{$n$-pullback diagram} is defined dually.
\end{definition}

We now state some of general properties of $n$-pushout diagrams.

\begin{proposition}
  \th\label{n-pushouts-and-weak-cokernels} Let $\C$ be an additive
  category. Suppose that we are given an $n$-pushout diagram
  \[
  \begin{tikzcd}
    X\dar{f}&X^0\rar\dar{g^0} &X^1\rar\dar&\cdots\rar&X^{n-1}\rar\dar&X^n\dar\\
    Y&Y^0\rar&Y^1\rar &\cdots\rar&Y^{n-1}\rar&Y^n
  \end{tikzcd}
  \]
  and let $k\in\set{0,1,\dots,n-2}$. If $d_Y^{k+1}$ is a weak cokernel
  of $d_Y^k$, then $d_X^{k+1}$ is a weak cokernel of $d_X^k$.
\end{proposition}
\begin{proof}
  Put $C:=C(f)$ and let $u\colon X^{k+1}\to M$ be a morphism such that
  $d_X^ku=0$. Consider the solid part of the following commutative
  diagram:
  \[
  \begin{tikzcd}
    {}&X^k\rar\dar{f^k}&X^{k+1}\rar\dar[swap]{f^{k+1}}\ar[bend left, near start]{dd}{u}&X^{k+2}\dar{f^{k+2}}\ar[dotted]{ddl}[swap, near start]{h}\\
    Y^{k-1}\rar&Y^k\rar\drar[swap]{0}&Y^{k+1}\rar\dar[dotted,swap]{v}&Y^{k+2}\dlar[dotted]{w}\\
    &&M
  \end{tikzcd}
  \]
  Given that $d_C^{k}\colon X^{k+1}\oplus Y^k\to X^{k+2}\oplus
  Y^{k+1}$ is a weak cokernel of $d_C^{-1}\colon X^k\oplus Y^{k-1}\to
  X^{k+1}\oplus Y^k$, there exist morphisms $v\colon Y^{k+1}\to M$ and
  $h\colon X^{k+2}\to M$ such that $d_Y^kv=0$ and
  $u-f^{k+1}v=d_X^{k+1}h$. Since $d_Y^{k+1}$ is a weak cokernel of
  $d_Y^k$ there exists a morphism $w\colon Y^{k+2}\to M$ such that
  $v=d_Y^{k+1}w$. Therefore we have
  \begin{align*}
    u=&d_X^{k+1}h+f^{k+1}v\\
    =&d_X^{k+1}h+f^{k+1}(d_Y^{k+1}w)\\
    =&d_X^{k+1}(h+f^{k+2}w).
  \end{align*}
  This shows that $d_X^{k+1}$ is a weak cokernel of $d_X^k$.
\end{proof}

Our choice of terminology in \th\ref{def:n-pushout-diagram} is
justified by the following property.

\begin{proposition}
  \th\label{universal-property-of-n-pushout-diagrams} Let $\C$ be an
  additive category, $g\colon X\to Z$ a morphism of complexes in
  $\CC^{n-1}(\C)$ and suppose there exists an $n$-pushout diagram of
  $X$ along $g^0$
  \[
  \begin{tikzcd}
    X\dar{f}&X^0\rar\dar{g^0} &X^1\rar\dar&\cdots\rar&X^{n-1}\rar\dar&X^n\dar\\
    Y&Y^0=Z^0\rar&Y^1\rar &\cdots\rar&Y^{n-1}\rar&Y^n
  \end{tikzcd}
  \]
  Then, there exists a morphism of complexes $p\colon Y\to Z$ such
  that $p^0=1_{Z^0}$ and a homotopy $h\colon fp\to g$ with
  $h^1=0$. Moreover, these properties determine $p$ uniquely up to
  homotopy.
\end{proposition}
\begin{proof}
  Let $h^1\colon X^1\to Z^0$ be the zero morphism, $p^0=1_{Z^0}$ and
  $C:=C(f)$. Inductively, suppose that $0\leq k\leq n$ and that for
  all $\ell\leq k$ we have constructed a morphism $p^\ell\colon
  Y^\ell\to Z^\ell$ such that
  \[
  d_Y^{\ell-1}p^\ell=p^{\ell-1}d_Z^{\ell-1}
  \]
  and a morphism $h^{\ell+1}\colon X^{\ell+1}\to Z^\ell$ such that
  \[
  f^\ell p^\ell-g^\ell=h^\ell d_Z^{\ell-1}+d_X^\ell h^{\ell+1}.
  \]
  We claim that the composition
  \[
  \begin{tikzcd}[column sep=toooooohuge, ampersand replacement=\&]
    X^k\oplus Y^{k-1}\rar{
      \begin{bmatrix}
        -d_X^k&0\\
        f^k&d_Y^{k-1}
      \end{bmatrix}
    }\&X^{k+1}\oplus Y^k\rar{
      \begin{bmatrix}
        g^{k+1}-h^{k+1}d_Z^k&p^k d_Z^k
      \end{bmatrix}
    }\&Z^{k+1}
  \end{tikzcd}
  \]
  vanishes.  Indeed, on one hand we have
  \[
  f^k(p^k d_Z^k)=(g^k+d_X^k h^{k+1})d_Z^k=d_X^k(g^{k+1}-h^{k+1}d_Z^k).
  \]
  On the other hand, we have
  \[
  d_Y^{k-1}(p^kd_Z^k)=p^{k-1}d_Z^{k-1}d_Z^k=0.
  \]
  The claim follows.

  Next, since $d_C^{k}$ is a weak cokernel of $d_C^{k-1}$, there
  exists a morphism $p^{k+1}\colon Y^{k+1}\to Z^{k+1}$ such that
  \[
  d_Y^k p^{k+1}=p^k d_Z^k
  \]
  and a morphism $h^{k+2}\colon X^{k+2}\to Y^{k+1}$ such that
  \[
  g^{k+1}+h^{k+1}d_Z^k=-d_X^{k+1}h^{k+2}+f^{k+1}p^{k+1}.
  \]
  This finishes the induction step, and the construction of the
  required morphism $p\colon Y\to Z$.  Moreover, $h\colon fp\to g$ is
  a homotopy (note that $h^{n+1}=0$).  The last claim follows
  immediately from the \th\ref{comparison-lemma}.
\end{proof}

\begin{defprop}
  \th\label{existence-of-good-n-pushout-diagrams} Let $\C$ be an
  additive category and $g^0\colon X^0\to Z^0$ a morphism in $\C$.
  Suppose that there exists an $n$-pushout diagram of
  $X\in\Ch^{n-1}(\C)$ along $g^0$
  \[
  \begin{tikzcd}
    X\dar{f}&X^0\rar\dar{g^0} &X^1\rar\dar&\cdots\rar&X^{n-1}\rar\dar&X^n\dar\\
    Y&Y^0=Z^0\rar&Y^1\rar &\cdots\rar&Y^{n-1}\rar&Y^n
  \end{tikzcd}
  \]
  Then, the following statements hold:
  \begin{enumerate}
  \item There exists an $n$-pushout diagram
    $\tilde{f}\colon X\to \tilde{Y}$ of $X$ along $g^0$ such that for
    every morphism $g\colon X\to Z$ of complexes lifting $g^0$ and
    such that $Z\in\Ch^{n-1}(\C)$ there exists a morphism of complexes
    $p\colon Y\to Z$ such that $p^0=1_{Z^0}$ and $\tilde{f}p=g$.
  \item For each $2\leq k\leq n$ the morphism $\tilde{f}^k$ is a split
    monomorphism.
  \item We have $\tilde{Y}=Y\oplus X'$ for a contractible complex
    $X'\in\CC^{n-1}(\C)$.
  \end{enumerate}
  We call the morphism $\tilde{f}\colon X\to\tilde{Y}$ a \emph{good
    $n$-pushout diagram of $X$ along $g^0$}.
\end{defprop}
\begin{proof}
  If $n=1$ the result is trivial, so we may assume that $n\geq2$.  For
  $C\in\C$ and $k\in\ZZ$, let $i_k(C)$ be the complex with $d^k=1_C$
  and which is 0 in each degree different from $k$ and $k+1$. We
  define
  \[
  X':= \bigoplus_{k=2}^n i_{k-1}(X^k)
  \]
  and $\tilde{Y}:=Y\oplus X'$.  Note that $\tilde{Y}^0=Y^0$ and that
  $X'$ is a contractible complex. It readily follows that the diagram
  \[
  \begin{tikzcd}[ampersand replacement=\&, row sep=large]
    X\dar{\tilde{f}}\&X^0\rar\dar{f^0=g^0} \&X^1\rar\dar{
      \begin{bmatrix}
	f^1\\d_{X^1}
      \end{bmatrix}}\&X^2\rar\dar{
      \begin{bmatrix}
	f^2\\1\\d_X^2
      \end{bmatrix}}\&\cdots\rar \&X^k\dar{
      \begin{bmatrix}
	f^k\\1\\d_X^k
      \end{bmatrix}}\rar\&\cdots\rar\&X^n\dar{
      \begin{bmatrix}
	f^n\\1
      \end{bmatrix}}\\
    \tilde{Y}\&Y^0\rar\&\tilde{Y}^1\rar\&\tilde{Y}^2\rar\&\cdots\rar\&\tilde{Y}^k\rar\&\cdots\rar\&\tilde{Y}^n
  \end{tikzcd}
  \]
  commutes. Observe that for each $2\leq k\leq n$ the morphism
  $\tilde{f}^k$ is a split monomorphism. Using
  \th\ref{universal-property-of-n-pushout-diagrams}, it is easy to
  show that $\tilde{f}$ has the required factorization property; the
  details are left to the reader.
\end{proof}

\section{$n$-abelian categories}
\label{sec:n-abelian-categories}

In this section we introduce $n$-abelian categories and establish
their basic properties; we give a characterization of semisimple
categories in terms of $n$-abelian categories. We also introduce
projective objects in $n$-abelian categories and study their basic
properties. Finally, we show that $n$-cluster-tilting subcategories of
abelian categories are $n$-abelian; we give a partial converse in the
case of $n$-abelian categories with enough projectives.

\subsection{Definition and basic properties}

The following definition is motivated by the axioms of abelian
categories given in \cite[Def. A.4.2]{weibel_introduction_1994}.

\begin{definition}
  \th\label{def:n-abelian-category} Let $n$ be a positive integer.  An
  \emph{$n$-abelian category} is an additive category $\M$ which
  satisfies the following axioms:
  \begin{axioms}
  \item[A0]{The category $\M$ is idempotent
      complete. \label{ax-ab:split-idempotents}}
  \item[A1]{Every morphism in $\M$ has an $n$-kernel and an
      $n$-cokernel. \label{ax-ab:n-co-kernel-exists}}
  \item[A2]{For every monomorphism $f^0\colon X^0\to X^1$ in $\M$ and
      for every $n$-cokernel $\tuple{f^1,\dots,f^n}$ of $f^0$, the
      following sequence is
      $n$-exact: \label{ax-ab:monos-are-admissible}}
    \[
    \begin{tikzcd}
      X^0\rar{f^0}&X^1\rar{f^1}&\cdots\rar{f^{n-1}}&X^n\rar{f^n}&X^{n+1}.
    \end{tikzcd}
    \]
  \item[A2$^\op$]{For every epimorphism $g^n\colon X^n\to X^{n+1}$ in
      $\M$ and for every $n$-kernel $\tuple{g^0,\dots,g^{n-1}}$ of
      $g^n$, the following sequence is
      $n$-exact: \label{ax-ab:epis-are-admissible}}
    \[
    \begin{tikzcd}
      X^0\rar{g^0}&X^1\rar{g^1}&\cdots\rar{g^{n-1}}&X^n\rar{g^n}&X^{n+1}.
    \end{tikzcd}
    \]
  \end{axioms}
\end{definition}

Let us give some important remarks regarding
\th\ref{def:n-abelian-category}.

\begin{remark}
  \th\label{weaker_A2}
  By \th\ref{n-exact-sequences-are-closed-under-homotopy-equivalence}
  and \th~\ref{weak-isos-induce-homotopy-equivalences} we can replace
  axiom \ref{ax-ab:monos-are-admissible} by the following weaker
  version:
  \begin{enumerate}
  \item[(A2')] For every monomorphism $f^0\colon X^0\to X^1$ in $\M$
    there exists an $n$-exact sequence:
    \[
    \begin{tikzcd}
      X^0\rar{f^0}&X^1\rar{f^1}&\cdots\rar{f^{n-1}}&X^n\rar{f^n}&X^{n+1}.
    \end{tikzcd}
    \]
  \end{enumerate}
  Naturally, we can weaken axiom \ref{ax-ab:epis-are-admissible} in a
  dual manner.
\end{remark}

\begin{remark}
  \th\label{n-abelian-monos-are-admissible} Let $\M$ be an $n$-abelian
  category. An immediate consequence of axioms
  \ref{ax-ab:n-co-kernel-exists} and \ref{ax-ab:monos-are-admissible}
  (resp. \ref{ax-ab:epis-are-admissible}) is that every monomorphism
  (resp. epimorphism) in $\M$ appears as the leftmost
  (resp. rightmost) morphism in some $n$-exact sequence.
\end{remark}

\begin{remark}
  Let $m$ and $n$ be distinct positive integers. Note that the only
  categories which are both $n$-abelian and $m$-abelian are the
  semisimple categories, see
  \th\ref{m-n-abelian-categories-are-contractible}.
\end{remark}

Note that 1-abelian categories are precisely abelian categories in the
usual sense. It is easy to see that abelian categories are idempotent
complete; thus, if $n=1$, then axiom \ref{ax-ab:split-idempotents} in
\th\ref{def:n-abelian-category} is redundant. However, if $n\geq2$,
then axiom \ref{ax-ab:split-idempotents} is independent from the
remaining axioms as shown by the following example.

\begin{example}
  Let $n\geq2$ and $K$ be a field. Consider the full subcategory $\V$
  of $\mod K$ given by the finite dimensional $K$-vector spaces of
  dimension different from 1. Then, $\V$ is not idempotent complete
  but it satisfies axioms \ref{ax-ab:n-co-kernel-exists},
  \ref{ax-ab:monos-are-admissible} and
  \ref{ax-ab:epis-are-admissible}.
\end{example}
\begin{proof}
  Firstly, note that $\V$ is an additive subcategory of $\mod K$. The
  fact that $\V$ is not idempotent complete is obvious (consider the
  idempotent $0\oplus 1_K\colon K^2\to K^2$ whose kernel is
  one-dimensional, for example).  Let us show that $\V$ satisfies
  axiom \ref{ax-ab:n-co-kernel-exists}.  Indeed, let $f\colon V\to W$
  be a morphism in $V$.  If $\coker f$ has dimension different from 1,
  then
  \[
  \begin{tikzcd}
    V\rar&W\rar&\coker f\rar&0\rar&\cdots\rar&0
  \end{tikzcd}
  \]
  gives an $n$-cokernel of $f$ in $\V$.  If $\coker f$ has dimension
  1, then we can construct an $n$-cokernel of $f$ in $\V$ by a
  commutative diagram
  \[
  \begin{tikzcd}
    V\rar{f}&W\dar\rar&K^3\rar&K^2\rar&0\rar&\cdots\rar&0\\
    &\coker f\urar
  \end{tikzcd}
  \]
  where $\coker f\to K^3\to K^2$ is a kernel-cokernel pair.  We can
  construct an $n$-kernel of $f$ in a dual manner.  This shows that
  $\V$ satisfies axiom \ref{ax-ab:n-co-kernel-exists}.  That $\V$
  satisfies axioms \ref{ax-ab:monos-are-admissible} and
  \ref{ax-ab:epis-are-admissible} follows from
  \th\ref{split-mono-implies-sequence-contracts} since contractible
  complexes with $n+2$ terms are in particular $n$-exact sequences by
  \th\ref{n-exact-sequences-are-closed-under-homotopy-equivalence}.
\end{proof}

\begin{lemma}
  \th\label{split-idempotents-weak-cokernel-to-cokernel} Let $\C$ be
  an idempotent complete additive category and suppose that we are
  given a sequence of morphisms in $\C$ of the form
  \[
  \begin{tikzcd}
    A\rar{f}&B\rar{g}&C\rar{h}&D
  \end{tikzcd}
  \]
  If $g$ is a weak cokernel of $f$, and $h$ is both a split
  epimorphism and a cokernel of $g$, then $f$ admits a cokernel in
  $\C$.
\end{lemma}
\begin{proof}
  Since $h$ is a split epimorphism there exists a morphism
  $i\colon{D}\to{C}$ such that $ih=1_D$. It follows that the morphism
  $e:=1_C-hi$ is idempotent. Since the category $\C$ is idempotent
  complete, there exists an object $E\in\C$ and morphisms $r\colon
  C\to E$ and $s\colon E\to C$ such that $sr=1_E$ and $rs=e$. Note
  that this implies that $r$ is an epimorphism and $sh=0$ for we have
  \[
  r(sh)=(1-hi)h=h-h=0.
  \]
  
  We claim that $gr$ is a cokernel of $f$. Indeed, let $u\colon B\to
  B'$ be a morphism such that $fu=0$. Since $g$ is a weak cokernel of
  $f$ there exists a morphism $v\colon C\to B'$ such that $u=gv$. It
  follows that
  \[
  u=gv=g(1-hi)v=(gr)(sv).
  \]
  This shows that $gr$ is a weak cokernel of $f$. It remains to show
  that $gr$ is an epimorphism. For this, let $w\colon E\to E'$ be a
  morphism such that $(gr)w=0$. Since $h$ is a cokernel of $g$ there
  exists a morphism $x\colon D\to E'$ such that $rw=hx$. It follows
  that
  \[
  w=(sr)w=s(hx)=0.
  \]
  This shows that $gr$ is an epimorphism. Therefore $gr$ is a cokernel
  of $f$.
\end{proof}

\begin{proposition}
  \th\label{weak-cokernel-can-be-extended-to-n-cokernel} Let $\M$ be
  an additive category which satisfies axioms
  \ref{ax-ab:split-idempotents} and \ref{ax-ab:n-co-kernel-exists},
  and let $X$ a complex in $\CC^{n-1}(\C)$. If for all $1\leq k\leq
  n-1$ the morphism $d^k$ is a weak cokernel of $d^{k-1}$, then
  $d^{n-1}$ admits a cokernel in $\M$.
\end{proposition}
\begin{proof}
  If $n=1$, then the result follows trivially from axiom
  \ref{ax-ab:n-co-kernel-exists}.  Hence we may assume that $n\geq 2$.
  By axiom \ref{ax-ab:n-co-kernel-exists} there exists an $n$-cokernel
  \[
  \tupleP{d^k\colon X^k\to X^{k+1}}{n\leq k\leq 2n-1}
  \]
  of $d^{n-1}$. Using axiom \ref{ax-ab:n-co-kernel-exists} again
  together with the factorization property of weak cokernels we obtain
  a commutative diagram
  \[
  \begin{tikzcd}[column sep=small]
    X^0\rar{d^0}\dar[equals]&X^1\rar{d^1}\dar[equals]&X^2\rar{d^2}\dar[dotted]{f^2}&\cdots\rar{d^{n-1}}&X^n\rar{d^n}\dar[dotted]{f^n}&X^{n+1}\rar{d^{n+1}}\dar[dotted]{f^{n+1}}&X^{n+2}\rar{d^{n+2}}\dar[dotted]&\cdots\rar{d^{2n-1}}&X^{2n}\dar[dotted]\\
    Y^0\rar\dar[equals]&Y^1\rar\dar[equals]&Y^2\rar\dar[dotted]{g^2}&\cdots\rar&Y^n\rar\dar[dotted]{g^n}&Y^{n+1}\rar\dar[dotted]{g^{n+1}}&0\rar\dar[dotted]&\cdots\rar&0\dar[dotted]\\
    X^0\rar{d^0}&X^1\rar{d^1}&X^2\rar{d^2}&\cdots\rar{d^{n-1}}&X^n\rar{d^n}&X^{n+1}\rar{d^{n+1}}&X^{n+2}\rar{d^{n+2}}&\cdots\rar{d^{2n-1}}&X^{2n}
  \end{tikzcd}
  \]
  in which the middle row gives an $n$-cokernel of $d^0$. The
  \th\ref{comparison-lemma} implies that there exists a morphism
  $h\colon X^{2n}\to X^{2n-1}$ such that $hd^{2n-1}=1$. Hence we may
  apply \th\ref{split-idempotents-weak-cokernel-to-cokernel} and
  reduce the length of the $n$-cokernel of $d^{n-1}$ by one
  morphism. Proceeding inductively, we deduce that $d^{n-1}$ has a
  cokernel in $\M$.
\end{proof}

The importance of axiom \ref{ax-ab:split-idempotents} becomes apparent
in the following result, which asserts that $n$-abelian categories
have $n$-pushout diagrams and $n$-pullback diagrams.

\begin{theorem}[Existence of $n$-pushout diagrams]
  \th\label{n-abelian-cats-have-n-pushout-diagrams} Let $\M$ be an
  additive category which satisfies axioms
  \ref{ax-ab:split-idempotents} and
  \ref{ax-ab:n-co-kernel-exists}. Let $X$ be a complex in
  $\CC^{n-1}(\C)$, and a morphism $f\colon X^0\to Y^0$. Then, the
  following statements hold:
  \begin{enumerate}
  \item\label{it:existences} Then, there exists an $n$-pushout diagram
    \[
    \begin{tikzcd}
      X^0\rar\dar{f}&X^1\rar\dar&\cdots\rar&X^{n-1}\rar\dar&X^n\dar\\
      Y^0\rar&Y^1\rar&\cdots\rar&Y^{n-1}\rar&Y^n
    \end{tikzcd}
    \]
  \item\label{it:is-mono} Suppose, moreover, that $\M$ is an
    $n$-abelian category. If $d_X^0$ is a monomorphism, then $d_Y^0$
    is a monomorphism.
  \end{enumerate}
\end{theorem}
\begin{proof}
  \eqref{it:existences} We shall construct the complex $Y$
  inductively.  Set $f^0:=f$ and
  \[
  \begin{tikzcd}[ampersand replacement=\&]
    d_C^{-1}=
    \begin{bmatrix}
      -d_X^0\\f^0
    \end{bmatrix}\colon X^0\rar\&X^1\oplus Y^0.
  \end{tikzcd}
  \]
  Let $0\leq k\leq n-2$ and suppose that for each $\ell\leq k$ we have
  constructed an object $Y^\ell$ and morphisms $f^\ell\colon X^\ell\to
  Y^\ell$ and $d_Y^{\ell-1}\colon Y^{\ell-1}\to Y^\ell$ such that
  $d_C^{\ell-2}d_C^{\ell-1}=0$ where
  \[
  \begin{tikzcd}[ampersand replacement=\&]
    d_C^{\ell-1}:=\begin{bmatrix}
      -d_X^{\ell}&0\\
      f^{\ell}&d_Y^{\ell-1}
    \end{bmatrix}\colon X^{\ell}\oplus Y^{\ell-1}\rar\&
    X^{\ell+1}\oplus Y^{\ell}
  \end{tikzcd}
  \]
  (compare with \eqref{differential_mapping_cone}).  Then, by axiom
  \ref{ax-ab:n-co-kernel-exists}, the morphism $d_C^{k-1}$ has a weak
  cokernel $g^k:=[f^{k+1}\ d_Y^k]\colon X^{k+1}\oplus Y^k\to Y^{k+1}$.
  We claim that
  \[
  \begin{tikzcd}[ampersand replacement=\&]
    d_C^{k}:=\begin{bmatrix}
      -d_X^{k+1}&0\\
      f^{k+1}&d_Y^k
    \end{bmatrix}\colon X^{k+1}\oplus Y^k\rar\& X^{k+2}\oplus
    Y^{k+1}
  \end{tikzcd}
  \]
  is also a weak cokernel of $d_C^{k-1}$.  Indeed, it is readily
  verified that $d_C^{k-1}d_C^{k}=0$.  Let $u\colon X^{k+1}\oplus
  Y^k\to M$ be a morphism such that $d_C^{k-1}u=0$.  Since $g^k$ is a
  weak cokernel of $d_C^{k-1}$, there exists a morphism $v\colon
  Y^{k+1}\to M$ such that $u=g^kv$.  It follows that the following
  diagram is commutative:
  \[
  \begin{tikzcd}[ampersand replacement=\&]
    X^{k+1}\oplus Y^k\rar{d_C^{k}}\drar{u}\&X^{k+2}\oplus Y^{k+1}
    \dar{
      \begin{bmatrix}
        0&v
      \end{bmatrix}
    }\\
    \&M
  \end{tikzcd}
  \]
  This shows that $d_C^{k}$ is a weak cokernel of $d_C^{k-1}$.
  Finally, \th\ref{weak-cokernel-can-be-extended-to-n-cokernel}
  implies that the morphism $d_C^{n-2}$ admits a cokernel
  \[
  d_C^{n-1}\colon X^n\oplus Y^{n-1}\to Y^n.
  \]
  This shows that the tuple $\tuple{d_C^0,d_C^1,\dots,d_C^{n-1}}$ is a
  weak cokernel of $d_C^{-1}$.  The existence of the required
  commutative diagram follows from the fact that $C$ is a complex.

  \eqref{it:is-mono} Finally, suppose that $\M$ is $n$-abelian and
  $d_X^0$ is a monomorphism. Note that this implies that $d_C^{-1}$ is
  also a monomorphism. Then, axiom \ref{ax-ab:monos-are-admissible}
  implies that the $C$ is an $n$-exact sequence.  In order to show
  that $d_Y^0$ is a monomorphism, let $u\colon M\to Y^0$ be a morphism
  such that $ud_Y^0=0$. It follows that the composition
  \[
  \begin{tikzcd}[ampersand replacement=\&, column sep=large]
    M\rar{
      \begin{bmatrix}
        0\\u
      \end{bmatrix}}\&X^1\oplus Y^0\rar{
      \begin{bmatrix}
        -d_X^1&0\\
        f^1&d_Y^0\\
      \end{bmatrix}}\&X^2\oplus Y^1
  \end{tikzcd}
  \]
  vanishes. Given that $d_C^{-1}$ is a kernel of $d_C^0$, there exists
  a morphism $v\colon M\to X^0$ such that $vd_X^0=0$ and
  $vf^0=u$. Since $d_X^0$ is a monomorphism, we have $u=0$.  This
  shows that $d_Y^0$ is a monomorphism.
\end{proof}

We remind the reader that an additive category $\C$ is
\emph{semisimple} if every morphism $f\colon A\to B$ in $\C$ factors
as $f=pi$, where $p$ is a split epimorphism and $i$ is a split
monomorphism.  The following result characterizes semisimple
categories in terms of $n$-abelian categories.

\begin{theorem}
  \th\label{characterization-of-contractible-n-abelian-categories} Let
  $\C$ be an additive category and $n$ a positive integer. Then, the
  $n$-abelian categories in which every $n$-exact sequence is
  contractible are precisely the semisimple categories.
\end{theorem}
\begin{proof}
  Suppose that $\C$ is a semisimple category. We only show that $\C$
  is idempotent complete. It is straightforward to verify that $\C$
  satisfies the remaining axioms of $n$-abelian categories, the fact
  that every $n$-exact sequence in $\C$ is contractible follows
  immediately from \th\ref{split-mono-implies-sequence-contracts}.
  Let us show then that $\C$ is idempotent complete. Let $e\colon A\to
  A$ be an idempotent in $\C$. Since $\C$ is semisimple, $e$ factors
  as $e=pi$ where $p\colon A\to B$ is a split epimorphism and $i\colon
  B\to A$ is a split monomorphism. We claim that $ip=1_B$. Indeed, let
  $h\colon A\to B$ be a morphism such that $ih=1_B$. Given that
  $e^2=e$ we have
  \[
  p=p(ih)=eh=e^2h=(pipi)h=p(ip).
  \]
  Since $p$ is an epimorphism we have $ip=1_B$ as claimed. This shows
  that $\C$ is idempotent complete.

  Conversely, suppose that $\C$ is an $n$-abelian category in which
  every $n$-exact sequence is contractible and let $f\colon A\to B$ be
  a morphism in $\C$. We claim that $f$ admits both a kernel and a
  cokernel in $\C$. Indeed, by axiom \ref{ax-ab:n-co-kernel-exists}
  there exists an $n$-cokernel $\tuple{f^1,\cdots,f^n}$ of $f$. By
  hypothesis, the epimorphism $f^n$ must split. Then,
  \th\ref{split-idempotents-weak-cokernel-to-cokernel} implies that
  $f^{n-2}$ has a cokernel in $\C$. By applying this argument
  inductively we deduce that $f$ admits a cokernel in $\C$. By
  duality, $f$ also admits a kernel in $\C$. The remaining part of the
  proof is classical, compare for example with the proof of
  \cite[Prop. 4.8]{buhler_exact_2010}.
  
  We need to show that $f$ factors as $f=pi$ where $p$ is a split
  epimorphism and $i$ is a split monomorphism. Given that $f$ has both
  a kernel and a cokernel in $\C$, is easy to construct a commutative
  diagram
  \[
  \begin{tikzcd}
    K\rar{i}&A\rar{f}\dar{p'}&B\rar{p}&C\\
    &J\rar[dotted]{g}&I\uar{i'}
  \end{tikzcd}
  \]
  where $i$ is a kernel of $f$, and $p$ is a cokernel of $f$, and the
  sequences $K\to A\to J$ and $I\to B\to C$ are kernel-cokernel pairs.
  We claim that $g$ is an isomorphism, for which it is enough to show
  that is both a monomorphism and an epimorphism as all such morphisms
  split by hypothesis. By duality we only need to show that $g$ is an
  epimorphism.

  Let $h\colon I\to I'$ be a morphism such that $gh=0$. Firstly, by
  \th\ref{n-abelian-cats-have-n-pushout-diagrams} there exists a
  commutative diagram
  \[
  \begin{tikzcd}
    I\rar{i'}\dar{h}&B\rar{p}\dar{h'}&C\\
    I'\rar{i''}&B'
  \end{tikzcd}
  \]
  where $i''$ is a monomorphism.  Secondly, we claim that
  $fh'=0$. Indeed, we have
  \[
  fh'=(p'gi')h'=p'(gh)i''=0.
  \]
  Therefore, since $p$ is a cokernel of $f$, there exists a morphism
  $j\colon C\to B'$ such that $pj=h'$. It follows that
  \[
  hi''=i'h'=i'(pj)=0.
  \]
  Finally, since $i''$ is a monomorphism, we have $h=0$. This shows
  that $g$ is an epimorphism.
\end{proof}

\begin{corollary}
  \th\label{m-n-abelian-categories-are-contractible} Let $m$ and $n$
  be two distinct positive integers and $\C$ an additive category. If
  $\C$ is both $m$-abelian and $n$-abelian, then $\C$ is a semisimple
  category.
\end{corollary}
\begin{proof}
  Without loss of generality we may assume that $m<n$. By
  \th\ref{characterization-of-contractible-n-abelian-categories,split-mono-implies-sequence-contracts}
  it is enough to show that every monomorphism in $\C$ splits.  Let
  $f^0\colon X^0\to X^1$ be a monomorphism in $\C$ and let
  $(f^1,\dots,f^n)$ be an $n$-cokernel of $f^0$, and $(g^1,\dots,g^m)$
  be an $m$-cokernel of $f^0$. It follows that there exists a
  commutative diagram
  \[
  \begin{tikzcd}[column sep=almostnormal]
    X^0\rar{f^0}\dar[equals]&X^1\rar{f^1}\dar[equals]&X^2\rar{f^2}\dar[dotted]&\cdots\rar{f^{m}}&X^{m+1}\rar{f^{m+1}}\dar[dotted]&X^{m+2}\rar{f^{m+2}}\dar[dotted]&\cdots\rar{f^n}&X^{n+1}\dar[dotted]\\
    Y^0\rar\dar[equals]&Y^1\rar{g^1}\dar[equals]&Y^2\rar{g^2}\dar[dotted]&\cdots\rar{g^m}&Y^{m+1}\rar\dar[dotted]&0\rar\dar[dotted]&0\rar&0\dar[dotted]\\
    X^0\rar{f^0}&X^1\rar{f^1}&X^2\rar{f^2}&\cdots\rar{f^{m}}&X^{m+1}\rar{f^{m+1}}&X^{m+2}\rar{f^{m+2}}&\cdots\rar{f^n}&X^{n+1}
  \end{tikzcd}
  \]
  By the \th\ref{comparison-lemma} there exists a morphism $h\colon
  X^{n+1}\to X^n$ such that $hf^n=1$. Thus $f^n$ is a split
  epimorphism. Then, as $(f^0,f^1,\dots,f^n)$ is an $n$-exact sequence
  by axiom \ref{ax-ab:monos-are-admissible}, the dual of
  \th\ref{split-mono-implies-sequence-contracts} implies that $f^0$ is
  a split monomorphism.
\end{proof}

\subsection{Projective objects in $n$-abelian categories}

We remind the reader of the following classical definition.

\begin{definition}
  Let $\C$ be an additive category. We say that $P\in\C$ is
  \emph{projective} if for every epimorphism $f\colon A\to B$ the
  sequence of abelian groups
  \[
  \begin{tikzcd}
    \C(P,A)\rar{?\cdot f}&\C(P,B)\rar&0
  \end{tikzcd}
  \]
  is exact. The concept of \emph{injective object in $\C$} is defined
  dually.
\end{definition}

Our aim is to prove the following important property of projective
objects in an $n$-abelian category.

\begin{theorem}
  \th\label{projectives-are-strongly-projective} Let $\M$ be an
  $n$-abelian category and $P\in\M$ a projective object. Then, for
  every morphism $f\colon L\to M$ and every weak cokernel $g\colon
  M\to N$ of $f$, the sequence of abelian groups
  \begin{equation}
    \label{eq:strong-projectivity}
    \begin{tikzcd}
      \M(P,L)\rar{?\cdot f}&\M(P,M)\rar{?\cdot g}&\M(P,N)
    \end{tikzcd}
  \end{equation} 
  is exact.
\end{theorem}

For the proof of \th\ref{projectives-are-strongly-projective} we need
the following result which, albeit technical, is interesting in its
own right.

\begin{proposition}
  \th\label{inductive-construction-of-n-kernel} Let $\M$ be an
  $n$-abelian category, $f^0\colon X^0\to X^1$ a morphism in $\M$ and
  $\tupleP{f^k\colon X^k\to X^{k+1}}{1\leq k\leq n}$ an $n$-cokernel
  of $f^0$. Then, for every $k\in\set{0,1,\dots,n}$ and every
  $\ell\in\set{1,\dots,n}$ there exists morphisms
  $g_k^\ell\colon Y_k^\ell\to Y_k^{\ell-1}$ (with $Y_k^0:=X^k$) and
  $p_k^{\ell-1}\colon Y_k^{\ell-1}\to Y_{k+1}^{\ell}$ satisfying the
  following properties:
  \begin{enumerate}
  \item For every $k\in\set{0,1,\dots,n}$ the diagram
    \begin{equation}
      \label{dia:strong-projectivity}
      \begin{tikzcd}
        Y_k^n\rar{g_k^n}\dar&Y_k^{n-1}\rar{g_k^{n-1}}\dar{p_k^{n-1}}&\cdots\rar{g_k^2}&Y_k^1\rar{g_k^1}\dar{p_k^1}&X^k\rar{f^k}\dar{p_k^0}&X^{k+1}\\
        0\rar&Y_{k+1}^n\rar{g_{k+1}^n}&\cdots\rar{g_{k+1}^3}&Y_{k+1}^2\rar{g_{k+1}^2}&Y_{k+1}^1\urar[swap]{g_{k+1}^1}
      \end{tikzcd}
    \end{equation}
    commutes.
  \item The sequence $\tuple{g_k^n,\dots,g_k^1}$ is an $n$-kernel of
    $f^k$.
  \item The diagram
    \[
    \begin{tikzcd}
      Y_k^n\rar{g_k^n}\dar&Y_k^{n-1}\rar{g_k^{n-1}}\dar{p_k^{n-1}}&\cdots\rar{g_k^2}&Y_k^1\rar{g_k^1}\dar{p_k^1}&X^k\dar{p_k^0}\\
      0\rar&Y_{k+1}^n\rar{g_{k+1}^n}&\cdots\rar{g_{k+1}^3}&Y_{k+1}^2\rar{g_{k+1}^2}&Y_{k+1}^1
    \end{tikzcd}
    \]
    is both an $n$-pullback diagram and a good $n$-pushout diagram,
    see \th\ref{existence-of-good-n-pushout-diagrams}. In particular,
    the morphism
    \[
    \begin{tikzcd}[ampersand replacement=\&]
      \begin{bmatrix}
	p_k^0&g_{k+1}^2
      \end{bmatrix}\colon X^k\oplus Y_{k+1}^2\rar\&Y_{k+1}^1
    \end{tikzcd}
    \]
    is an epimorphism.
  \item If $k\neq0$, then the sequence
    $\tuple{g_k^{k-1},\dots,g_k^1,f^k,\dots,f^n}$ is an $n$-cokernel
    of the morphism $g_k^k$.
  \end{enumerate}
\end{proposition}
\begin{proof}
  We proceed by induction on $k$, beginning with the case $1\neq k=n$;
  the case $n=1$ of $n$-abelian categories is straightforward. By
  axiom \ref{ax-ab:n-co-kernel-exists} there exists an $n$-kernel
  $\tuple{g_n^n,\dots,g_1^n}$ of $f^n$ and by axiom
  \ref{ax-ab:epis-are-admissible} the sequence
  $\tuple{g_n^n,\dots,g_1^1,f^n}$ is an $n$-exact sequence. Note that
  this implies that the diagram
  \[
  \begin{tikzcd}
    Y_n^n\rar{g_n^n}\dar&Y_n^{n-1}\rar{g_n^{n-1}}\dar&\cdots\rar{g_n^2}&Y_n^1\rar{g_n^1}\dar&X^n\rar{f^n}\dar{f^n}&X^{n+1}\rar&0\\
    0\rar&0\rar&\cdots\rar&0\rar&X^{n+1}\urar[swap]{1}
  \end{tikzcd}
  \]
  is a both an $n$-pullback diagram and an $n$-pushout diagram. By
  \th\ref{existence-of-good-n-pushout-diagrams} we can replace this
  diagram by a good $n$-pushout diagram
  \[
  \begin{tikzcd}[row sep=large,ampersand replacement=\&]
    Y_n^n\rar{g_n^n}\dar\&Y_n^{n-1}\rar{g_n^{n-1}}\dar\&\cdots\rar{g_n^2}\&Y_n^1\rar{g_n^1}\dar\&X^n\rar{f^n}\dar{
      \begin{bmatrix}
	f^n\\1
      \end{bmatrix}}\&X^{n+1}\rar\&0\\
    0\rar\&Y_n^{n-2}\rar\&\cdots\rar\&Y_n^1\oplus X^{n+1}\rar\&X^{n+1}\oplus
    X^n\urar[swap]{
      \begin{bmatrix}
	1&0
      \end{bmatrix}}
  \end{tikzcd}
  \]
  Note that this passage does not change the fact that the top row
  gives an $n$-kernel of $f^n$.  This shows that the result is holds
  in this case.

  Let $2\leq k\leq n$ and suppose that we have constructed a
  commutative diagram of the form \eqref{dia:strong-projectivity} with
  the required properties. Since $f^{k-1}f^k=0$ and $g_k^1$ is a weak
  kernel of $f^k$ there exists a morphism $p_{k-1}^0\colon X^{k-1}\to
  Y_k^1$ such that $f^{k-1}=p_{k-1}^0g_k^1$.  We claim that the
  morphism
  \[
  \begin{tikzcd}[ampersand replacement=\&]
    \begin{bmatrix}
      p_{k-1}^0&g_{k}^2
    \end{bmatrix}\colon X^{k-1}\oplus Y_{k}^2\rar\&Y_{k}^1
  \end{tikzcd}
  \]
  is an epimorphism. Indeed, let $u\colon Y_k^1\to M$ be a morphism
  such that $p_{k-1}^0u=0$ and $g_{k}^2u=0$. Given that $g_k^1$ is a
  weak cokernel of $g_k^2$ there exists a morphism $v\colon X^k\to M$
  such that $u=g_k^1v$. It follows that
  \[
  f^{k-1}v=(p_{k-1}^0g_k^1)v=p_{k-1}^0u=0.
  \]
  Then, since $f^k$ is a weak cokernel of $f^{k-1}$, there exists a
  morphism $w\colon X^{k+1}\to M$ such that $v=f^kw$. Thus, we have
  \[
  u=g_k^1v=g_k^1(f^kw)=0.
  \]
  The claim follows. By
  \th\ref{n-abelian-cats-have-n-pushout-diagrams} there exists an
  $n$-pullback diagram
  \[
  \begin{tikzcd}
    Y_{k-1}^n\rar{g_{k-1}^n}\dar&Y_{k-1}^{n-1}\rar{g_{k-1}}\dar{p_{k-1}^{n-1}}&\cdots\rar{g_{k-1}^2}&Y_{k-1}^1\rar{g_{k-1}^1}\dar{p_{k-1}^1}&X^{k-1}\dar{p_{k-1}^0}\\
    0\rar&Y_{k}^n\rar{g_{k}^n}&\cdots\rar{g_{k}^3}&Y_{k}^2\rar{g_{k}^2}&Y_{k}^1
  \end{tikzcd}
  \]
  Note that axiom \ref{ax-ab:epis-are-admissible} implies that it this
  diagram is also an $n$-pushout diagram, and that
  \th\ref{n-pushouts-and-weak-cokernels} implies that it has the
  required properties, except that it need not be a good $n$-pushout
  diagram. Using \th\ref{existence-of-good-n-pushout-diagrams} we may
  replace this diagram by a good $n$-pushout diagram. Note that since
  the passage to a good $n$-pushout diagram amounts to adding a
  contractible $n$-exact sequence it does not alter the properties of
  the previously constructed diagrams.

  Finally, let $k=1$. We have a commutative diagram
  \[
  \begin{tikzcd}
    Y_0^n\rar{g_0^n}\dar&Y_0^{n-1}\rar{g_0^{n-1}}\dar{p_0^{n-1}}&\cdots\rar{g_0^2}&Y_0^1\rar{g_0^1}\dar{p_0^1}&X^0\rar{f^0}\dar{p_0^0}&X^{1}\\
    0\rar&Y_{1}^n\rar{g_{1}^n}&\cdots\rar{g_{1}^3}&Y_{1}^2\rar{g_{1}^2}&Y_{1}^1\urar[swap]{g_{1}^1}
  \end{tikzcd}
  \]
  where the leftmost $n$ squares form an $n$-pullback diagram; we
  claim that they form moreover an $n$-pushout diagram. To show this,
  by axiom \ref{ax-ab:epis-are-admissible} it is enough to show that
  the morphism
  \[
  \begin{tikzcd}[ampersand replacement=\&]
    \begin{bmatrix}
      p_{0}^0&g_{1}^2
    \end{bmatrix}\colon X^{0}\oplus Y_{1}^2\rar\&Y_{1}^1
  \end{tikzcd}
  \]
  is an epimorphism. Let $u\colon Y_1^1\to M$ be a morphism such that
  $p_0^0u=0$ and $g_1^2=0$. Put $Y_0^0:=X^0$ and $u^0:=u$. Proceeding
  inductively, since the diagrams we constructed are good $n$-pushout
  diagrams, for each $k\in\set{1,\dots,n-1}$ we obtain a commutative
  diagram
  \[
  \begin{tikzcd}
    {}&Y_k^{k+1}\rar\dar&Y_k^k\ar[bend left]{ddr}{u^k}\dar{p_k^k}\\
    Y_{k+1}^{k+2}\rar&Y_{k+1}^{k+2}\rar\ar[swap]{drr}{0}&Y_{k+1}^{k+1}\drar[dotted]{u^{k+1}}\\
    &&&M
  \end{tikzcd}
  \]
  Moreover, by \th\ref{n-abelian-cats-have-n-pushout-diagrams} there
  exists a commutative diagram
  \[
  \begin{tikzcd}
    Y_n^n\rar\dar{u^n}&Y_n^{n-1}\dar{p_n^{n-1}}\\
    M\rar{v}&N
  \end{tikzcd}
  \]
  where $v$ is a monomorphism.  It readily follows that
  \[
  u=p_1^1\cdots p_{n-1}^{n-1} u^n\quad\text{and}\quad
  uv=g_1^1p_1^0\cdots p_n^{n-1}.
  \]
  Next, observe that
  \[
  f^0(p_1^0\cdots p_n^{n-1})=p_0^0g_1^1p_1^0\cdots
  p_n^{n-1}=p_0^0uv=0.
  \]
  Given that $f^1$ is a weak cokernel of $f^0$, there exists a
  morphism $w\colon X^1\to N$ such that $f^1w=p_1^0\cdots p_n^{n-1}$.
  Finally, we have
  \[
  uv=g_1^1(p_1^0\cdots p_n^{n-1})=g_1^1f^1w=0.
  \]
  Since $v$ is a monomorphism, we have $u=0$ which is what we needed
  to show.
\end{proof}

We are ready to give the proof of
\th\ref{projectives-are-strongly-projective}.

\begin{proof}[Proof of \th\ref{projectives-are-strongly-projective}]
  By \th\ref{inductive-construction-of-n-kernel} there exists a
  commutative diagram
  \[
  \begin{tikzcd}
    {}&L\rar{f}\dar{p}&M\rar{g}&N\\
    K_2\rar{q}&K_1\urar{i}
  \end{tikzcd}
  \]
  where $i$ is a weak kernel of $g$, we have $qi=0$, and
  \[
  \begin{tikzcd}[ampersand replacement=\&]
    \begin{bmatrix}
      p&q
    \end{bmatrix}\colon L\oplus K_2\rar\&K_1
  \end{tikzcd}
  \]
  is an epimorphism.

  Let $h\colon P\to M$ be a morphism such that $hg=0$.  Since $i$ is a
  weak kernel of $g$, there exists a morphism $j\colon P\to K_1$ such
  that $h=ji$. Given that $P$ is projective and $L\oplus K_2\to K_1$
  is an epimorphism, there exist morphisms $j'\colon P\to L$ and
  $j''\colon P\to K_2$ such that $j=j'p+j''q$. It follows that
  \[
  h=ji=(j'p+j''q)i=j'pi=j'f.
  \]
  Hence $h$ factors through $f$. This shows that the sequence
  \eqref{eq:strong-projectivity} is exact.
\end{proof}

\subsection{$n$-abelian categories and cluster-tilting}

Recall that a subcategory $\C$ of an abelian category $\A$ is
\emph{cogenerating} if for every object $X\in\A$ there exists an
object $Y\in\C$ and a monomorphism $X\to Y$.  The concept of
\emph{generating} subcategory is defined dually.  We use the following
variant of the definition of a cluster-tilting subcategory of an
abelian category.

\begin{definition}
  \th\label{def:n-cluster-tilting-abelian} Let $\A$ be an abelian
  category and $\M$ a generating-cogenerating subcategory of $\A$.  We
  say that $\M$ is an \emph{$n$-cluster-tilting subcategory of $\A$}
  if $\M$ is functorially finite (see Subsection
  \ref{sec:conventions}) in $\A$ and
  \begin{align*}
    \M=&\setP{X\in\A}{\forall i\in\set{1,\dots,n-1}\
      \Ext_\A^i(X,\M)=0}\\
    =&\setP{X\in\A}{\forall i\in\set{1,\dots,n-1}\ \Ext_\A^i(\M,X)=0}.
  \end{align*}
  Note that $\A$ itself is the unique {1-cluster-tilting} subcategory
  of $\A$.
\end{definition}

\begin{remark}
  \th\label{n-ct-abelian-approximations-mono} Let $\A$ be an abelian
  category and $\M$ an $n$-cluster-tilting subcategory of $\A$. Since
  $\M$ is a cogenerating subcategory of $\A$, for all $A\in\A$ each
  left $\M$-approximation of $A$ is a monomorphism.
\end{remark}

The following result provides us with a way of recognizing $n$-abelian
categories.

\begin{theorem}
  \th\label{recognition-thm-abelian} Let $\A$ be an abelian category
  and $\M$ an $n$-cluster-tilting subcategory of $\A$. Then, $\M$ is
  an $n$-abelian category.
\end{theorem}

To prove \th\ref{recognition-thm-abelian} we need the following
technical results.

\begin{proposition}
  \th\label{n-ct_resolution} Let $\A$ be an abelian category and $\M$
  an $n$-cluster-tilting subcategory of $\A$. Then, for all $A\in\A$
  there exists an exact sequence
  \[
  \begin{tikzcd}[column sep=small, row sep=small]
    0\ar{rr}&&A\ar{rr}{f^0}\drar[equals]&&M^1\ar{rr}{f^1}&&\cdots\ar{rr}{f^{n-2}}\drar[swap]{g^{n-2}}&&M^{n-1}\ar{rr}{f^{n-1}}\drar[swap]{g^{n-1}}&&M^n\ar{rr}&&0\\
    &&&C^1\urar[swap]{h^1}&&&&C^{n-1}\urar[swap]{h^{n-1}}&&C^n\urar[equals]
  \end{tikzcd}
  \]
  satisfying the following properties:
  \begin{enumerate}
  \item For each $k\in\set{1,\dots,n}$ we have $M^k\in\M$.
  \item For each $k\in\set{1,\dots,n-1}$ the morphism $h^k\colon
    C^k\to M^k$ is a left $\M$-approximation.
  \item For each $k\in\set{1,\dots,n-1}$ the morphism $g^k\colon
    M^k\to C^{k+1}$ is a cokernel of $h^k\colon C^k\to M^k$.
  \item For all $M\in\M$ the induced sequence of abelian groups
    \[
    \begin{tikzcd}
      0\rar&\A(M^n,M)\rar&\cdots\rar&\A(M^1,M)\rar&\A(A,M)\rar&0
    \end{tikzcd}
    \]
    is exact.
  \end{enumerate}
\end{proposition}
\begin{proof}
  This proof is an adaptation of the proof of
  \cite[Thm. 2.2.3]{iyama_higher-dimensional_2007}.  Given that for
  all $k\in\set{1,\dots,n-1}$ the morphism $h^k\colon C^k\to M^k$ is a
  left $\M$-approximation, it readily follows that the sequence
  \[
  \begin{tikzcd}
    0\rar&\A(M^n,M)\rar&\cdots\rar&\A(M^1,M)\rar&\A(A,M)\rar&0
  \end{tikzcd}
  \]
  is exact. It remains to show that $C^n\in\M$.
  
  We claim that for each $M\in\M$ and each $k\in\set{2,\dots,n}$ we
  have $\Ext_\A^i(C^k,M)=0$ for all $1\leq i\leq k-1$.  We proceed by
  induction on $k$. First, note that for all $M\in\M$ applying the
  contravariant functor $\A(-,M)$ to the exact sequence $0\to
  A\xto{f^0} M^2\to C^2$ we have an exact sequence
  \[
  \begin{tikzcd}
    \A(M^1,M)\rar{f^0\cdot?}&\A(A,M)\rar&\Ext_\A^1(C^2,M)\rar&\Ext_\A^1(M^1,M)=0.
  \end{tikzcd}
  \]
  Moreover, the morphism $\A(M^1,M)\to \A(A,M)$ is an epimorphism for
  $f^0$ is a left $\M$-approximation of $\A$. Thus we have
  $\Ext_\A^1(C^2,M)=0$ as required.

  Let $2\leq k\leq n-1$ and suppose that the claim holds for all
  $\ell\leq k$.  Note that since $\M$ is a cogenerating subcategory of
  $\A$, the morphism $h^k$ is a monomorphism. In particular, we have
  that $h^k$ is a kernel of $g^k$.  For all $M\in\M$ and for each
  $2\leq i\leq k$, applying the contravariant functor $\A(-,M)$ to the
  exact sequence $0\to C^k\to M^k\to C^{k+1}\to 0$ yields an exact
  sequence of the form
  \[
  \begin{tikzcd}
    0=\Ext_\A^{i-1}(C^k,M)\rar&\Ext_\A^i(C^{k+1},M)\rar&\Ext_\A^i(M^k,M)=0.
  \end{tikzcd}
  \]
  Therefore $\Ext_\A^i(C^{k+1},M)=0$ for all $2\leq i\leq k$.
  Moreover, since $h^k$ is a left $\M$-approximation of $Y^k$, the
  induced morphism $\A(M^k,M)\to\A(C^k,M)$ is an epimorphism.  Hence,
  applying the contravariant functor $\A(-,M)$ to the exact sequence
  $0\to C^k\to M^k\to C^{k+1}\to 0$ yields an exact sequence
  \[
  \begin{tikzcd}
    \A(M^k,M)\rar&\A(C^k,M)\rar&\Ext_\A^1(C^{k+1},M)\rar&\Ext_\A^1(M^k,M)=0
  \end{tikzcd}
  \]
  where the left morphism is an epimorphism.  Thus
  $\Ext_\A^1(C^{k+1},M)=0$.  This finishes the induction step.  We
  have shown that for all $M\in\M$ we have $\Ext_\A^i(C^n,M)=0$ for
  all $i\in\set{1,\dots,n-1}$.  Since $\M$ is an $n$-cluster-tilting
  subcategory of $\M$, this implies $C^n=M^n\in\M$ as required.
\end{proof}

\begin{proposition}
  \th\label{lemma-exts} Let $\A$ be an abelian category, $B\in\A$, and
  $\M$ a subcategory of $\A$ such that $\Ext_\A^k(\M,B)=0$ for all
  $k\in\set{0,1,\dots,n-1}$. Consider an exact sequence
  \[
  \begin{tikzcd}
    M_n\rar&M_{n-1}\rar&\cdots\rar&M_1\rar&M_0\rar&A\rar&0
  \end{tikzcd}
  \]
  in $\A$ such that $M_k\in\M$ for all $k\in\set{0,1,\dots,n-1}$. Then, for each
  $k\in\set{1,\dots,n-1}$ there is an isomorphism between
  $\Ext_\A^k(A,B)$ and the cohomology of the induced complex
  \begin{equation}
    \label{eq:lemma-exts}
    \begin{tikzcd}[column sep=small]
      \A(M_0,B)\rar&\A(M_1,B)\rar&\cdots\rar&\A(M_{n-1},B)\rar&\A(M_n,B)
    \end{tikzcd}
  \end{equation}
  at $\A(M_k,B)$.
\end{proposition}
\begin{proof}
  Let $A_k:=\coker(M_{k+1}\to M_k)$. Note that $A_0=A$. Firstly, let
  us show that for each $k\in\set{1,\dots,n-1}$ (the case $k=0$ being
  clear) there exist isomorphisms
  \[
  \Ext_\A^k(A_0,B)\cong\Ext_\A^{k-1}(A_1,B)\cong\cdots\cong\Ext_\A^1(A_{k-1},B).
  \]
  The case $k=1$ is obvious. If $2\leq k\leq n-1$, then for
  each $2\leq\ell\leq k$ applying the functor $\A(-,B)$ to the exact
  sequence $0\to A_{k-\ell+1}\to M_{k-\ell}\to A_{k-\ell}\to0$ yields
  an exact sequence
  \[
  \begin{tikzcd}[column sep=tiny]
    0={}_\A^{\ell-1}(M_{k-\ell},B)\rar&{}_\A^{\ell-1}(A_{k-\ell+1},B)\rar&{}_\A^{\ell}(A_{k-\ell},B)\rar&{}_\A^{\ell}(M_{k-\ell},B)=0
  \end{tikzcd}
  \]
  where we omitted $\Ext_\A$ because of lack of space.  The claim
  follows.

  Secondly, let us show that $\Ext_\A^1(A_{k-1},B)$ is isomorphic to the
  cohomology of the complex \eqref{eq:lemma-exts} at $\A(M_k,B)$. This
  follows by definition from the commutative diagram
  \[
  \begin{tikzcd}[column sep=tiny, row sep=tiny]
    \A(M_{k-1},B)\ar{rr}\drar&&\A(M_k,B)\ar{r}&\A(M_{k+1},B)\\
    &\A(A_k,B)\drar\urar\\
    0\urar&&\Ext_\A^1(A_{k-1},B)\drar\\
    &&&\Ext_\A^1(M_{k-1},B)=0
  \end{tikzcd}
  \]
  which is the glueing of two exact sequences.
  This concludes the proof.
\end{proof}

Now we give the proof of \th\ref{recognition-thm-abelian}.

\begin{proof}[Proof of \th\ref{recognition-thm-abelian}]
  We shall show that $\M$ satisfies the axioms of
  \th\ref{def:n-abelian-category}.
  
  \ref{ax-ab:split-idempotents} Since the abelian category $\A$ is
  idempotent complete, it follows immediately from the definition of
  $n$-cluster-tilting subcategory that $\M$ also is idempotent
  complete.
    
  \ref{ax-ab:n-co-kernel-exists} Let $f\colon A\to B$ be a morphism in
  $\M$.  Let $B\to C$ be a cokernel of $f$, applying
  \th\ref{n-ct_resolution} to $C$ gives the desired $n$-cokernel of
  $f$. By duality, $f$ has an $n$-cokernel.

  \ref{ax-ab:monos-are-admissible} Let $f^0\colon X^0\to X^1$ be a
  monomorphism in $\A$ such that $X^0,X^1\in\M$ and let
  $\tupleP{f^k\colon X^k\to X^{k+1}}{1\leq k\leq n}$ be an
  $n$-cokernel of $f^0$ in $\M$ obtained as in the previous paragraph
  (we remind the reader of \th\ref{weaker_A2}). Applying the dual of
  \th\ref{lemma-exts} to the exact sequence
  \[
  \begin{tikzcd}
    0\rar&X^0\rar{f^0}&X^1\rar{f^1}&\cdots\rar{f^{n-1}}&X^n\rar{f^n}&X^{n+1},
  \end{tikzcd}
  \]
  we obtain that for all $Y\in\M$ and for all $k\in\set{1,\dots,n-1}$
  the cohomology of the induced complex
  \[
  \begin{tikzcd}
    \A(Y,X^1)\rar&\cdots\rar&\A(Y,X^n)\rar&\A(Y,X^{n+1})
  \end{tikzcd}
  \]
  at $\A(Y,X^{k+1})$ is isomorphic to $\Ext_\A^k(Y,X^0)$ which
  vanishes since $\M$ is an $n$-cluster-tilting subcategory of $\A$.
  This shows that $(f^0,\dots,f^{n-1})$ is an $n$-kernel of $f^n$ in
  $\M$.  That $\M$ also satisfies axiom
  \ref{ax-ab:epis-are-admissible} now follows by duality.  This
  concludes the proof of the theorem.
\end{proof}

\begin{definition}
  Let $\M$ be an $n$-abelian category. We say that \emph{$\M$ is
    projectively generated} if for every object $M\in\M$ there exists
  a projective object $P\in\M$ and an epimorphism $P\to M$.  The
  notion of \emph{injectively cogenerated $n$-abelian category} is
  defined dually.
\end{definition}

Our next aim is to show that a partial converse of
\th\ref{recognition-thm-abelian} holds for projectively generated
$n$-abelian categories. For this, we remind the reader of the notion
of coherent modules on an additive category.

Let $\C$ be a \emph{small} additive category. A \emph{$\C$-module} is
a contravariant functor $F\colon \C\to \Mod\ZZ$. The category $\Mod\C$
of $\C$-modules is an abelian category. Morphisms in $\Mod\C$ are
natural transformations of contravariant functors.  If $M,N\in\Mod\C$,
then we denote the set of natural transformations $M\to N$ by
$\Hom_\C(M,N)$. As a consequence of Yoneda's lemma, representable
functors are projective objects in $\Mod\C$. The \emph{category of
  coherent $\C$-modules}, denoted by $\mod\C$, is the full subcategory
of $\Mod\C$ whose objects are the $\C$-modules $F$ such that there
exists a morphism $f\colon X \to Y$ in $\C$ and an exact sequence of
functors
\[
\begin{tikzcd}
  \C(-,X)\rar{?\cdot f}&\C(-,Y)\rar&F\rar&0.
\end{tikzcd}
\]
Note that $\mod\C$ is closed under cokernels and extensions in
$\Mod\C$. Moreover, $\mod\C$ is closed under kernels in $\Mod\C$ if
and only if $\C$ has weak kernels, in which case $\mod\C$ is an
abelian category, see \cite[page 41]{auslander_representation_1971}.
For further information on coherent $\C$-modules we refer the reader
to \cite{auslander_coherent_1966}.

Our aim is to prove the following theorem.

\begin{theorem}
  \th\label{simple-embedding-thm} Let $\M$ be a small projectively generated $n$-abelian
  category. Let $\P$ be the category of projective objects in $\M$
  and $F\colon \M\to\mod\P$ the functor defined by
  $FX:=\M(-,X)|_{\P}$.  Also, let
  \[
    F\M:=\setP{M\in\mod\P}{\exists X\in\M\text{ such that } M\cong FX}
  \]
  be the essential image of $F$. If there exist an exact duality
  $D\colon\mod\P\to\mod\P^\op$, then $F$ is a fully faithful functor
  such that $F\M$ is an $n$-cluster-tilting subcategory of $\mod\P$.
\end{theorem}

\begin{remark}
  In \th\ref{simple-embedding-thm}, the condition on existence of an
  exact duality $D\colon\mod\P\to\mod\P^\op$ is satisfied, for
  example, if $\P$ is a dualizing $k$-variety over a commutative
  artinian ring in the sense of \cite{auslander_stable_1974}.
\end{remark}

In fact, instead of proving \th\ref{simple-embedding-thm}, we prove
the following more general statement.

\begin{lemma}
  \th\label{embedding-thm} Let $\M$ be a small projectively generated
  $n$-abelian category.  Let $\P$ be the category of projective
  objects in $\M$ and $F\colon \M\to\mod\P$ the functor defined by
  $FX:=\M(-,X)|_{\P}$.  Also, let
  \[
  F\M:=\setP{M\in\mod\P}{\exists X\in\M\text{ such that } M\cong FX}
  \]
  be the essential image of $F$.  Then, the following statements hold:
  \begin{enumerate}
  \item\label{it:modP-abelian} The category $\mod\P$ is abelian.
  \item\label{it:F-embedding} The functor $F\colon \M\to\mod\P$ is
    fully faithful.
  \item\label{it:FM-ct-property} For all $k\in\set{1,\dots,n-1}$ we
    have $\Ext_\P^k(F\M,F\M)=0$.
  \item\label{it:FM-ct-A}
    We have
    \[
    F\M=\setP{X\in\mod\P}{\forall
      k\in\set{1,\dots,n-1}\ \Ext_\P^k(X,F\M)=0}
    \]
    where $\Ext_\P^k$ is the bifunctor of degree $k$ extensions in the
    abelian category $\mod\P$.
  \item\label{it:FM-ct-B} We have
    \[
    F\M=\setP{X\in\mod\P}{\forall k\in\set{1,\dots,n-1}\
      \Ext_\P^k(F\M,X)=0}.
    \]
  \item\label{it:contra-finite} The subcategory $F\M$ is
    contravariantly finite in $\mod\P$.
  \item\label{it:cova-finite} If there exist an exact duality
    $D\colon\mod\P\to\mod\P^\op$, then $F\M$ is covariantly finite in
    $\mod\P$. Hence, $F\M$ is a functorially finite subcategory of
    $\mod\P$ in this case.
  \item\label{it:generating-cogenerating} If there exist an exact
    duality $D\colon\mod\P\to\mod\P^\op$, then $F\M$ is a
    generating-cogenerating subcategory of $\mod\P$.
  \end{enumerate}
\end{lemma}
\begin{proof}
  \eqref{it:modP-abelian} This statements follows from the fact that
  $\M$ has weak kernels \cite{auslander_coherent_1966}.

  \eqref{it:F-embedding} Let $M\in\M$. Since $\M$ is projectively
  generated, there exists an $n$-exact sequence
  \begin{equation}
    \label{eq:enough-in-proof}
    \begin{tikzcd}
      K_{n}\rar&\cdots\rar&K_1\rar&K_0\rar{u}&M
    \end{tikzcd}
  \end{equation}
  where $K^0$ is a projective object.  For the same reason, there
  exist a projective object $P^1$ and an epimorphism $v\colon P^1\to
  K_1$. Let $f=vu$ and put $P^0:=K_0$. It follows that the sequence
  \[
  \begin{tikzcd}
    FP^1\rar{Ff}&FP^0\rar&FM\rar&0
  \end{tikzcd}
  \]
  is exact in $\Mod\P$. This shows that $FM\in\mod\P$. That $F$ is
  fully faithful follows from Yoneda's lemma and the existence of a
  sequence of the form $FP^1\xto{Ff} FP^0\to FM\to 0$ with $P^0,P^1\in\M$ for each $M\in\M$.

  \eqref{it:FM-ct-property} Let us show that for every
  $M,N\in\M$ we have $\Ext_\P^k(FM,FN)=0$ for all
  $k\in\set{1,\dots,n-1}$. Consider an $n$-exact sequence of the form
  \eqref{eq:enough-in-proof}. Applying $F$ to
  \eqref{eq:enough-in-proof} yields an exact sequence
  \begin{equation}
    \label{eq:last}
    \begin{tikzcd}[column sep=small]
      0\rar&FK_n\rar&FK_{n-1}\rar&\cdots\rar&FK_1\rar&FK_0\rar&FM\rar&0.
    \end{tikzcd}
  \end{equation}
  Since $F$ is fully faithful there is an isomorphism of complexes
  \[
  \begin{tikzcd}
    \Hom_\P(FK_0,FN)\rar\dar&\cdots\rar&\Hom_\P(FK_{n-1},FN)\rar\dar&\Hom_\P(FK_{n},FN)\dar\\
    \A(K_0,N)\rar&\cdots\rar&\A(K_{n-1},N)\rar&\A(K_n,N)
  \end{tikzcd}
  \]
  Note that the bottom row is exact by the property of $n$-exact
  sequences, hence the top row is also exact. Put
  $C_\ell:=\coker(FK_{\ell+1}\to FK_{\ell})$. Note that
  $C_0:=FM$. There is an exact sequence
  \[
  \begin{tikzcd}[column sep=tiny]
    \Hom_\P(FK_0,FN)\rar&\Hom_\P(C_1,FN)\rar&\Ext_\P^1(FM,FN)\rar&\Ext_\P^1(FK_0,FN)=0
  \end{tikzcd}
  \]
  (recall that $FK_0$ is projective in $\mod\P$).  Since the top row
  of the above diagram is exact, this implies that
  $\Ext_\P^1(FM,FN)=0$. Hence we have
  $\Ext_\P^1(F\M,F\M)=0$.  We show that
  we have a sequence of isomorphisms
  \[
  \Ext_\P^{\ell}(C_0,FN)\cong\Ext_\P^{\ell-1}(C_1,FN)\cong\cdots\cong\Ext_\P^{1}(C_{\ell-1},FN)=0
  \]
  for all $\ell\in\set{1,\dots,n-1}$. If $\ell=1$, then the claim is
  trivial. Inductively, suppose that we have shown that
  $\Ext_\P^{m}(F\M,F\M)=0$ for all $1\leq m\leq \ell-1$. Firstly, note
  that
  applying the functor $\Hom_\P(-,FN)$ to the exact sequence $0\to
  C_{\ell}\to FK_{\ell-1}\to C_{\ell-1}\to0$ gives an exact sequence
  \[
  \begin{tikzcd}[column sep=tiny]
    \Hom_\P(FK_{\ell-1},FN)\rar&\Hom_\P(C_{\ell},FN)\rar&\Ext_\P^1(C_{\ell-1},FN)\rar&\Ext_\P^1(FK_{\ell-1},FN)=0,
  \end{tikzcd}
  \]
  Since the top row of the above diagram is exact, this implies that
  $\Ext_\P^1(C_{\ell-1},FN)=0$. Secondly, by the induction
  hypothesis, for each $1\leq m\leq \ell-1$ applying the functor
  $\Hom_\P(-,FN)$ to the exact sequence $0\to C_{\ell-m}\to
  FK_{\ell-m-1}\to C_{\ell-m-1}\to0$ yields an exact sequence
  \[
  {\small
    \begin{tikzcd}[column sep=tiny]
      0={}_\P^{m}(FK_{\ell-m-1},FN)\rar&{}_\P^{m}(C_{\ell-m},FN)\rar&{}_\P^{m+1}(C_{\ell-m-1},FN)\rar&{}_\P^{m+1}(FK_{\ell-m-1},FN)=0
    \end{tikzcd}
  }
  \]
  where we omitted $\Ext$ because of lack of space (for $m=\ell-1$,
  recall that $FK_0$ is projective in $\mod\P$). The claim follows
  since $C_0=FM$.

  \eqref{it:FM-ct-A} Let $G\in\mod\P$ be such that for all $N\in\M$ and for all
  $k\in\set{1,\dots,n-1}$ we have $\Ext_\P^k(G,FN)=0$.  We need to
  show that there exists $M\in\M$ such that $G$ and $FM$ are
  isomorphic. For this, let
  \[
  \begin{tikzcd}
    FP_1\rar{Ff_0}&FP_0\rar&G\rar&0
  \end{tikzcd}
  \]
  be a projective presentation of $G$ in $\mod\P$ and let 
  \[
  \begin{tikzcd}
    X_n\rar{f_n}&\cdots\rar&X_1\rar{f_1}&X_0
  \end{tikzcd}
  \]
  be an $n$-kernel of $f_0$ (by convention, $X_0:=P_1$). Let $N\in\M$. Applying the
  functor $\Hom_\P(F(-),FN)$ to the sequence
  $\tuple{f_n,\dots,f_1,f_0}$ together with the fact that $F$ is fully
  faithful yields a commutative diagram
  \[
  \begin{tikzcd}
    \Hom_\P(FP_0,FN)\rar\dar&\Hom_\P(FX_0,FN)\rar\dar&\cdots\rar&\Hom_\P(FX_n,FN)\dar\\
    \A(P_0,N)\rar&\A(X_0,N)\rar&\cdots\rar&\A(X_n,N)
  \end{tikzcd}
  \]
  where the vertical arrows are isomorphisms. By what we showed in the
  previous paragraph and \th\ref{lemma-exts}, for all
  $k\in\set{1,\dots,n-1}$ the homology of the top row at
  $\Hom_\P(FX_{k-1},FN)$ is isomorphic to $\Ext_\P^k(G,FN)$ which
  vanishes by hypothesis. It follows that the bottom row is an exact
  sequence. By applying
  \th\ref{weak-cokernel-can-be-extended-to-n-cokernel} to $f_{n-1}$
  and the sequence $\tuple{f_{n-2},\dots,f_1,f_0}$ we deduce that
  $f_0$ admits a cokernel $P_0\to M$ in $\M$. Finally,
  \th\ref{projectives-are-strongly-projective} implies that the
  sequence
  \[
  \begin{tikzcd}
    FP_1\rar{Ff_0}&FP_0\rar&FM\rar&0
  \end{tikzcd}
  \]
  is exact in $\mod\P$.  Therefore $G$ is isomorphic to $FM$ which is
  what we needed to show.

  \eqref{it:FM-ct-B} Let us show that if $G\in\mod\P$ is such that for all
  $M\in\M$ and for all $k\in\set{1,\dots,n-1}$ we have
  $\Ext_\P^k(FM,G)=0$, then $G\in F\M$. Indeed, let
  \[
  \begin{tikzcd}
    FP_1\rar{Ff^0}&FP_0\rar&G\rar&0
  \end{tikzcd}
  \]
  be a projective presentation of $G$ in $\mod\P$.  By axiom
  \ref{ax-ab:n-co-kernel-exists}, there exists an $n$-cokernel
  $\tupleP{f^k\colon M^{k-1}\to M^k}{1\leq k\leq n}$ of $f^0$ in $\M$
  (by convention, $M^0:=P^0$). Then,
  \th\ref{projectives-are-strongly-projective} implies that the
  sequence
  \[
  \begin{tikzcd}
    FP^1\rar{Ff^0}&FP^0\rar{Ff^1}&FM^1\rar{Ff^2}&\cdots\rar{Ff^n}&FM^n\rar&0
  \end{tikzcd}
  \]
  is exact in $\mod\P$. It follows that there is an exact sequence
  \[
  \begin{tikzcd}
    0\rar&G\rar&FM^1\rar{Ff^2}&\cdots\rar{Ff^n}&FM^n\rar&0.
  \end{tikzcd}
  \]
  For each $k\in\set{1,\dots,n-1}$ let $G^k:=\ker Ff^{k+1}$; we claim
  that $G^k\in F\M$. Note that $\Ext_\P^1(FM^k,G^k)=0$ by
  hypothesis. In particular, $G^{n-1}$ is a direct summand of
  $FM^{n-1}$. Since $\M$ is idempotent complete, there exists an
  object $L\in\M$ such that $G^{n-1}\cong FL\in F\M$. Inductively, we
  deduce that $G^1=G\in F\M$.

  \eqref{it:contra-finite} Let $G\in\mod\P$ and take a projective
  presentation
  \[
  \begin{tikzcd}
    FP^1\rar{Ff}&FP^0\rar{p}&G\rar&0
  \end{tikzcd}
  \]
  of $G$ in $\mod\P$. Let $g\colon P^0\to M$ be a weak cokernel of $f$
  in $\M$. We obtain the solid part of the following commutative
  diagram:
  \[
  \begin{tikzcd}
    FP^1\rar{Ff}&FP^0\rar{Fg}\dar{p}&FM\\
    &G\urar\rar[dotted]{h}&FN
  \end{tikzcd}
  \]
  Let $h\colon G\to FN$ be a morphism in $\mod\P$. Since $g$ is a weak
  cokernel there exists a morphism $Fq\colon FM\to FN$ such that
  the diagram
  \[
  \begin{tikzcd}
    FP^0\rar{Fg}\dar{p}&FM\dar{Fq}\\
    G\rar{h}&FN
  \end{tikzcd}
  \]
  is commutative. Finally, given that $p$ is an epimorphism, we
  conclude that the diagram
  \[
  \begin{tikzcd}
    G\rar\drar[swap]{h}&FM\dar{Fq}\\
    &FN
  \end{tikzcd}
  \]
  commutes. This shows that $G\to FM$ is a right $F\M$-approximation
  of $G$.

  \eqref{it:cova-finite} and \eqref{it:generating-cogenerating} It is
  clear that the assumptions imply that $\mod\P$ is injectively
  cogenerated. From part \eqref{it:FM-ct-property} we deduce that for
  every injective object $I\in\mod\P$ we have $I\in F\M$. Hence
  \eqref{it:cova-finite} follows by duality from part
  \eqref{it:contra-finite}. Also, \eqref{it:generating-cogenerating}
  follows since it is now clear that $F\M$ is a cogenerating
  subcategory of $\mod\P$.
\end{proof}

\section{$n$-exact categories}
\label{sec:n-exact-categories}

In this section we introduce $n$-exact categories and establish their
basic properties. We show that $n$-cluster-tilting subcategories of
exact categories have a natural $n$-exact structure.

\subsection{Definition and basic properties}

The treatment of this section is parallel to B\"uhler's exposition of
the basics of the theory of exact categories given in
\cite[Sec. 2]{buhler_exact_2010}.

Let $\C$ be an additive category. If $\X$ is a class of $n$-exact
sequences in $\C$, then we call its members \emph{$\X$-admissible
  $n$-exact sequences}.  Furthermore, if
\[
\begin{tikzcd}
  X^0\rar[tail]{d^0}&X^1\rar{d^1}&\cdots\rar{d^{n-1}}&X^n\rar[two
  heads]{d^n}&X^{n+1}
\end{tikzcd}
\]
is an $\X$-admissible $n$-exact sequence, we say that $d^0$ is an
\emph{$\X$-admissible monomorphism} and that $d^n$ is an
\emph{$\X$-admissible epimorphism}.  In analogy with
\cite{buhler_exact_2010}, we depict $\X$-admissible monomorphisms by
$\mono$ and $\X$-admissible epimorphisms by $\epi$.  A sequence
$\mono\to\cdots\to\epi$ of $n+1$ morphisms always denotes an
$\X$-admissible $n$-exact sequence.  When the class $\X$ is clear from
the context, we write ``admissible'' instead of ``$\X$-admissible''.

\begin{definition}
  We say that a morphism $f\colon X\to Y$ of $n$-exact sequences in
  $\CC^{n}(\C)$ is a \emph{weak isomorphism} if $f^k$ and $f^{k+1}$
  are isomorphisms for some $k\in\set{0,1,\dots,n+1}$ with $n+2:=0$
  (this terminology is borrowed from the theory of $(n+2)$-angulated
  categories, see Section \ref{sec:n-angulated-categories}).  Note
  that weak isomorphisms induce isomorphisms in $\KK(\C)$ by
  \th\ref{weak-isos-induce-homotopy-equivalences}.
\end{definition}

\begin{definition}
  \th\label{def:n-exact-category} Let $n$ be a positive integer and
  $\M$ an additive category.  An \emph{$n$-exact structure on $\M$} is
  a class $\X$ of $n$-exact sequences in $\M$, closed under weak
  isomorphisms of $n$-exact sequences, and which satisfies the
  following axioms:
  \begin{axioms}
  \item[E0]{The sequence $0\mono0\to\cdots\to0\epi0$ is an
      $\X$-admissible $n$-exact sequence.\label{ax-ex:zero-seq}}
  \item[E1]{The class of $\X$-admissible monomorphisms is closed under
      composition.\label{ax-ex:admissible-monos-closed-under-composition}}
  \item[E1$^\op$]{The class of $\X$-admissible epimorphisms is closed
      under
      composition.\label{ax-ex:admissible-epis-closed-under-composition}}
  \item[E2]{For each $\X$-admissible $n$-exact sequence $X$ and each
      morphism $f\colon X^0\to Y^0$, there exists an $n$-pushout
      diagram of $(d_X^0,\dots,d_X^{n-1})$ along $f$ such that $d_Y^0$
      is an $\X$-admissible monomorphism. The situation is illustrated
      in the following commutative
      diagram:\label{ax-ex:n-pushout-exists}}
    \[
    \begin{tikzcd}
      X^0\rar[tail]{d_X^0}\dar{f}&X^1\rar{d_X^1}\dar[dotted]&\cdots\rar{d_X^{n-1}}&X^n(\dar[dotted]\rar[two heads]{d_X^{n+1}}&X^{n+1})\\
      Y^0\rar[tail,dotted]{d_Y^0}&Y^1\rar[dotted]{d_Y^1}&\cdots\rar[dotted]{d_Y^{n-1}}&Y^n
    \end{tikzcd}
    \]
  \item[E2$^\op$]{For each $\X$-admissible $n$-exact sequence $X$ and
      each morphism $g\colon Y^{n+1}\to X^{n+1}$, there exists an
      $n$-pullback diagram of $(d_X^1,\dots,d_X^n)$ along $g$ such
      that $d_Y^n$ is an $\X$-admissible epimorphism. The situation is
      illustrated in the following commutative
      diagram:\label{ax-ex:n-pullback-exists}}
    \[
    \begin{tikzcd}
      {}&Y^1\rar[dotted]{d_Y^1}\dar[dotted]&\cdots\rar[dotted]{d_Y^{n-1}}&Y^n\rar[two heads,dotted]{d_Y^n}\dar[dotted]&Y^{n+1}\dar{g}\\
      (X^0\rar[tail]{d_X^0}&)X^1\rar{d_X^1}&\cdots\rar{d_X^{n-1}}&X^n\rar[two
      heads]{d_X^n}&X^{n+1}
    \end{tikzcd}
    \]
  \end{axioms}
  An \emph{$n$-exact category} is a pair $(\M,\X)$ where $\M$ is an
  additive category and $\X$ is an $n$-exact structure on $\M$. If the
  class $\X$ is clear from the context, we identify $\M$ with the pair
  $(\M,\X)$.
\end{definition}

\begin{remark}
  Our choice of axioms for $n$-exact categories is inspired by
  Keller's minimal list of axioms for exact categories
  \cite[App. A]{keller_chain_1990}, although we opt for a more
  convenient self-dual collection.  In particular, we point out to the
  reader who is more familiar with Quillen's axioms that the so-called
  ``obscure axiom'', axiom c) of \cite[Sec. 2]{quillen_higher_1973},
  is redundant, see \cite[p. 4, Prop. 2.16]{buhler_exact_2010}.
\end{remark}

The following result shows that $n$-abelian categories have a
canonical $n$-exact structure.  Therefore the class of $n$-exact
categories contains the class of $n$-abelian categories.

\begin{theorem}
  \th\label{n-abelian-cats-are-n-exact} Let $\M$ be an $n$-abelian
  category and $\X$ the class of all $n$-exact sequences in $\M$.
  Then, $(\M,\X)$ is an $n$-exact category.
\end{theorem}
\begin{proof}
  We shall show that $(\M,\X)$ satisfies the axioms of
  \th\ref{def:n-exact-category}.  It is obvious that the class $\X$ is
  closed under weak isomorphisms and that axiom \ref{ax-ex:zero-seq}
  is satisfied.  By axioms \ref{ax-ab:n-co-kernel-exists} and
  \ref{ax-ab:monos-are-admissible} in \th\ref{def:n-abelian-category},
  every monomorphism in $\M$ is the leftmost morphism in an $n$-exact
  sequence, see \th\ref{n-abelian-monos-are-admissible}.  Since the
  composition of two monomorphisms is again a monomorphism, axiom
  \ref{ax-ex:admissible-monos-closed-under-composition} is satisfied.
  That axiom \ref{ax-ex:admissible-epis-closed-under-composition} is
  also satisfied then follows by duality.  Finally,
  \th\ref{n-abelian-cats-have-n-pushout-diagrams} and its dual show
  that axioms \ref{ax-ex:n-pushout-exists} and
  \ref{ax-ex:n-pullback-exists} are satisfied.  This shows that
  $(\M,\X)$ is an $n$-exact category.
\end{proof}

We begin our investigation of the properties of $n$-exact categories
with a simple but useful observation.

\begin{lemma}
  \th\label{n-exact-satisfies-A2} Let $(\M,\X)$ be an $n$-exact
  category and $X^0\mono X^1$ and admissible monomorphism.  If the
  sequence $\tupleP{X^{k}\to X^{k+1}}{1\leq k\leq n}$ is an
  $n$-cokernel of $f^0$, then the sequence
  \[
  \begin{tikzcd}
    X:X^0\rar[tail]&X^1\rar&\cdots\rar&X^n\rar&X^{n+1}
  \end{tikzcd}
  \]
  is an admissible $n$-exact sequence.
\end{lemma}
\begin{proof}
  Since $X^0\mono X^1$ is an admissible monomorphism, there exists an
  admissible $n$-exact sequence $Y$ whose first morphism is $X^0\mono
  X^1$. By the factorization property of $n$-cokernels, there exists a
  weak isomorphism $X\to Y$. Then, $X\in\X$ since the class $\X$ is
  closed under weak isomorphisms.
\end{proof}

The next result shows that the $n$-exact structure of an $n$-exact
category is closed under direct sums.

\begin{proposition}
  \th\label{X-is-closed-under-direct-sums} Let $(\M,\X)$ be an
  $n$-exact category, and $X$ and $Y$ be admissible $n$-exact
  sequences.  Then, their direct sum $X\oplus Y$ is an admissible
  $n$-exact sequence.
\end{proposition}
\begin{proof}
  This is an adaptation of the proof of
  \cite[Prop. 2.9]{buhler_exact_2010}.  Clearly, $X\oplus Y$ is an
  $n$-exact sequence.  We claim that $d_X^0\oplus 1_{Y^0}$ is an
  admissible monomorphism.  Indeed, the sequence
  \[
  \begin{tikzcd}[column sep=large]
    X^0\oplus Y^0\rar{d_X^0\oplus 1_{Y^0}}&X^1\oplus Y^0\rar{(d_X^1\
      0)}&X^2\rar{d_X^2}&\cdots\rar[two heads]{d_X^{n}}&X^{n+1}
  \end{tikzcd}
  \]
  is an $n$-exact sequence.  Since $d_X^{n}$ is an admissible
  epimorphism, it follows from the dual of
  \th\ref{n-exact-satisfies-A2} that this sequence is moreover an
  admissible $n$-exact sequence and therefore $d_X^0\oplus 1_{Y^0}$ is
  an admissible monomorphism.  We can show that $1_{X^1}\oplus d_Y^0$
  is an admissible monomorphism with a similar argument.  Next,
  observe that
  \[
  d_X^0\oplus d_Y^0=(d_X^0\oplus 1_{Y^0})\cdot(1_{X^1}\oplus d_Y^0).
  \]
  By axiom \ref{ax-ex:admissible-monos-closed-under-composition} we
  have that $d_X^0\oplus d_Y^0\colon X^0\oplus Y^0\to X^1\oplus Y^1$
  is an admissible monomorphism.  Since $X\oplus Y$ is an admissible
  $n$-exact sequence, the claim follows from
  \th\ref{n-exact-satisfies-A2}.
\end{proof}

\begin{remark}
  \th\label{contractible-sequences-are-admissible} Let $(\M,\X)$ be an
  $n$-exact category with $n\geq 2$. It follows from
  \ref{ax-ex:zero-seq} and the fact that $\X$ is assumed to be closed
  under weak isomorphisms that that for all $X\in\M$ the morphism
  $1_X$ is an admissible monomorphism. Indeed, there is a weak
  isomorphism
  \[
    \begin{tikzcd}
      X\rar{1_X}\dar&X\rar{0}\dar&\cdots\rar&0\dar\rar&0\dar\\
      0\rar[tail]&0\rar&\cdots\rar&0\rar[two heads]&0
    \end{tikzcd}
  \]
  Moreover, one can easily show that two contractible $n$-exact
  sequences with the same end terms are weakly isomorphic, hence
  \th\ref{X-is-closed-under-direct-sums} implies that all contractible
  $n$-exact sequences are admissible.  Also, it is straightforward to
  verify that the class of all contractible $n$-exact sequences in an
  additive category $\M$ is an $n$-exact structure; in fact, this is
  the smallest $n$-exact structure on $\M$.  In particular, every
  additive category can be considered as an $n$-exact category with a
  ``contractible $n$-exact structure'', for each positive integer $n$.
\end{remark}

The following characterization of $n$-pushout diagrams of $n$-exact
sequences should be compared with
\cite[Prop. 2.12]{buhler_exact_2010}.

\begin{proposition}
  \th\label{props-of-n-pushout-diagrams} Let $(\M,\X)$ be an $n$-exact
  category.  Suppose that we are given a commutative diagram
  \begin{equation}
    \label{eq:candidate-n-pushout-diagram}
    \begin{tikzcd}
      X^0\rar[tail]{d_X^0}\dar{f^0}&X^1\rar{d_X^1}\dar{f^1}&\cdots\rar{d_X^{n-1}}&X^{n}\dar{f^{n}}(\rar[two
      heads]{d_X^{n}}&X^{n+1})\\
      Y^0\rar[tail]{d_Y^0}&Y^1\rar{d_Y^1}&\cdots\rar{d_Y^{n-1}}&Y^{n}
    \end{tikzcd}
  \end{equation}
  in which the top row is an admissible $n$-exact sequence and $d_Y^0$
  is an admissible monomorphism.  Then the following statements are
  equivalent:
  \begin{enumerate}
  \item\label{it:dia-is-a-n-pushout-diagram} Diagram
    \eqref{eq:candidate-n-pushout-diagram} is an $n$-pushout diagram.
  \item\label{it:mapping-cone-is-n-exact} The mapping cone of diagram
    \eqref{eq:candidate-n-pushout-diagram} is an admissible $n$-exact
    sequence.
  \item\label{it:dia-is-both-n-pushout-and-pullback} Diagram
    \eqref{eq:candidate-n-pushout-diagram} is both an $n$-pushout and
    an $n$-pullback diagram.
  \item\label{it:induced-n-exact-sequence} There exists a commutative
    diagram
    \[
    \begin{tikzcd}
      X\dar{f}&X^0\rar[tail]{d_X^0}\dar{f^0}&X^1\rar{d_X^1}\dar{f^1}&\cdots\rar{d_X^{n-1}}&X^{n}\dar{f^{n}}\rar[two
      heads]{d_X^{n}}&X^{n+1}\dar[equals]\\
      Y&Y^0\rar[tail]{d_Y^0}&Y^1\rar{d_Y^1}&\cdots\rar{d_Y^{n-1}}&Y^{n}\rar[dotted,
      two heads]{d_Y^{n}}&X^{n+1}
    \end{tikzcd}
    \]
    whose rows are admissible $n$-exact sequences.
  \end{enumerate}
\end{proposition}
\begin{proof}
  \eqref{it:dia-is-a-n-pushout-diagram} implies
  \eqref{it:mapping-cone-is-n-exact}.  This is an adaptation of the
  proof of \cite[Prop. 2.12]{buhler_exact_2010}.  Since the leftmost
  $n$ squares in \eqref{eq:candidate-n-pushout-diagram} form an
  $n$-pushout diagram, by definition its mapping cone gives an
  $n$-cokernel of the morphism
  \[
  d_C^{-1}=[-d_X^0\ f^0]^\top\colon X^0\to X^1\oplus Y^0.
  \]
  Hence, by \th\ref{n-exact-satisfies-A2}, it is sufficient to show
  that $d_C^{-1}$ is an admissible monomorphism.  For this, observe
  that $d_C^{-1}$ equals the composition
  \[
  \begin{tikzcd}[ampersand replacement=\&, column sep=large]
    X^0\rar[tail]{
      \begin{bmatrix}
        1\\
        0
      \end{bmatrix}
    }\&X^0\oplus Y^0\rar[tail]{
      \begin{bmatrix}
        -1&0\\
        f^0&1
      \end{bmatrix}
    }[swap]{\sim}\& X^0\oplus Y^0\rar[tail]{
      \begin{bmatrix}
        d_X^0&0\\
        0&1
      \end{bmatrix}
    }\&X^1\oplus Y^0
  \end{tikzcd}
  \]
  where the rightmost morphism is an admissible monomorphism by
  \th\ref{X-is-closed-under-direct-sums} and the remaining morphisms
  are admissible monomorphisms by
  \th\ref{contractible-sequences-are-admissible}.  Thus $d_C^{-1}$ is
  an admissible monomorphism by axiom
  \ref{ax-ex:admissible-monos-closed-under-composition}.
  
  That \eqref{it:mapping-cone-is-n-exact} implies
  \eqref{it:dia-is-both-n-pushout-and-pullback} follows directly from
  the definitions of $n$-pushout and $n$-pullback diagrams.  That
  \eqref{it:dia-is-both-n-pushout-and-pullback} implies
  \eqref{it:dia-is-a-n-pushout-diagram} is obvious.  Therefore
  statements \eqref{it:dia-is-a-n-pushout-diagram},
  \eqref{it:mapping-cone-is-n-exact} and
  \eqref{it:dia-is-both-n-pushout-and-pullback} are equivalent.

  \eqref{it:dia-is-a-n-pushout-diagram} implies
  \eqref{it:induced-n-exact-sequence}.  We begin by constructing the
  morphism $d_Y^n\colon Y^n\to X^{n+1}$.  Since $d_C^{n-1}$ is a
  cokernel of $d_C^{n-2}$ (here $C$ is the mapping cone of $f$), there
  exists a unique morphism $d_Y^n\colon Y^n\to X^{n+1}$ such that
  $d_X^n=f^nd_Y^n$ and $d_Y^{n-1}d_Y^n=0$, see
  \eqref{differential_mapping_cone}.  Since $d_X^n$ is an epimorphism,
  it follows immediately that so is $d_Y^n$.  It remains to show that
  $d_Y^n$ is a cokernel of $d_Y^{n-1}$.  For this, let
  $u\colon Y^n\to M$ be a morphism such that $d_Y^{n-1}u=0$.  Then we
  have that
  \[
  d_X^{n-1}(f^nu)=(f^{n-1}d_Y^{n-1})u=0.
  \]
  Since $d_X^n$ is a cokernel of $d_X^{n-1}$, there exists a morphism
  $v\colon X^{n+1}\to M$ such that $f^nu=d_X^nv$.  It follows that
  \[
  f^nu=d_X^nv=f^n(d_Y^nv)
  \]
  and
  \[
  d_Y^{n-1}u=0=d_Y^{n-1}(d_Y^nv).
  \]
  Since $d_C^{n-1}$ is a cokernel of $d_C^{n-2}$, we have $u=d_Y^nv$.
  This shows that the epimorphism $d_Y^n$ is a cokernel of
  $d_Y^{n-1}$.
  
  Let $2\leq k\leq n$.  We need to show that $d_Y^k$ is a weak
  cokernel of $d_Y^{k-1}$.  Let $u\colon Y^k\to M$ be a morphism such
  that $d_Y^{k-1}u=0$.  Then we have
  \[
  d_X^{k-1}(f^ku)=(f^{k-1}d_Y^{k-1})u=0.
  \]
  Since $d_X^k$ is a weak cokernel of $d_X^{k-1}$, there exists a
  morphism $v\colon X^{k+1}\to M$ such that $f^ku=d_X^kv$.  Hence,
  given that $d_C^{k-1}$ is a weak cokernel of $d_C^{k-2}$, there
  exists a morphism $w\colon Y^{k+1}\to M$ such that $d_Y^kw=u$, see
  \eqref{differential_mapping_cone}.  Therefore $d_Y^k$ is a weak
  cokernel of $d_Y^{k-1}$, as required.  This shows that
  $(d_Y^1,\dots,d_Y^n)$ is an $n$-cokernel of $d_Y^0$, so we have
  finished the construction of the required commutative diagram.
  Moreover, by \th\ref{n-exact-satisfies-A2}, the sequence $Y$ is an
  admissible $n$-exact sequence.

  \eqref{it:induced-n-exact-sequence} implies
  \eqref{it:mapping-cone-is-n-exact}.  We assume that $n\geq 2$.  The
  case $n=1$ can be shown by combining the arguments below, and is
  easily found in the literature, see for example
  \cite[Prop. 2.12]{buhler_exact_2010}.  By definition, we need to
  show that in the mapping cone $C=C(f)$ we have that
  $(d_C^{0},d_C^{1},\dots,d_C^n)$ is an $n$-cokernel of $d_C^{-1}$.
  
  Let $2\leq k\leq n$.  We shall show that $d_C^{k-1}$ is a weak
  cokernel of $d_C^{k-2}$, see \eqref{differential_mapping_cone}.
  Indeed, let $u\colon Y^{k-1}\to M$ and $v\colon X^k\to M$ be
  morphisms such that $d_Y^{k-2}u=0$ and $d_X^{k-1}v=f^{k-1}u$.  Since
  $d_Y^{k-1}$ is a weak cokernel of $d_Y^{k-2}$, there exists a
  morphism $w\colon Y^k\to M$ such that $u=d_Y^{k-1}w$.  Moreover,
  note that
  \[
  d_X^{k-1}v =f^{k-1}u =f^{k-1}(d_Y^{k-1}w) =d_X^{k-1}(f^kw)
  \]
  for $f$ is a morphism of complexes.  Given that $d_X^k\colon X^k\to
  X^{k+1}$ is a weak cokernel of $d_X^{k-1}\colon X^{k-1}\to X^k$,
  there exists a morphism $h^{k+1}\colon X^{k+1}\to M$ such that
  $f^kw-v=d_X^kh^{k+1}$.  If $k\neq n$, then the claim follows.  If
  $k=n$, then let $w':=w-d_Y^nh^{n+1}$.  It follows that
  $d_Y^{n-1}w'=d_Y^{n-1}w=u$ and
  \[
  f^nw'=f^nw-f^nd_Y^nh^{n+1}=v+d_X^nh^{k+1}-d^nh^{k+1}=v.
  \]
  This shows that $d_C^{n-1}$ is a weak cokernel of $d_C^{n-2}$.
  
  We need to show that $d_C^{n-1}$ is a cokernel of $d_C^{n-2}$.  For
  this it is enough to show that $d_C^{n-1}$ is an epimorphism.  Let
  $p\colon Y^n\to M$ be a morphism such that $d_Y^{n-1}p=0$ and
  $f^np=0$.  Then, since $d_Y^n$ is a cokernel of $d_Y^{n-1}$, there
  exists a morphism $q\colon X^{n+1}\to M$ such that $p=d_Y^nq$.
  Thus,
  \[
  d_X^nq =(f^nd_Y^n)q =f^np=0.
  \]
  Since $d_X^n$ is an epimorphism, we have $q=0$ which implies that
  $p=0$.  This shows that $d_C^{n-1}$ is an epimorphism.
  
  It remains to show that $d_C^0$ is a weak cokernel of $d_C^{-1}$.
  Let $y\colon Y^0\to Z^0$ and $v\colon X^0\to Z^0$ be morphisms such
  that $f^0u=d_X^0v$.  By axiom \ref{ax-ex:n-pushout-exists} there
  exists an $n$-pushout diagram of $(d_Y^0,\dots,d_Y^{n-1})$ along
  $u$.  Moreover, since we have shown the implication from
  \eqref{it:dia-is-a-n-pushout-diagram} to
  \eqref{it:induced-n-exact-sequence}, we can construct a commutative
  diagram
  \[
  \begin{tikzcd}
    X\dar{f}&X^0\rar[tail]\dar&X^1\rar\dar&\cdots\rar&X^n\rar[two heads]\dar&X^{n+1}\dar[equals]\\
    Y\dar{g}&Y^0\rar[tail]\dar{u}&Y^1\rar\dar&\cdots\rar&Y^n\rar[two heads]\dar&X^{n+1}\dar[equals]\\
    Z&Z^0\rar[tail]&Z^1\rar&\cdots\rar&Z^n\rar[two
    heads,dotted]&X^{n+1}
  \end{tikzcd}
  \]
  where the leftmost $n$ squares of the two bottom rows form a pushout
  diagram.  It follows that the following diagram commutes
  \[
  \begin{tikzcd}
    X^0\rar[tail]{d_X^0}\dar{0}&X^1\rar{d_X^1}\dar{f^0u-vd_Z^0}&X^2\rar\dar{f^2g^2}&\cdots\rar&X^n\rar[two heads]{d_X^n}\dar{f^n\gamma^n}&X^{n+1}\dar[equals]\\
    Z^0\rar[tail]{d_Z^0}&Z^1\rar{d_Z^1}&Z^2\rar&\cdots\rar&Z^n\rar[two
    heads]&X^{n+1}
  \end{tikzcd}
  \]
  Then, by the \th\ref{comparison-lemma}, there exists a morphism
  $h\colon X^{n+1}\to Z^n$ such that $hd_Z^n=1_{X^{n+1}}$.  Therefore
  $d_Z^{n+1}$ is a split epimorphism.  From the dual of
  \th\ref{split-mono-implies-sequence-contracts} we conclude that
  $d_Z^0$ is a split monomorphism.  It follows that there exists a
  morphism $w\colon Z^1\to Z^0$ such that $d_Z^0w=1_{Z^0}$.  Finally,
  \[
  d_Y^0(g^1w) =u(d_Z^0w)=u
  \]
  and
  \[
  f^0(uw)-d_X^0(vw) =(f^0u-d_X^0v)w =(vd_Z^0)w=v
  \]
  which is what we needed to show.  This concludes the proof.
\end{proof}

The following property is a refinement of
\th\ref{universal-property-of-n-pushout-diagrams}.

\begin{proposition}
  \th\label{universal-property-of-n-pushout-diagrams-n-exact} Let
  $(\M,\X)$ be an $n$-exact category and $g\colon X\to Z$ a morphism
  of admissible $n$-exact sequences.  Then, for every morphism of
  admissible $n$-exact sequences $g\colon X\to Y$ there exists a
  commutative diagram
  \[
  \begin{tikzcd}
    X\dar{f}&X^0\rar[tail]\dar{g^0}&X^1\rar\dar&\cdots\rar&X^n\dar\rar[two heads]&X^{n+1}\dar[equals]\\
    Y\dar{p}&Y^0\rar[tail]\dar[equals]&Y^1\rar\dar[dotted]&\cdots\rar&Y^n\dar[dotted]\rar[two heads]&Y^{n+1}\dar{g^{n+1}}\\
    Z&Z^0\rar[tail]&Z^1\rar&\cdots\rar&Z^n\rar[two heads]&Z^{n+1}
  \end{tikzcd}
  \]
  where $f^0=g^0$ and $p^{n+1}=g^{n+1}$. Moreover, there exists a
  homotopy $h\colon fp\to g$ with $h^1=0$ and $h^{n+1}=0$.
\end{proposition}
\begin{proof}
  By \th\ref{props-of-n-pushout-diagrams} the morphism of admissible
  $n$-exact sequences $f$ exists. Then, by
  \th\ref{universal-property-of-n-pushout-diagrams}, we only need to
  show that $p^nd_Z^n=d_Y^ng^{n+1}$.  Indeed, on one hand we have
  \[
  f^n(g^nd_Z^n)=d_X^ng^{n+1}=f^n(d_Y^ng^{n+1}).
  \]
  On the other hand,
  \[
  d_Y^{n-1}(g^nd_Z^n)=g^{n-1}(d_Z^{n-1}d_Z^n)=0=d_Y^{n-1}(d_Y^ng^{n+1}).
  \]
  Since in the mapping cone $C(f)$ we have that $d_C^{n-1}$ is a
  cokernel of $d_C^{n-2}$, we have $p^nd_Z^n=d_Y^ng^{n+1}$, see
  \eqref{differential_mapping_cone}.  This concludes the proof.
\end{proof}

The next result shows that, in an $n$-exact category, equivalences of
admissible $n$-exact sequences induce isomorphisms in the homotopy
category, \cf \th\ref{weak-isos-induce-homotopy-equivalences} and the
remark after it.

\begin{proposition}
  \th\label{X-closed-under-equivalences} Let $(\M,\X)$ be an $n$-exact
  category, $f\colon X\to Y$ an equivalence of admissible $n$-exact
  sequences. Then, there exists an equivalence of $n$-exact sequences
  $g\colon Y\to X$  such that $f$ and $g$ are mutually inverse
  isomorphisms in $\KK(\M)$.
\end{proposition}
\begin{proof}
  By \th\ref{props-of-n-pushout-diagrams}, the mapping cone $C=C(f)$
  of the diagram
  \[
  \begin{tikzcd}
    X^0\rar\dar[equals]&X^1\rar\dar{f^1}&\cdots\rar&X^{n-1}\rar\dar{f^{n-1}}&X^n\dar{f^n}\\
    Y^0\rar[tail]&Y_1\rar&\cdots\rar&Y^{n-1}\rar&Y^n
  \end{tikzcd}
  \]
  is an admissible $n$-exact sequence. Since $d_C^{-1}\colon X^0\to
  X^1\oplus X^0$ is a split monomorphism, $C$ is a contractible
  $n$-exact sequence, see
  \th\ref{split-mono-implies-sequence-contracts}.  Hence there exists
  a null-homotopy $h\colon 1_C\to 0$. It is straightforward to verify
  that $h$ induces an equivalence of admissible $n$-exact sequences
  $g\colon Y\to X$. Finally, the \th\ref{comparison-lemma} implies
  that $f$ and $g$ induce mutually inverse isomorphisms in $\KK(\M)$.
\end{proof}

By axiom \ref{ax-ex:admissible-monos-closed-under-composition}, the
class of admissible monomorphisms in an $n$-exact category is closed under
composition.  The next result, an analog of Quillen's obscure axiom
for exact categories, shows that a partial converse also holds, \cf
\cite[Prop. 2.16]{buhler_exact_2010}.

\begin{proposition}[Obscure axiom]
  \th\label{obscure-axiom} Let $(\M,\X)$ be an $n$-exact category.
  Suppose there is a commutative diagram
  \begin{equation}
    \label{eq:obscure-axiom}
    \begin{tikzcd}
      X\dar{f}&X^0\rar\dar[equals]&X^1\rar\dar&\cdots\rar&X^n\rar\dar&X^{n+1}\dar\\
      Y&Y^0\rar[tail]&Y^1\rar&\cdots\rar&Y^n\rar[two heads]&Y^{n+1}
    \end{tikzcd}
  \end{equation}
  where the bottom row is an admissible $n$-exact sequence and
  $(d_X^1,\dots,d_X^n)$ is an $n$-cokernel of $d_X^0$.  Then, the top
  row is also an admissible $n$-exact sequence.
\end{proposition}
\begin{proof}
  This is an adaptation of the proof of
  \cite[Prop. 2.16]{buhler_exact_2010}, which is due to Keller.  Since
  a pushout of the bottom row along $d_X^0$ exists, by
  \th\ref{props-of-n-pushout-diagrams} the morphism $[-d_Y^0\
  d_X^0]^\top\colon X^0\to Y^1\oplus X^1$ is an admissible
  monomorphism.  Moreover, the diagram
  \[
  \begin{tikzcd}[ampersand replacement=\&, row sep=large, column
    sep=huge]
    X^0\rar[tail]{
      \begin{bmatrix}
        -d_Y^0& d_X^0
      \end{bmatrix}^\top }\dar[equals]\&Y^1\oplus X^1 \dar[tail]{
      \begin{bmatrix}
        1&f^1\\
        0& 1
      \end{bmatrix}
    }\\
    X^0 \rar{
      \begin{bmatrix}
        0& d_X^0
      \end{bmatrix}^\top } \& Y^1\oplus X^1
  \end{tikzcd}
  \]
  commutes.  Hence $[0\ d_X^0]^\top\colon X^0\to Y^1\oplus X^1$ is the
  composition of an admissible monomorphism with an isomorphism and,
  by axiom \ref{ax-ex:admissible-monos-closed-under-composition} and
  \th\ref{contractible-sequences-are-admissible}, we conclude that
  itself is an admissible monomorphism.  Using the factorization
  property of weak cokernels we can construct a commutative diagram
  \[
  \begin{tikzcd}[ampersand replacement=\&, row sep=large, column
    sep=large]
    X\dar{p}\&X^0\rar{d_X^0}\dar[equals] \&X^1\rar\dar{
      \begin{bmatrix}
        0\\
        1
      \end{bmatrix}
    }\&X^2\rar\dar[dotted]\&\cdots\rar\&X^{n+1}\dar[dotted]\\
    Z\dar{q}\&X^0\rar[tail]{
      \begin{bmatrix}
        0 & d_X^0
      \end{bmatrix}^\top }\dar[equals]\&Y^1\oplus X^1\rar\dar{
      \begin{bmatrix}
        0&1
      \end{bmatrix}
    }\&Z^2\rar\dar[dotted]\&\cdots
    \rar[two heads]\&Z^{n+1}\dar[dotted]\\
    X\&X^0\rar{d_X^0}\& X^1\rar\&X^2\rar\&\cdots\rar\&X^{n+1}
  \end{tikzcd}
  \]
  where the middle row is an admissible $n$-exact sequence.  By the
  \th\ref{comparison-lemma} applied to the compositions $pq$ and $qp$,
  the morphisms $p$ and $q$ induce mutually inverse isomorphisms in
  $\KK(\M)$. By
  \th\ref{n-exact-sequences-are-closed-under-homotopy-equivalence} we
  have that $X$ is an $n$-exact sequence. Then,
  \th\ref{universal-property-of-n-pushout-diagrams-n-exact} gives an
  equivalence between $X$ and an admissible $n$-exact sequence which
  is obtained by $n$-pullback from $Y$ along $f^{n+1}$.  Since $\X$ is
  closed under weak isomorphisms (in particular, equivalences) of
  $n$-exact sequences, we have that $X$ is moreover an admissible
  $n$-exact sequence.
\end{proof}

The following result shows that the class of $n$-exact sequences in an
$n$-exact category is closed under direct summands.

\begin{proposition}
  \th\label{X-is-closed-under-direct-summands} Let $(\M,\X)$ be an
  $n$-exact category, and $X$ and $Y$ complexes in $\C^n(\C)$.  If
  $X\oplus Y$ is an $n$-exact sequence, then so are $X$ and $Y$.
\end{proposition}
\begin{proof}
  This is an adaptation of the proof of
  \cite[Cor. 2.18]{buhler_exact_2010}.  Clearly, $X$ and $Y$ are
  $n$-exact sequences.  In particular, $d_X^0$ admits an $n$-cokernel.
  Moreover, we have a commutative diagram
  \[
  \begin{tikzcd}[ampersand replacement=\&, row sep=large]
    X^0\ar{rr}{d_X^0}\dar[equals]\& \& X^1 \dar{
      \begin{bmatrix}
        1_{X^1}\\ 0
      \end{bmatrix}
    } \\
    X_1\rar[tail]{
      \begin{bmatrix}
        1_{X^0}\\0
      \end{bmatrix}
    } \&X^0\oplus X^1\rar[tail]{d_X^0\oplus d_Y^0} \& X^1\oplus Y^1
  \end{tikzcd}
  \]
  in which the composition along the bottom row yields an admissible
  monomorphism by
  \ref{ax-ex:admissible-monos-closed-under-composition}. 
  We conclude
  from \th\ref{obscure-axiom} that $X$ is an admissible $n$-exact
  sequence.  Analogously one can show that $Y$ is an admissible
  $n$-exact sequence.
\end{proof}

\subsection{$n$-exact categories and cluster-tilting}

Let $(\E,\X)$ be an exact category. Recall that a morphism $X\to Y$ in
$\E$ is \emph{proper} if it has a factorization $X\epi Z\mono Y$
through an $\X$-admissible epimorphism and an $\X$-admissible
monomorphism.  A sequence of proper morphisms
\[
\begin{tikzcd}[row sep=small,column sep=small]
  \cdots\ar{rr}&&X^{k-1}\ar{rr}\drar[two heads]&&X^k\ar{rr}\drar[two heads]&&X^{k+1}\ar{rr}&&\cdots\\
  &&&Y^{k-1}\urar[tail]&&Y^k\urar[tail]
\end{tikzcd}
\]
is an \emph{$\X$-acyclic complex} if for each $k\in\ZZ$ the sequence
$Y^{k-1}\mono X^k\epi Y^k$ is an $\X$-admissible exact sequence. If
the class $\X$ is clear from the complex, then we may write
``acyclic'' instead of ``$\X$-acyclic''. We denote the full
subcategory of $\KK(\E)$ given by the acyclic complexes by $\Ac$.

Following Neeman \cite{neeman_derived_1990}, the derived category
$\D=\DD(\E,\X)$ is by definition the Verdier quotient
$\KK(\E)/\thick(\Ac(\X))$. Then, for all $k\geq1$ and for all $E\in\E$
we can define the functor $\Ext_\X^k(E,-)\colon \E\to \Mod\ZZ$ by
$F\mapsto \Hom_\D(E,F[k])$. Equivalently, one can define
$\Ext_\X^k(E,F)$ as the set of Yoneda equivalence classes of Yoneda
splicings of $\X$-admissible exact sequences in the usual way, see
\cite[Sec. 6]{frerick_exact_2010}.

We use the following variant of the definition of $n$-cluster-tilting
subcategory of an exact category.  Note that in the case of abelian
categories this definition agrees with
\th\ref{def:n-cluster-tilting-abelian}, see
\th\ref{n-ct-abelian-approximations-mono}.

\begin{definition}
  Let $(\E,\X)$ be a small exact category and $\M$ a subcategory of
  $\E$.  We say that $\M$ is an \emph{$n$-cluster-tilting subcategory
    of $(\E,\X)$} if the following conditions are satisfied:
  \begin{enumerate}
  \item Every object $E\in\E$ has a left $\M$-approximation by an
    $\X$-admissible monomorphism $E\mono M$.
  \item Every object $E\in\E$ has a right $\M$-approximation by an
    $\X$-admissible epimorphism $M'\epi E$.
  \item We have
    \begin{align*}
      \M=&\setP{E\in\E}{\text{for all }
        i\in\set{1,\dots,n-1}\ \Ext_\X^i(E,\M)=0}\\
      =&\setP{E\in\E}{\text{for all } i\in\set{1,\dots,n-1}\
        \Ext_\X^i(\M,E)=0}.
    \end{align*}
  \end{enumerate}
  Note that $\E$ itself is the unique {1-cluster-tilting} subcategory
  of $\E$.
\end{definition}

Our aim in this section is to prove the following result, which is
analogous to \th\ref{recognition-thm-abelian} in the case of exact
categories.

\begin{theorem}
  \th\label{recognition-thm-exact} Let $(\E,\X)$ be an exact category
  and $\M$ an $n$-cluster-tilting subcategory of $(\E,\X)$.  Let
  $\Y=\Y(\M,\X)$ be the class of all $\X$-acyclic complexes
  \begin{equation}
    \label{eq:acyclic-n-exact-sequence}
    \begin{tikzcd}
      X^0\rar[tail]&X^1\rar&\cdots\rar&X^n\rar[two heads]&X^{n+1}
    \end{tikzcd}
  \end{equation}
  such that for all $k\in\set{0,1,\dots,n+1}$ we have $X^k\in\M$.
  Then $(\M,\Y)$ is an $n$-exact category.
\end{theorem}

To prove \th\ref{recognition-thm-exact}, we need the following result,
whose proof is completely analogous to the proof of
\th\ref{n-ct_resolution}.

\begin{proposition}
  \th\label{n-ct_resolution-exact} Let $(\E,\X)$ be an exact category
  and $\M$ an $n$-cluster-tilting subcategory of $\E$. Then, for each
  $E\in\E$ there exists an acyclic complex of the form
  \[
  \begin{tikzcd}[column sep=small, row sep=small]
    0\ar{rr}&&E\ar[tail]{rr}{f^0}\drar[equals]&&M^1\ar{rr}{f^1}&&\cdots\ar{rr}{f^{n-2}}\drar[two
    heads,swap]{g^{n-2}}&&M^{n-1}\ar[two heads]{rr}{f^{n-1}}\drar[two heads,swap]{g^{n-1}}&&M^n\ar{rr}&&0\\
    &&&C^1\urar[tail,swap]{h^1}&&&&C^{n-1}\urar[tail,swap]{h^{n-1}}&&C^n\urar[equals]
  \end{tikzcd}
  \]
  satisfying the following properties:
  \begin{enumerate}
  \item For each $k\in\set{1,\dots,n}$ we have $M^k\in\M$.
  \item For each $k\in\set{1,\dots,n-1}$ the morphism $h^k\colon
    C^k\mono M^k$ is an $\X$-admissible monomorphism and a left
    $\M$-approximation of $C^k$.
  \item For each $k\in\set{1,\dots,n-1}$ the pair $C^k\mono M^k\epi
    C^{k+1}$ is an $\X$-admissible exact sequence.
  \end{enumerate}
\end{proposition}

We now give the proof of \th\ref{recognition-thm-exact}.

\begin{proof}[Proof of \th\ref{recognition-thm-exact}]
  If $n=1$, then the result is trivial. Let $n\geq2$.  We observe that
  the class $\Y$ indeed consists of $n$-exact sequences. To see this,
  given $X\in\Y$ and $M\in\M$, apply the functor $\E(M,-)$ to $X$ to
  obtain a sequence
  \[
  \begin{tikzcd}[column sep=small]
    0\rar&\E(M,X^0)\rar&\E(M,X^1)\rar&\cdots\rar&\E(M,X^n)\rar&\E(M,X^{n+1})
  \end{tikzcd}
  \]
  which is exact at $\E(M,X^0)$ and $\E(M,X^1)$. By the version of the
  dual of
  \th\ref{lemma-exts} for exact categories, for each
  $k\in\set{1,\dots,n-1}$ the homology of this complex at
  $\E(M,X^{k+1})$ is isomorphic to $\Ext_\X^k(M,X^0)$ which vanishes
  by assumption. Combining with the dual argument, this shows that $X$
  is an $n$-exact sequence.

  We only show that $(\M,\Y)$ satisfies axioms
  \ref{ax-ex:admissible-monos-closed-under-composition} and
  \ref{ax-ex:n-pushout-exists} in \th\ref{def:n-exact-category}, since
  the remaining axioms are dual.  Note that axiom \ref{ax-ex:zero-seq}
  is obviously satisfied.
  
  \emph{The pair $(\M,\Y)$ satisfies axiom
    \ref{ax-ex:admissible-monos-closed-under-composition}:} Let
  $M\mono N$ be an $\X$-admissible monomorphism such that $M,N\in\M$
  and $N\epi E$ a cokernel of $M\mono N$. Applying
  \th\ref{n-ct_resolution-exact} to $N$ yields an acyclic complex in
  $\Y$ with $M\mono N$ as the leftmost morphism.  This shows that the
  class of $\Y$-admissible monomorphisms coincides with the class of
  $\X$-admissible monomorphisms.  Therefore $(\M,\Y)$ satisfies axiom
  \ref{ax-ex:admissible-monos-closed-under-composition}.  That
  $(\M,\Y)$ satisfies axiom
  \ref{ax-ex:admissible-epis-closed-under-composition} then follows by
  duality.

  \emph{The pair $(\M,\Y)$ satisfies axiom
    \ref{ax-ex:n-pushout-exists}:} Let $X$ be a $\Y$-admissible
  $n$-exact sequence and $f^0\colon X^0\to Y^0$ a morphism in $\M$.
  We need to construct a commutative diagram
  \begin{equation}
    \label{eq:n-pushout-diagram-exact-cat}
    \begin{tikzcd}
      X^0\rar[tail]{d_X^0}\dar{f^0}&X^1\rar{d_X^1}\dar{f^1}&\cdots\rar{d_X^{n-1}}&X^{n}\dar{f^{n}}\rar[two
      heads]{d_X^{n}}&X^{n+1}\\
      Y^0\rar[tail]{d_Y^0}&Y^1\rar{d_Y^1}&\cdots\rar{d_Y^{n-1}}&Y^{n}
    \end{tikzcd}
  \end{equation} 
  such that $d_Y^0$ is a $\Y$-admissible monomorphism or, equivalently
  by the previous paragraph, an $\X$-admissible monomorphism.

  Step 1: We claim that there is a $\Y$-admissible $n$-exact sequence
  \begin{equation}
    \label{eq:in-proof-mapping-cone}
    \begin{tikzcd}
      C:&X^0\rar[tail]&X^1\oplus Y^0\rar&\cdots\rar&X^n\oplus
      Y^{n-1}\rar[two heads]&Y^n
    \end{tikzcd}
  \end{equation}
  with differentials
  \[
  \begin{tikzcd}
    d_C^{k-1}:=\begin{bmatrix}
      -d_X^k&0\\
      f^k&d_Y^{k-1}
    \end{bmatrix}\colon X^k\oplus Y^{k-1}\rar& X^{k+1}\oplus
    Y^k.
  \end{tikzcd}
  \]
  We shall construct $C$ inductively.  Set $Z^0:=X^0$, $u^0:=1_{X^0}$,
  $p^0:=-d_X^0$ and $q^0:=f^0$.  Let $0\leq k\leq n-2$ and suppose
  that we have constructed a commutative diagram
  \[
  \begin{tikzcd}[ampersand replacement=\&]
    X^k\oplus Y^{k-1}\ar{rr}{
      \begin{bmatrix}
        -d_X^k&0\\
        f^k&d_Y^{k-1}
      \end{bmatrix}
    }\drar[two heads,swap]{
      \begin{bmatrix}
        u^k&v^k
      \end{bmatrix}
    }\&\&X^{k+1}\oplus Y^k\\
    \&Z^k\urar[tail,swap]{
      \begin{bmatrix}
        p^k\\q^k
      \end{bmatrix}
    }
  \end{tikzcd}
  \]
  Since $Z^k\mono X^{k+1}\oplus Y^k$ is an $\X$-admissible
  monomorphism, there exist morphisms $u^{k+1}\colon X^{k+1}\to
  Z^{k+1}$ and $v^{k+1}\colon Y^k\to Z^{k+1}$ such that $[u^{k+1}\
  v^{k+1}]\colon X^{k+1}\oplus Y^k\epi Z^{k+1}$.  Therefore the square
  \begin{equation}
    \label{eq:in-proof-pushout-diagram}
    \begin{tikzcd}
      Z^k\rar{p^k}\dar{q^k}&X^{k+1}\dar{u^{k+1}}\\
      Y^k\rar{v^{k+1}}&Z^{k+1}
    \end{tikzcd}
  \end{equation}
  is a pushout diagram.

  It is readily verified that the composition
  \[
  \begin{tikzcd}[ampersand replacement=\&, column sep=huge]
    X^k\oplus Y^{k-1}\rar{
      \begin{bmatrix}
        -d_X^k&0\\
        f^k&d_Y^{k-1}
      \end{bmatrix}
    }\&X^{k+1}\oplus Y^k\rar{
      \begin{bmatrix}
        d_X^{k+1}&0
      \end{bmatrix}
    }\&X^{k+2}
  \end{tikzcd}
  \]
  is zero.  Then, given that $X^k\oplus Y^{k-1}\epi Z^k$ is an
  epimorphism, the composition
  \[
  \begin{tikzcd}[ampersand replacement=\&, column sep=large]
    Z^k\rar[tail]{
      \begin{bmatrix}
        p^k\\q^k
      \end{bmatrix}
    }\&X^{k+1}\oplus Y^k\rar{
      \begin{bmatrix}
        d_X^{k+1}&0
      \end{bmatrix}
    }\&X^{k+2}
  \end{tikzcd}
  \]
  is also zero.  Hence we have $p^kd_X^0=0$.
  
  Since \eqref{eq:in-proof-pushout-diagram} is a pushout diagram,
  there exists a morphism $p^{k+1}\colon Z^{k+1}\to X^{k+2}$ such that
  $u^{k+1}p^{k+1}=-d_X^{k+1}$ and $v^{k+1}p^{k+1}=0$.  Let
  $q^{k+1}\colon Z^{k+1}\mono Y^{k+1}$ be a left $\M$-approximation of
  $Z^{k+1}$.  Given that $-q^{k+1}$ is an $\X$-admissible
  monomorphism, \th\ref{props-of-n-pushout-diagrams} applied to the
  exact category $(\E,\X)$ implies that the morphism $[q^{k+1}\
  p^{k+1}]\colon Z^{k+1}\mono X^{k+2}\oplus Y^{k+1}$ is an
  $\X$-admissible monomorphism.  Set $f^{k+1}:=u^{k+1}q^{k+1}$ and
  $d_Y^k:=v^{k+1}q^{k+1}$.  It is readily verified that the following
  diagram commutes:
  \[
  \begin{tikzcd}[ampersand replacement=\&]
    d_C^{k}:X^{k+1}\oplus Y^k\ar{rr}{
      \begin{bmatrix}
        -d_X^{k+1}&0\\
        f^{k+1}&d_Y^k
      \end{bmatrix}
    }\drar[two heads,swap]{
      \begin{bmatrix}
        u^{k+1}&v^{k+1}
      \end{bmatrix}
    }\&\&X^{k+2}\oplus Y^{k+1}\\
    \&Z^{k+1}\urar[tail,swap]{
      \begin{bmatrix}
        p^{k+1}\\q^{k+1}
      \end{bmatrix}
    }
  \end{tikzcd}
  \]
  Finally, let $d_C^{n-1}\colon X^n\oplus Y^{n-1}\to Y^n$ be a
  cokernel of $Z^{n-1}\mono X^n\oplus Y^{n-1}$. It follows that
  $C\in\Y$, and hence is an $n$-exact sequence.

  Step 2: We claim that the morphism $d_Y^0$ is a $\Y$-admissible
  monomorphism.  Indeed, we have $d_Y^0=v^1q^1$.  Moreover, $v^1$ is
  an $\X$-admissible monomorphism for it is defined by the following
  pushout diagram in the exact category $(\E,\X)$:
  \[
  \begin{tikzcd}
    X^0\rar[tail]{d_X^0}\dar{f^0}&X^{1}\dar{u^1}\\
    Y^0\rar[tail]{v^1}&Z^1
  \end{tikzcd}
  \]
  Therefore $d_Y^0$ is the composition of two $\X$-admissible
  monomorphisms, hence itself is an $\X$-admissible monomorphism.
  Since $\X$-admissible monomorphisms and $\Y$-admissible morphisms
  coincide, the claim follows.  The existence of a commutative diagram
  of the form \eqref{eq:n-pushout-diagram-exact-cat} follows since $C$
  is a complex.  This shows that $(\M,\Y)$ satisfies axiom
  \ref{ax-ex:n-pushout-exists}.  That $(\M,\Y)$ satisfies axiom
  \ref{ax-ex:n-pullback-exists} then follows by duality.
\end{proof}

\section{Frobenius $n$-exact categories}
\label{sec:frobeinus-n-exact-categorties}

In this section we introduce Frobenius $n$-exact categories and show
that their stable categories have the structure of an
$(n+2)$-angulated category; this allows us to introduce algebraic
$(n+2)$-angulated categories. We give a method to construct Frobenius
$n$-exact categories (and thus also algebraic $n$-angulated
categories) from certain $n$-cluster-tilting subcategories of
Frobenius exact categories. We show that our construction is closely
related to the standard construction of $(n+2)$-angulated categories
given in \cite[Thm. 1]{geiss_n-angulated_2013}, see
\th\ref{standard-construction-n-exact}.

\subsection{Reminder on $(n+2)$-angulated categories}
\label{sec:n-angulated-categories}

We follow the exposition of \cite[Sec. 2]{geiss_n-angulated_2013}
although we prefer to set the base case of triangulated categories as
the case $n=1$, in agreement with our convention for abelian and exact
categories. Thus we consider $(n+2)$-angulated categories instead of
`$n$-angulated categories' as done in \loccit.

Let $n$ be a positive integer and $\F$ an additive category equipped
with an automorphism $\Sigma\colon\F\to\F$.  An
\emph{$n$-$\Sigma$-sequence} in $\F$ is a sequence of morphisms
\begin{equation}
  \label{eq:n-angle}
  \begin{tikzcd}
    X^0\rar{\alpha^0}&X^1\rar{\alpha^1}&X^2\rar{\alpha^2}&\cdots\rar{\alpha^n}&X^{n+1}\rar{\alpha^{n+1}}&\Sigma
    X^0.
  \end{tikzcd}
\end{equation}
Its \emph{left rotation} is the $n$-$\Sigma$-sequence
\[
\begin{tikzcd}[column sep=large]
  X^1\rar{\alpha^1}&X^2\rar{\alpha^2}&\cdots\rar{\alpha^n}&X^{n+1}\rar{\alpha^{n+1}}&\Sigma
  X^0\rar{(-1)^n\Sigma\alpha^0}&\Sigma X^1
\end{tikzcd}
\]
A \emph{morphism of $n$-$\Sigma$-sequences} is a commutative diagram
\[
\begin{tikzcd}
  X^0\rar{\alpha^0}\dar{\phi^0}&X^1\rar{\alpha^1}\dar{\phi^1}&X^2\rar{\alpha^2}\dar{\phi^2}&\cdots\rar{\alpha^n}&X^{n+1}\rar{\alpha^{n+1}}\dar{\phi^{n+1}}&\Sigma X^0\dar{\Sigma \phi^0}\\
  Y^0\rar{\beta^0}&Y^1\rar{\beta^1}&Y^2\rar{\beta^2}&\cdots\rar{\beta^n}&Y^{n+1}\rar{\beta^{n+1}}&\Sigma{Y^0}
\end{tikzcd}
\]
where each row is an $n$-$\Sigma$-sequence. The mapping cone $C(\phi)$
of the above morphism is the $n$-$\Sigma$-sequence
\[
\begin{tikzcd}
  X^1\oplus Y^0\rar{\gamma^0}& X^2\oplus Y^1\rar{\gamma^1}& \cdots
  \rar{\gamma^n}& \Sigma X^0\oplus Y^{n+1}\rar{\gamma^{n+1}}& \Sigma
  X^1\oplus \Sigma Y^0
\end{tikzcd}
\]
where for each $k\in\set{0,\dots,n}$ we define
\[
\begin{tikzcd}[ampersand replacement=\&]
  \gamma^k:=
  \begin{bmatrix}
    -\alpha^{k+1}&0\\
    \phi^{k+1}&\beta^k
  \end{bmatrix}\colon X^{k+1}\oplus Y^k\rar\& X^{k+2}\oplus
  Y^{k+1}
\end{tikzcd}
\]
(by convention, $\alpha^{n+2}:=\Sigma\alpha^0$,
$\phi^{n+2}:=\Sigma\phi^0$).  A \emph{weak isomorphism} is a morphism
of $n$-$\Sigma$-sequences such that $\phi^k$ and $\phi^{k+1}$ are
isomorphisms for some $k\in\set{0,1,\dots,n+1}$.  Abusing the
terminology, we say that two $n$-$\Sigma$-sequences are \emph{weakly
  isomorphic} if they are connected by a finite zigzag of weak
isomorphisms.

\begin{definition}
  \th\label{def:n-angulated-category} \cite{geiss_n-angulated_2013} A
  \emph{pre-$(n+2)$-angulated category} is a triple $(\F,\Sigma,\S)$
  where $\F$ is an additive category, $\Sigma\colon\F\to\F$ is an
  automorphism\footnote{ One may consider the more general case when
    $\Sigma\colon\F\to\F$ is an autoequivalence.  As mentioned in
    \cite[2.2 Rmks.]{geiss_n-angulated_2013}, it can be shown that the
    assumption of $\Sigma$ being invertible is but a mild sacrifice,
    \cf \cite[Sec. 2]{keller_sous_1987}.}, and $\S$ is a class of
  $n$-$\Sigma$-sequences (whose members we call $(n+2)$-angles) which
  satisfies the following axioms:
  \begin{enumerate}[label=(F\arabic{*}), ref=(F\arabic{*})]
  \item\label{ax-ang:F1} The class $\S$ is closed under taking direct
    summands and making direct sums.  For all $X\in\F$ the
    \emph{trivial sequence}
    \[
    \begin{tikzcd}
      X\rar{1_X} & X\rar & 0\rar & \cdots \rar & 0 \rar & \Sigma X
    \end{tikzcd}
    \]
    belongs to $\S$.  For each morphism $\alpha$ in $\M$ there exists
    an $n$-angle whose first morphism is $\alpha$.
  \item\label{ax-ang:F2} An $n$-$\Sigma$-sequence of the form
    \eqref{eq:n-angle} is a $(n+2)$-angle if and only if its left
    rotation is a $(n+2)$-angle.
  \item\label{ax-ang:F3} Each commutative diagram
    \[
    \begin{tikzcd}
      X^0\rar{\alpha^0}\dar{\phi^0} & X^1\rar{\alpha^1}\dar{\phi^1} & X^2\rar{\alpha^2}\dar[dotted]{\phi^2}&\cdots\rar{\alpha^n} & X^{n+1} \rar{\alpha^{n+1}}\dar[dotted]{\phi^{n+1}} & \Sigma X^0\dar{\Sigma \phi^0}\\
      Y^0\rar{\beta^0} & Y^1\rar{\beta^1} &
      Y^2\rar{\beta^2}&\cdots\rar{\beta^n} & Y^{n+1} \rar{\beta^{n+1}}
      & \Sigma Y^0
    \end{tikzcd}
    \]
    whose rows are $(n+2)$-angles can be completed to a morphism of
    $n$-$\Sigma$-sequences.
  \end{enumerate}
  We call the class $\S$ a \emph{pre-$(n+2)$-angulation of $\F$ with
    respect to $\Sigma$}.  If moreover the following axiom is
  satisfied, then $(\F,\Sigma,\S)$ is called an \emph{$n$-angulated
    category} and the class $\S$ is called a
  \emph{$(n+2)$-angulation}.
  \begin{enumerate}[label=(F\arabic{*}), ref=(F\arabic{*}), start=4]
  \item\label{ax-ang:F4} In the situation of axiom \ref{ax-ang:F3},
    the morphisms $\phi^2,\dots,\phi^{n+1}$ can be chosen in such a
    way that the mapping cone $C(\phi)$ is a $(n+2)$-angle.
  \end{enumerate}
\end{definition}

An $n$-$\Sigma$-sequence $X$ is \emph{exact} if for all $F\in\F$ and
for all $k\in\ZZ$ the induced sequence of abelian groups
\[
\begin{tikzcd}[column sep=small]
  \F(F,\Sigma^{k-1}X^{n+1})\rar&\F(F,\Sigma^k
  X^0)\rar&\cdots\rar&\F(F,\Sigma^k X^{n+1})\rar&\F(F,\Sigma^{k+1}
  X^0)
\end{tikzcd}
\]
is exact.  We need the following result.

\begin{proposition}
  \th\label{props-pre-n} \cite[Prop. 2.5(c)]{geiss_n-angulated_2013}
  Let $(\F,\Sigma,\S)$ be a pre-$(n+2)$-angulated category. If $\S'$
  is a pre-$(n+2)$-angulation of $\F$ such that $\S'\subseteq \S$, then
  $\S=\S'$.
\end{proposition}

Let $(\F,\Sigma_\F,\S_\F)$ and $(\G,\Sigma_\G,\S_\G)$ be
$(n+2)$-angulated categories. An \emph{exact functor} is a pair
$(F,\eta)$ wehre $F\colon\F\to\G$ is an additive functor and
$\eta\colon F\Sigma_\F\to \Sigma_\G F$ is a natural transformation
such that for every $(n+2)$-angle in $\F$
\[
  \begin{tikzcd}
    X^0\rar{\alpha^0}&X^1\rar{\alpha^1}&X^2\rar{\alpha^2}&\cdots\rar{\alpha^n}&X^{n+1}\rar{\alpha^{n+1}}&\Sigma_\F
    X^0
  \end{tikzcd}
\]
we have that
\[
  \begin{tikzcd}
    FX^0\rar{F\alpha^0}&FX^1\rar{F\alpha^1}&FX^2\rar{F\alpha^2}&\cdots\rar{F\alpha^n}&FX^{n+1}\rar{\beta}&\Sigma_\G
    FX^0
  \end{tikzcd}
\]
is a $(n+2)$-angle in $\G$, where $\beta:=(F\alpha^{n+1})\eta_{X^0}$.

\subsection{Frobenius $n$-exact categories and algebraic
  $(n+2)$-angulated categories}

Our approach in this subsection is analogous to
\cite[Sec. I.2]{happel_triangulated_1988}.

\begin{definition}
  Let $(\M,\X)$ be an $n$-exact category. An object $I\in\M$ is
  $\X$-injective if for every admissible monomorphism $f\colon M\mono
  N$ the sequence of abelian groups
  \[
  \begin{tikzcd}
    \M(N,I)\rar{f\cdot?}&\M(M,I)\rar&0
  \end{tikzcd}
  \]
  is exact.  We say that $(\M,\X)$ \emph{has enough $\X$-injectives}
  if for every object $M\in\M$ there exists $\X$-injective objects
  $I^1,\dots,I^n$ and an admissible $n$-exact sequence
  \[
  \begin{tikzcd}
    M\rar[tail]&I^1\rar&\cdots\rar&I^n\rar[two heads]&N
  \end{tikzcd}
  \]
  The notion of \emph{having enough $\X$-projectives} is defined
  dually.
\end{definition}

\begin{remark}
  If $n=1$, then the notions of $\X$-injectively cogenerated exact
  category and of an exact category with enough $\X$-injectives
  coincide. On the other hand, if $n\geq2$, then there are $n$-exact
  categories which are $\X$-injectively cogenerated but do not have
  enough $\X$-injectives.  Indeed, let $\Lambda$ be a
  finite-dimensional selfinjective algebra. By
  \th\ref{recognition-thm-abelian} we know that an $n$-cluster-tilting
  subcategories $\M$ of $\mod\Lambda$ is an injectively cogenerated
  $n$-abelian category. However, $\M$ has enough injectives if and
  only if $\M$ is stable under taking $n$-th cosyzygies, \cf
  \th\ref{standard-construction-n-exact}.
\end{remark}

Let $(\M,\X)$ be an $n$-exact category. For objects $M,N\in\M$, we
denote by $I(M,N)$ the subgroup of $\M(M,N)$ of morphisms which factor
through an $\X$-injective object. The \emph{$\X$-injectively stable
  category of $\M$}, denoted by $\overline{\M}$, is the category with
the same objects as $\M$ and with morphisms groups defined by
\[
\overline{\M}(M,N) := \M(M,N)/I(M,N).
\]
If $\alpha\colon M\to N$ is a morphism in $\M$, we denote its
equivalence class in $\overline{\M}(M,N)$ by $\overline{\alpha}$.  It
easy to see that $\overline{\M}$ is also an additive category.  The
\emph{$\X$-projectively stable category of $\M$}, denoted by $\sM$, is
defined dually.

\begin{definition}
  We say that an $n$-exact category $(\M,\X)$ is \emph{Frobenius} if
  it has enough $\X$-injectives, enough $\X$-projectives, and if
  $\X$-injective and $\X$-projective objects coincide.  In this case
  one has $\overline{\M}=\sM$, and we refer to this category as the
  \emph{stable category of $\M$}.
\end{definition}

\begin{remark}
  In keeping the convention of the classical theory, if $(\M,\X)$ is a
  Frobenius $n$-exact category, then we denote its stable category by
  $\sM$.
\end{remark}

Our aim is to show that the stable category of a Frobenius $n$-exact
category, has a natural structure of a $(n+2)$-angulated category.  We
begin with the construction of an autoequivalence
$\Sigma\colon\sM\to\sM$.  The following result should be compared with
\cite[Lemma I.2.2]{happel_triangulated_1988}.

\begin{lemma}
  \th\label{shift-well-defined-on-objects} Let $(\M,\X)$ be an
  $n$-exact category.  Suppose that we are given two admissible
  $n$-exact sequences $X$ and $Y$ such that $X^0=Y^0$ and, for
  $k\in\set{1,n}$, the objects $X^k$ and $Y^k$ are $\X$-injective.
  Then, $X^{n+1}$ and $Y^{n+1}$ are isomorphic in $\overline{\M}$.
\end{lemma}
\begin{proof}
  Since $X^1$ and $Y^1$ are $\X$-injective and $X$ and $Y$ are
  admissible $n$-exact sequences, we can construct a commutative
  diagram
  \[
  \begin{tikzcd}
    X\dar{f}&X^0\rar[tail]\dar[equals]&X^1\rar\dar[dotted]&\cdots\rar&X^n\rar[two heads]\dar[dotted]&X^{n+1}\dar[dotted]\\
    Y\dar{g}&Y^0\rar[tail]\dar[equals]&Y^1\rar\dar[dotted]&\cdots\rar&Y^n\rar[two heads]\dar[dotted]&Y^{n+1}\dar[dotted]\\
    X&X^0\rar[tail]&X^1\rar&\cdots\rar&X^n\rar[two heads]&X^{n+1}
  \end{tikzcd}
  \]
  By the \th\ref{comparison-lemma} there exists a morphism $h\colon
  X^{n+1}\to X^n$ such that
  \[
  f^{n+1}g^{n+1}-1=hd_X^n.
  \]
  Since $X^n$ is $\X$-injective, we have
  $\overline{f^{n+1}g^{n+1}}=\overline{1}$.  A similar argument shows
  that $\overline{g^{n+1}f^{n+1}}=\overline{1}$.  Therefore $X^{n+1}$
  and $Y^{n+1}$ are isomorphic in $\overline{\M}$.
\end{proof}

Let $(\M,\X)$ be an $n$-exact category with enough $\X$-injectives.
For each $M\in\M$ we choose an admissible $n$-exact sequence
\[
\begin{tikzcd}
  I(M)\colon M\rar[tail]&I^1(M)\rar&\cdots\rar&I^n(M)\rar[two
  heads]&SM.
\end{tikzcd}
\]
such that for each $k\in\set{1,\dots,n}$ the objects $I^k(M)$ are
$\X$-injective.  It follows from
\th\ref{shift-well-defined-on-objects} that the isomorphism class of
$SM$ in $\sM$ does not depend on the choice of $I(M)$.

Let $f\colon M\to N$ be a morphism in $\M$.  Since $I^1(N)$ is
$\X$-injective, there is a commutative diagram of admissible $n$-exact
sequences
\[
\begin{tikzcd}
  I(M)\dar{I(f)}&M\rar[tail]\dar{f}&I^1(M)\dar{I^1(f)}\rar&\cdots\rar&I^n(M)\rar[two heads]\dar{I^n(f)}&SM\dar{Sf}\\
  I(N)&N\rar[tail]&I^1(N)\rar&\cdots\rar&I^n(N)\rar[two heads]&SY
\end{tikzcd}
\]
The \th\ref{comparison-lemma} implies that $\overline{Sf}$ does not
depend on the choices of $I^1(f),\dots,I^n(f)$.  It is readily
verified that the correspondences $M\mapsto SM$ and
$f\mapsto \overline{Sf}$ define a functor
$\Sigma\colon\overline{\M}\to\overline{\M}$. Note that the
construction of $\Sigma$ involves choices of both $I(M)$ and $I(f)$.

The proof of the following result is straightforward, \cf
\cite[Prop. 2.2]{happel_triangulated_1988} and the remark following. We leave the details to
the reader.

\begin{proposition}
  Let $(\M,\X)$ be a Frobenius $n$-exact category.  Then
  $\Sigma\colon\sM\mapsto\sM$ is an autoequivalence.  Moreover, any
  two choices of assignments $M\mapsto SM$ and $M\mapsto S'M$ yield
  isomorphic functors.
\end{proposition}

\begin{remark}
  Let $(\M,\X)$ be a Frobenius $n$-exact category.  In analogy with
  \cite[Sec. 2]{keller_sous_1987} we assume that
  $\Sigma\colon\sM\to\sM$ is not only an autoequivalence but an
  automorphism of $\sM$ (see also \cite[Rmk. 2.2(d)]{geiss_n-angulated_2013}).
\end{remark}

Let $(\M,\X)$ be a Frobenius $n$-exact category.  We define a class
$\S=\S(\X)$ of $n$-$\Sigma$-sequences in $\sM$ as follows.  Let
$\alpha^0\colon X^0\to X^1$ be a morphism in $\M$.  Then, for every
morphism of $n$-exact sequences of the form
\[
\begin{tikzcd}
  X^0\rar[tail]\dar{\alpha^0}&I^1(X^0)\rar\dar&\cdots\rar&I^n(X^0)\rar[two heads]\dar&SX^0\dar[equals]\\
  X^1\rar[tail]{\alpha^1}&X^2\rar{\alpha^2}&\cdots\rar{\alpha^n}&X^{n+1}\rar[two
  heads]{\alpha^{n+1}}&SX^0
\end{tikzcd}
\]
the sequence
\[
\begin{tikzcd}
  X^0\rar{\overline{\alpha^0}}&X^1\rar{\overline{\alpha^1}}&X^2\rar{\overline{\alpha^2}}&\cdots\rar{\overline{\alpha^n}}&X^{n+1}\rar{\overline{\alpha^{n+1}}}&\Sigma
  X^0
\end{tikzcd}
\]
is called a \emph{standard $(n+2)$-angle}. An $n$-$\Sigma$-sequence
$Y$ in $\sM$ belongs to $\S$ if and only if it is isomorphic to a
standard $(n+2)$-angle.

We need the following result, which shows how admissible $n$-exact
sequences give rise to standard $(n+2)$-angles, \cf \cite[Lemma
I.2.7]{happel_triangulated_1988}.

\begin{lemma}
  \th\label{n-exact-sequence-induce-n+2-angle} Let $(\M,\X)$ be a
  Frobenius $n$-exact category and $X$ an admissible $n$-exact
  sequence in $\M$.  The following statements hold:
  \begin{enumerate}
  \item\label{it:n-exact-sequence-induce-n+2-angle-diagram} There
    exists a commutative diagram
    \begin{equation}
      \label{eq:n-exact-sequence-induce-n+2-angle-diagram}
      \begin{tikzcd}
        X\dar{f}&X^0\rar[tail]\dar[equals]&X^1\rar\dar&\cdots\rar&X^n\rar[two heads]\dar&X^{n+1}\dar\\
        I(X^0)&X^0\rar[tail]&I^1(X^0)\rar&\cdots\rar&I^n(X^0)\rar[two
        heads]&SX^0
      \end{tikzcd}
    \end{equation}
  \item\label{it:n-exact-sequence-induce-n+2-angle} The sequence
    \[
    \begin{tikzcd}[column sep=large]
      X^0\rar{\overline{d^0}} & X^1\rar{\overline{d^1}} & \cdots
      \rar{\overline{d^n}} & X^{n+1} \rar{(-1)^n\overline{f^{n+1}}} &
      \Sigma X^0
    \end{tikzcd}
    \]
    is a standard $(n+2)$-angle.
  \end{enumerate}
\end{lemma}
\begin{proof}
  \eqref{it:n-exact-sequence-induce-n+2-angle-diagram} The existence
  of the required commutative diagram follows from the fact that
  $I^1(M)$ is $\X$-injective and $X$ is an $n$-exact sequence.

  \eqref{it:n-exact-sequence-induce-n+2-angle} The dual of
  \th\ref{props-of-n-pushout-diagrams} implies that the mapping cone
  $C(f)$ of the right most part of
  \eqref{eq:n-exact-sequence-induce-n+2-angle-diagram} is an
  admissible $n$-exact sequence.  For each $k\in\set{1,\dots,n}$ we
  define
  \[
  \begin{tikzcd}[ampersand replacement=\&]
    g^k:=\begin{bmatrix} 0&(-1)^{k-1}
    \end{bmatrix}^\top\colon I^k(X^0)\rar\&X^{k+1}\oplus I^k(X^0).
  \end{tikzcd}
  \]
  It readily follows that the diagram
  \[
  \begin{tikzcd}
    X^0\rar[tail]\dar{\alpha^0}&I^1(X^0)\rar\dar{g^1}&\cdots\rar&I^n(X^0)\rar[two heads]\dar{g^n}&SX^0\dar[equals]\\
    X^1\rar[tail]{d_C^{-1}}&X^2\oplus
    I^1(X^0)\rar{d_C^0}&\cdots\rar{d_C^{n-2}}&X^{n+1}\oplus
    I^n(X^0)\rar[two heads]{(-1)^nd_C^{n-1}}&SX^0
  \end{tikzcd}
  \]
  is commutative, and that it gives rise to the standard $(n+2)$-angle
  \[
  \begin{tikzcd}[column sep=large]
    X^0\rar{\overline{d^0}} & X^1\rar{\overline{d^1}} & \cdots
    \rar{\overline{d^n}} & X^{n+1} \rar{(-1)^n\overline{f^{n+1}}} &
    \Sigma X^0.\qedhere
  \end{tikzcd}
  \]
\end{proof}

The following result is a higher analog of
\cite[Thm. I.2.6]{happel_triangulated_1988}.

\begin{theorem}
  \th\label{n-happel-theorem} Let $(\M,\X)$ a Frobenius $n$-exact
  category.  Then, $(\sM,\Sigma,\S(\X))$ is an $(n+2)$-angulated
  category.
\end{theorem}
\begin{proof}
  We need to show that $(\sM,\Sigma,\S(\X))$ satisfies the axioms of
  $(n+2)$-angulated categories, see \th\ref{def:n-angulated-category}.
  For axioms \ref{ax-ang:F1}, \ref{ax-ang:F2} and \ref{ax-ang:F3}, our
  proof is an adaptation of the proof of
  \cite[Thm. I.2.6]{happel_triangulated_1988}.

  \ref{ax-ang:F1} Firstly, recall that $\X$ is closed under direct
  sums and direct summands, see
  \th\ref{X-is-closed-under-direct-sums,X-is-closed-under-direct-summands}.
  It is easy to show that this implies that the same is true for $\S$.
  Secondly, the diagram
  \[
  \begin{tikzcd}
    M\rar[tail]\dar[equals] & I^1(M)\rar\dar[equals]&\cdots \rar &
    I^n(M) \rar[two heads]\dar[equals] & SM\dar[equals]\\
    X\rar[tail] & I^1(M)\rar&\cdots \rar & I^n(M) \rar[two heads] & SM
  \end{tikzcd}
  \]
  shows that the $n$-$\Sigma$-sequence
  \[
  \begin{tikzcd}
    X\rar{1}&X\rar&0\rar&\cdots\rar&0\rar&\Sigma X.
  \end{tikzcd}
  \]
  is a standard $(n+2)$-angle.  Finally, by the definition of the
  class $\S$, every morphism is the first morphism of some standard
  $(n+2)$-angle.  This shows that $(\sM,\Sigma,\S(\X))$ satisfies
  axiom \ref{ax-ang:F1}.

  \ref{ax-ang:F2} It suffices to consider the case of standard
  $(n+2)$-angles.  Let
  \begin{equation}
    \label{eq:pushout-n-exact-proof}
    \begin{tikzcd}
      I(X^0)\dar{f}&X^0\rar[tail]\dar{\alpha^0} & I^1(X^0)\rar\dar&\cdots \rar & I^n(X^0) \rar[two heads]\dar & SX^0\dar[equals]\\
      X&X^1\rar[tail]{\alpha^1} & X^2\rar{\alpha^2} &\cdots
      \rar{\alpha^n} & X^{n+1} \rar[two heads]{\alpha^{n+1}} & SX^0
    \end{tikzcd}
  \end{equation}
  be a commutative diagram giving rise to the standard $(n+2)$-angle
  \begin{equation}
    \label{eq:standard-angle-proof}
    \begin{tikzcd}
      X^0\rar{\overline{\alpha^0}} & X^1\rar{\overline{\alpha^1}} &
      \cdots \rar{\overline{\alpha^n}} & X^{n+1}
      \rar{\overline{\alpha^{n+1}}} & \Sigma X^0.
    \end{tikzcd}
  \end{equation}
  We need to show that its left rotation is a standard $(n+2)$-angle.
  
  Firstly, by the definition of $\Sigma\colon\sM\to\sM$ we have a
  commutative diagram
  \[
  \begin{tikzcd}
    X^0\rar[tail]\dar{\alpha^0} & I^1(X^0)\rar\dar&\cdots \rar &
    I^n(X^0) \rar[two heads]\dar & SX^0\dar{S(\alpha^0)}\\
    X^1\rar[tail] & I^1(X^1)\rar &\cdots \rar & I^n(X^1) \rar[two
    heads] & SX^1
  \end{tikzcd}
  \]
  Secondly, \th\ref{universal-property-of-n-pushout-diagrams-n-exact}
  yields a commutative diagram
  \[
  \begin{tikzcd}
    I(X^0)\dar{f}&X^0\rar[tail]\dar{\alpha^0}&I^1(X^0)\rar\dar&\cdots\rar&I^n(X^0)\rar[two heads]\dar&SX^0\dar[equals]\\
    X\dar{p}&X^1\rar[tail]{\alpha^1}\dar[equals]&X^2\rar{\alpha^2}\dar&\cdots\rar&X^{n+1}\rar[two heads]{\alpha^{n+1}}\dar& SX^0\dar{S\alpha^0}\\
    I(X^1)&X^1\rar[tail]&I^1(X^1)\rar&\cdots\rar&I^n(X^1)\rar[two
    heads]&SX^1
  \end{tikzcd}
  \]
  By the dual of \th\ref{props-of-n-pushout-diagrams}, the mapping
  cone $C=C(p)$ is an admissible exact sequence. Thirdly, for each
  $k\in\set{1,\dots,n}$ we define
  \[
  \begin{tikzcd}[ampersand replacement=\&]
    g^k:=\begin{bmatrix} 0&(-1)^{k-1}
    \end{bmatrix}^\top\colon I^k(X^1)\rar\&X^{k+2}\oplus I^k(X^1)
  \end{tikzcd}
  \]
  (by convention, $X^{n+1}:=\Sigma X^0$).  It follows that the diagram
  \[
  \begin{tikzcd}
    X^1 \rar[tail]\dar{\alpha^2}&I^1(X^1)\rar\dar{g^1}&\cdots\rar&I^n(X^1)\rar[two heads]\dar{g^n}&SX^1\dar[equals]\\
    X^2\rar[tail]{d_C^{-1}}&X^3\oplus
    I^1(X^1)\rar{d_C^0}&\cdots\rar{d_C^{n-2}}&SX^0\oplus
    I^n(X^1)\rar[two heads]{(-1)^nd_C^{n-1}}&SX^1
  \end{tikzcd}
  \]
  is commutative. Note that the bottom row is an admissible $n$-exact
  sequence for it is isomorphic to $C$.  Finally, the standard
  $(n+2)$-angle induced by this diagram is isomorphic in $\sM$ to the
  left rotation of \eqref{eq:standard-angle-proof}.
  
  Conversely, suppose that there is a commutative diagram
  \[
  \begin{tikzcd}
    I(X^1)\dar{f}&X^1\rar[tail]\dar{\alpha^1}&I^1(X^1)\rar\dar&\cdots\rar&I^n(X^1)\rar[two heads]\dar&SX^1\dar[equals]\\
    Y&X^2\rar[tail]{\alpha^2}&X^3\rar{\alpha^3}&\cdots\rar{\alpha^{n+1}}&SX^0\rar[two
    heads]{(-1)^dS\alpha^0}&SX^1
  \end{tikzcd}
  \]
  which gives rise to a standard $(n+2)$-angle of the form
  \[
  \begin{tikzcd}[column sep=large]
    X^1\rar{\overline{\alpha^1}}&X^2\rar{\overline{\alpha^2}}&\cdots\rar{\overline{\alpha^n}}&X^{n+1}\rar{\overline{\alpha^{n+1}}}&\Sigma
    X^0\rar{(-1)^n\Sigma\overline{\alpha^0}}&\Sigma X^1
  \end{tikzcd}
  \]
  On one hand, the definition of $\Sigma\colon\sM\to\sM$ yields the
  top two rows in the following commutative diagram:
  \[
  \begin{tikzcd}
    I(X^0)\dar{I(\alpha^0)}&X^0\rar[tail]\dar{\alpha^0}&I^1(X^0)\rar\dar&\cdots\rar&I^n(X^0)\rar[two heads]{d^n}\dar & SX^0\dar{S\alpha^0}\\
    I(X^1)\dar{f}&X^1\rar[tail]\dar{\alpha^1}&I^1(X^1)\rar\dar&\cdots\rar&I^n(X^1)\rar[two heads]\dar&SX^1\dar[equals]\\
    Y&X^2\rar[tail]{\alpha^2}&X^3\rar{\alpha^3}&\cdots\rar{\alpha^{n+1}}&SX^0\rar[two
    heads]{(-1)^dS\alpha^0}&SX^1
  \end{tikzcd}
  \]
  On the other hand, we have a commutative diagram
  \[
  \begin{tikzcd}
    I(X^0)\dar{q}&X^0\rar[tail]\dar{0}&\cdots\rar&I^{n-1}(X^0)\rar\dar{0}&I^n(X^0)\rar[two heads]{d^n}\dar{(-1)^nd^n}&SX^0\dar{S\alpha^0}\\
    Y&X^2\rar[tail]{\alpha^2}&\cdots\rar{\alpha^n}&X^{n+1}\rar{\alpha^{n+1}}&SX^0\rar[two
    heads]{(-1)^nS\alpha^0}&SX^1
  \end{tikzcd}
  \]
  Then, the dual of the \th\ref{comparison-lemma} implies the
  existence of a homotopy $h\colon I(\alpha^0)f\to q$.  For each
  $k\in\set{1,\dots,n}$ we define
  \[
  \begin{tikzcd}[ampersand replacement=\&]
    g^k:=\begin{bmatrix} [(-1)^{k}I^k(\alpha^0)\&(-1)^{k-1}h^k]
    \end{bmatrix}^\top\colon I^k(X^0)\rar\& I^k(X^k)\oplus X^{k+1}.
  \end{tikzcd}
  \]
  It is straightforward to verify that the diagram
  \[
  \begin{tikzcd}
    X^0\rar[tail]\dar{\alpha^0}&I^1(X^0)\rar\dar{g^1}&\cdots\rar&I^n(X^0)\rar[two heads]\dar{g^n}&SX^0\dar[equals]\\
    X^1\rar[tail]{d_C^{-1}}& I^1(X^1)\oplus
    X^2\rar{d_C^0}&\cdots\rar{d_C^{n-2}}&I^n(X^1)\oplus
    X^{n+1}\rar[two heads]{d_C^{n-1}}&SX^0
  \end{tikzcd}
  \]
  commutes, where the bottom row is given by $C(f)$. Finally, the
  standard $(n+2)$-angle induced by this diagram is isomorphic to the
  $n$-$\Sigma$-sequence
  \[
  \begin{tikzcd}
    X^0\rar{\overline{\alpha^0}}&X^1\rar{\overline{\alpha^1}}&\cdots\rar{\overline{\alpha^n}}&X^{n+1}\rar{\overline{\alpha^{n+1}}}&\Sigma
    X^0.
  \end{tikzcd}
  \]
  This shows that $(\sM,\X,\S(\X))$ satisfies axiom \ref{ax-ang:F2}.

  \ref{ax-ang:F3} for standard $(n+2)$-angles.  Let
  \[
  \begin{tikzcd}
    I(X^0)\dar{f}&X^0\rar[tail]{d_{IX}^0}\dar{\alpha^0}&I^1(X^0)\rar{d_{IX}^1}\dar&\cdots\rar{d_{IX}^{n-1}}&I^n(X^0)\rar[two heads]{d_{IX}^n}\dar&SX^0\dar[equals]\\
    X&X^1\rar[tail]{\alpha^1}&X^2\rar{\alpha^2}&\cdots\rar{\alpha^n}&X^{n+1}\rar[two
    heads]{\alpha^{n+1}}&SX^0
  \end{tikzcd}
  \]
  and
  \[
  \begin{tikzcd}
    I(Y^0)\dar{g}&Y^0\rar[tail]{d_{IY}^0}\dar{\beta^0}&I^1(Y^0)\rar{d_{IY}^1}\dar{g^1}&\cdots\rar{d_{IY}^{n-1}}&I^n(X^0)\rar[two heads]{d_{IY}^n}\dar&SY^0\dar[equals]\\
    Y&Y^1\rar[tail]{\beta^1}&Y^2\rar{\beta^2}&\cdots\rar{\beta^n}&Y^{n+1}\rar[two
    heads]{\beta^{n+1}}&SY^0
  \end{tikzcd}
  \]
  be pushout diagrams in $\M$.  We set $C:=C(f)$.

  Also, let $\phi^0\colon X^0\to Y^0$ and $\phi^1\colon X^1\to Y^1$ be
  morphisms such that
  $\overline{\phi^0\beta^0}=\overline{\alpha^0\phi^1}$.  Thus, we have
  a diagram
  \begin{equation}
    \label{eq:in-proof-axiom-F3}
    \begin{tikzcd}
      X^0\rar{\overline{\alpha^0}}\dar{\overline{\phi^0}} & X^1\rar{\overline{\alpha^1}}\dar{\overline{\phi^1}} & X^2\rar{\overline{\alpha^2}}&\cdots\rar{\overline{\alpha^n}} & X^{n+1} \rar{\overline{\alpha^{n+1}}} & \Sigma X^0\dar{\Sigma \overline{\phi^0}}\\
      Y^0\rar{\overline{\beta^0}} & Y^1\rar{\overline{\beta^1}} &
      Y^2\rar{\overline{\beta^2}}&\cdots\rar{\overline{\beta^n}} &
      Y^{n+1} \rar{\overline{\beta^{n+1}}} & \Sigma Y^0
    \end{tikzcd}
  \end{equation}
  whose rows are standard $(n+2)$-angles.  Recall that by the
  definition of $\Sigma\colon\sM\to\sM$, there is a commutative
  diagram
  \[
  \begin{tikzcd}
    I(X^0)\dar{I(\phi^0)}&X^0\rar[tail]\dar{\phi^0}&I^1(X^0)\dar\rar&\cdots\rar&I^n(X^0)\rar[two heads]\dar&SX^0\dar{S\phi^0}\\
    I(Y^0)&Y^0\rar[tail]&I^1(Y^0)\rar&\cdots\rar&I^n(Y^0)\rar[two
    heads]&SY^0
  \end{tikzcd}
  \]

  We shall construct a commutative diagram
  \[
  \begin{tikzcd}
    X\dar{\phi}&X^1\rar[tail]{\alpha^1}\dar{\phi^1}&X^2\rar{\alpha^2}\dar[dotted]{\phi^2}&\cdots\rar{\alpha^n}&X^{n+1}\rar[two heads]{\alpha^{n+1}}\dar[dotted]{\phi^{n+1}}&SX^0\dar{S\phi^0}\\
    Y&Y^1\rar[tail]{\beta^1} & Y^2\rar{\beta^2}&\cdots \rar{\beta^n} &
    Y^{n+1} \rar[two heads]{\beta^{n+1}} & SY^0
  \end{tikzcd}
  \]
  together with a homotopy $h\colon f\phi\to I(\phi^0)g$ such that
  $h^{n+1}\colon SX^0\to Y^{n+1}$ is the zero morphism.  Note that
  this gives the required completion of diagram
  \eqref{eq:in-proof-axiom-F3}.
  
  We begin with the construction of $h^1$ and $\phi^2$.  Since
  $\overline{\phi^0\beta^0}=\overline{\alpha^0\phi^1}$, there exists
  an $\X$-injective object $I\in\M$ and morphisms $u\colon X^0\to I$
  and $v\colon I\to Y^1$ such that $\alpha^0\phi^1-\phi^0\beta^0=uv$.
  Then, given that $d_{IX}^0$ is an admissible monomorphism and $I$ is
  $\X$-injective, we can construct a commutative diagram
  \[
  \begin{tikzcd}
    X^0\rar[tail]{d_{IX}^0}\drar{u}& I^1(X^0)\dar[dotted]\drar[dotted]{h^1}\\
    &I\rar{v}&Y^1
  \end{tikzcd}
  \]
  Hence $\alpha^0\phi^1-\phi^0\beta^0=d_{IX}^0h^1$, as required.  Then
  we have
  \begin{align*}
    d_{IX}^0(I(\phi^0)^1g^1+h^1\beta^1)
    =&\phi^0d_{IY}^0g^1+(\alpha^0\phi^1-\phi^0g^1)\beta^1\\
    =&\alpha^0\phi^1\beta^1.
  \end{align*}
  Since $d_C^{0}$ is a weak cokernel of $d_C^{-1}$, there exists
  morphisms $\phi^2\colon X^2\to Y^2$ and $h^2\colon I^2(X^0)\to Y^2$
  such that $\phi^1\beta^1=\alpha^1\phi^2$ and
  $f^1\phi^2-(I(\phi^0)^1g^1+h^1\beta^1)=d_{IX}^2h^2$ or,
  equivalently,
  \[
  f^1\phi^2-I(\phi^0)^1g^1=h^1\beta^1+d_{IX}^1h^2
  \]

  Let $2\leq k\leq n$ and suppose that for each $\ell\leq k$ we have
  constructed morphisms $\phi^\ell\colon X^\ell\to Y^\ell$ and
  $h^\ell\colon I^\ell(X^0)\to Y^\ell$ such that
  $\alpha^{\ell-1}\phi^\ell=\phi^{\ell-1}\beta^{\ell-1}$ and
  \[
  f^{\ell-1}\phi^\ell-I(\phi^0)^{\ell-1}
  g^{\ell-1}=h^{\ell-1}\beta^{\ell-1}+d_{IX}^{\ell-1}h^\ell.
  \]
  Then we have
  \begin{align*}
    d_{IX}^{k-1}(I(\phi^0)^kg^k+h^k\beta^k)
    =&I(\phi^0)^{k-1}d_{IY}^{k-1}g^k+(f^{k-1}\phi^k-I(\phi^0)^{k-1}g^{k-1}-h^{k-1}\beta^{k-1})\beta^k\\
    =&f^{k-1}\phi^k\beta^k.
  \end{align*}
  Moreover, $\alpha^{k-1}(\phi^k\beta^k)=$ Since $d_C^{k-1}$ is a weak
  cokernel of $d_C^{k-2}$, there exists morphisms $\phi^{k+1}\colon
  X^{k+1}\to Y^{k+1}$ and $h^{k+1}\colon I^{k+1}(X^0)\to Y^{k+1}$ such
  that $\alpha^k\phi^{k+1}=\phi^k\beta^k$ and
  \[
  f^k\phi^{k+1}-I(\phi^0)^k g^k=h^k\beta^k+d_{IX}^kh^{k+1}.
  \]
  This finishes the induction step.

  It remains to show that $\alpha^{n+1}S\phi^0=\phi^{n+1}\beta^{n+1}$.
  Indeed, we have $\alpha^n(\alpha^{n+1}S\phi^0)=0$ and
  $\alpha^n(\phi^{n+1}\beta^{n+1})=\phi^n(\beta^n\beta^{n+1})=0$.
  Moreover,
  \[
  f^{n+1}(\phi^{n+1}\beta^{n+1})
  =(I^n(\phi^0)g^n-h^n\beta^n)\beta^{n+1} =d_{IX}^nS(\phi^0)
  =f^{n+1}(\alpha^{n+1}S\phi^0).
  \]
  Since $d_C^{n-1}$ is a cokernel of $d_C^{n-2}$, we have
  $\alpha^{n+1}S(\phi^0)=\phi^{n+1}\beta^{n+1}$, as required.
  
  This shows that $(\sM,\Sigma,\S(\X))$ satisfies axiom
  \ref{ax-ang:F3} in the case of standard $(n+2)$-angles. The general
  case is left to the reader.
  
  \ref{ax-ang:F4} for standard $(n+2)$-angles.  We shall show that the
  mapping cone of the morphism of standard $(n+2)$-angles that we
  constructed in the proof of axiom \ref{ax-ang:F3} is a
  $(n+2)$-angle.  We keep the notation and morphisms of the previous
  paragraphs.
  
  For each $k\in\set{0,1\dots,n-1}$ we define
  \[
  \begin{tikzcd}
    r^k:=[\phi^{k+1}\ g^k]\colon X^{k+1}\oplus I^k(Y^0)\rar&Y^{k+1}.
  \end{tikzcd}
  \]
  Also, we define
  \[
  \begin{tikzcd}[ampersand replacement=\&]
    r^n:=
    \begin{bmatrix}
      \alpha^{n+1}&0\\
      \phi^{n+1}&g^n
    \end{bmatrix}
    \colon X^{n+1}\rar\&I^n(Y^0)\to Y^{n+1}
  \end{tikzcd}
  \]
  and
  \[
  \begin{tikzcd}[ampersand replacement=\&]
    r^{n+1}:=
    \begin{bmatrix}
      1_{SX^0}&0\\
      S\phi^0&1_{SY^0}
    \end{bmatrix}\colon SX^0\oplus SY^0\rar\&SX^0\oplus SY^0.
  \end{tikzcd}
  \]
  It is straightforward to check that this defines a morphism of
  admissible $n$-exact sequences
  \[
  \begin{tikzcd}
    r\colon X\oplus I(Y^0)\rar&T(SX^0,0)\oplus Y.
  \end{tikzcd}
  \]
  (recall that the direct sum of two admissible $n$-exact sequences is
  again an admissible $n$-exact sequence, see
  \th\ref{X-is-closed-under-direct-sums}).  Since $r^{n+1}$ is an
  isomorphism, we have that the mapping cone $C(r)$ is an admissible
  $n$-exact sequence, see \th\ref{props-of-n-pushout-diagrams}.

  Next, by the definition of $\Sigma\colon\sM\to\sM$ there is a
  commutative diagram
  \[
  \begin{tikzcd}
    I(X^0)\dar{I(\alpha^0)}&X^0\rar[tail]\dar{\alpha^0}&I^1(X^0)\dar\rar&\cdots\rar&I^n(X^0)\rar[two heads]\dar&SX^0\dar{S\alpha^0}\\
    I(X^1)&X^1\rar[tail]&I^1(X^1)\rar&\cdots\rar&I^n(X^1)\rar[two
    heads]&SX^1
  \end{tikzcd}
  \]
  Then, by \th\ref{universal-property-of-n-pushout-diagrams-n-exact}
  and the dual of \th\ref{props-of-n-pushout-diagrams} there exists a
  commutative diagram
  \[
  \begin{tikzcd}
    X\dar{s}&X^1\rar[tail]\dar[equals]&X^2\rar\dar&\cdots\rar&X^{n+1}\rar[two heads]\dar&SX^0\dar{S\alpha^0}\\
    I(X^1)&X^1\rar&I^1(X^1)\rar&\cdots\rar&I^n(X^1)\rar&SX^1
  \end{tikzcd}
  \]
  Let $t:=1_{X^1}\oplus 1_{Y^0}$.  For each $k\in\set{1,\dots,n}$ we
  define
  \[
  \begin{tikzcd}[ampersand replacement=\&]
    t^k:=
    \begin{bmatrix}
      s^k&0&0\\
      0&1_{I^k(Y^0)}&0
    \end{bmatrix}
    \colon X^{k+1}\oplus I^k(Y^0)\oplus Y^k\rar\&I^k(X^1)\oplus
    I^1(Y^0)
  \end{tikzcd}
  \]
  and
  \[
  \begin{tikzcd}[ampersand replacement=\&]
    t^{n+1}:=(-1)^n\begin{bmatrix}
      S\alpha^0&0\\
      -S\phi^0&\beta^{n+1}
    \end{bmatrix}
    \colon SX^0\oplus Y^{n+1}\rar\&SX^1\oplus SY^0.
  \end{tikzcd}
  \]
  It is readily verified that these morphisms define a morphism of
  admissible $n$-exact sequences
  \[
  \begin{tikzcd}
    t\colon C(r)\rar& I(X^1)\oplus I(Y^0).
  \end{tikzcd}
  \]
  Finally, applying \th\ref{n-exact-sequence-induce-n+2-angle} to the
  morphism $t$ yields that the sequence
  \[
  \begin{tikzcd}
    X^1\oplus Y^0\rar{\gamma^0}& X^2\oplus Y^1\rar{\gamma^1}& \cdots
    \rar{\gamma^n}& \Sigma X^0\oplus Y^{n+1}\rar{\gamma^{n+1}}& \Sigma
    X^1\oplus \Sigma Y^0
  \end{tikzcd}
  \]
  where for each $k\in\set{0,\dots,n}$ we have
  \[
  \begin{tikzcd}[ampersand replacement=\&]
    \gamma^k=
    \begin{bmatrix}
      -\overline{\alpha}^{k+1}&0\\
      \overline{\phi}^{k+1}&\overline{\beta}^k
    \end{bmatrix}\colon X^{k+1}\oplus Y^k\rar\& X^{k+2}\oplus
    Y^{k+1}
  \end{tikzcd}
  \]
  is a standard $(n+2)$-angle.

  This shows that $(\sM,\Sigma,\S(\X))$ satisfies axiom
  \ref{ax-ang:F4} in the case of standard $(n+2)$-angles. The general
  case is left to the reader.
\end{proof}

\th\ref{n-happel-theorem} allows us to define the following class of
$(n+2)$-angulated categories.

\begin{definition}
  We say that a $(n+2)$-angulated category $(\F,\Sigma_\F,\S)$ is
  \emph{algebraic} if there exists a Frobenius $n$-exact category
  $(\M,\X)$ together with an equivalence of $(n+2)$-angulated
  categories between $(\sM,\Sigma_\M,\S(\X))$ and $(\F,\Sigma_\F,\S)$.
\end{definition}

\subsection{Standard construction}

We remind the reader of the definition of an $n$-cluster-tilting
subcategory of a triangulated category.

\begin{definition}
  \th\label{def:n-cluster-tilting-tri} Let $(\T,\Sigma,\S)$ be a
  triangulated category and $\M$ a subcategory of $\T$.  We say that
  $\M$ is an \emph{$n$-cluster-tilting subcategory of $\T$} if $\M$ is
  functorially finite (see subsection \ref{sec:conventions}) in $\T$
  and
  \begin{align*}
    \M=&\setP{X\in\T}{\forall i\in\set{1,\dots,n-1}\
      \Ext_\T^i(X,\M)=0}\\
    =&\setP{X\in\T}{\forall i\in\set{1,\dots,n-1}\ \Ext_\T^i(\M,X)=0}.
  \end{align*}
\end{definition}

In \cite[Sec. 4]{geiss_n-angulated_2013}, Gei\ss-Keller-Oppermann give
a standard construction of $(n+2)$-angulated categories from
$n$-cluster-tilting categories of a triangulated category which are
closed under the $n$-th power of the suspension functor.  More
precisely, they prove the following theorem.

\begin{theorem}
  \th\label{standard-construction-n-angulated}
  \cite[Thm. 1]{geiss_n-angulated_2013} Let $(\T,\Sigma_3,\S)$ be a
  triangulated category with an $n$-cluster-tilting subcategory $\C$
  such that $\Sigma_3^n(\C)\subseteq\C$. Then,
  $(\C,\Sigma_3^n,\S(\C))$ is an $(n+2)$-angulated category where
  $\S(\C)$ is the class of all sequences
  \[
  \begin{tikzcd}
    X^0\rar{\alpha^0}&X^1\rar{\alpha^1}&X^2\rar{\alpha^2}&\cdots\rar{\alpha^n}&X^{n+1}\rar{\alpha^{n+1}}&\Sigma
    X^0.
  \end{tikzcd}
  \]
  such that there exists a diagram
  \[
  \begin{tikzcd}[column sep=tiny, row sep=small]
    {}&X^1\drar\ar{rr}{\alpha^1}&&X^2\drar&&\cdots&&X^n\drar{\alpha^n}\\
    X^0\urar{\alpha^0}&&X^{1.5}\ar{ll}[description]{|}\urar&&X^{2.5}\ar{ll}[description]{|}&\cdots&X^{n-1.5}\urar&&X^{n+1}\ar{ll}[description]{|}
  \end{tikzcd}
  \]
  with $X^k\in\C$ for all $k\in\ZZ$ such that all oriented triangles
  are triangles in $\T$, all non-oriented triangles commute, and
  $\alpha^{n+1}$ is the composition along the lower edge of the
  diagram.
\end{theorem}

Our aim is to give an analogous construction for Frobenius $n$-exact
categories. For this, we need some terminology.

Let $(\E,\X)$ be a Frobenius exact category and $E\in\E$. An
\emph{$n$-th cosyzygy $\mho^n(E)$ of $E$} is defined by an acyclic
complex
\[
\begin{tikzcd}
  E\rar[tail]&I^1\rar&\cdots\rar&I^n\rar[two heads]&\mho^n(E)
\end{tikzcd}
\]
where for all $k\in\set{1,\dots,n}$ the object $I^k$ is
$\X$-injective.

\begin{proposition}
  \th\label{cosyzygy-independent} Let $(\E,\X)$ be an exact category,
  $\M$ an $n$-cluster-tilting subcategory of $\E$ and $M\in\M$. If an
  $n$-th cosyzygy $\mho^n(M)$ of $M$ satisfies $\mho^n(M)\in\M$, then
  so does any other $n$-th cosyzygy $\tilde{\mho}^n(M)$ of $M$.
\end{proposition}
\begin{proof}
  Note that for all $k\in\set{1,\dots,n-1}$ we have
  \[
  \Ext_\X^k(-,\mho^n(E))\cong\Ext_\X^{k+n}(-,E)\cong\Ext_\X^k(-,\tilde{\mho}^n(E)).
  \]
  Then, it follows from the definition of $n$-cluster-tilting
  subcategory that $\mho^n(E)\in\M$ if and only if
  $\tilde{\mho}^n(E)\in\M$.
\end{proof}

Let $(\E,\X)$ be a Frobenius exact category and for each $E\in\E$ fix
a choice of $n$-th cosyzygy $\mho^n(E)$ of $E$. This defines a map on
objects $\mho^n\colon \Obj(\E)\to \Obj(\E)$. Note that if $\M$ is an
$n$-cluster-tilting subcategory of $\E$, then
\th\ref{cosyzygy-independent} shows that the condition
$\mho^n(\M)\subseteq\M$ is independent of the choice of $\mho^n$.  We
have the following result, which is closely related to
\th\ref{standard-construction-n-angulated}.

\begin{theorem}
  \th\label{standard-construction-n-exact} Let $(\E,\X)$ be a
  Frobenius exact category with an $n$-cluster-tilting subcategory
  $\M$ such that $\mho^n(\M)\subseteq\M$, and let $(\M,\Y)$ be the
  $n$-exact structure on $\M$ given in \th\ref{recognition-thm-exact}.
  Then, the following statements hold:
  \begin{enumerate}
  \item\label{it:is-frobenius} The pair $(\M,\Y)$ is a Frobenius
    $n$-exact category.
  \item\label{it:sM-is-CT} Let $(\sE,\Sigma_{\sE},\S(\X))$ be the
    standard triangulated structure of $\sE$. Then, $\sM$ is an
    $n$-cluster-tilting subcategory of $\sE$.
  \item\label{it:structures-coincide} Let $(\sM,\Sigma,\S(\sM))$ be
    the standard $(n+2)$-angulated structure of $\sM$. Then, we have
    an equivalence of $(n+2)$-angulated categories between
    $(\sM,\Sigma,\S(\sM))$ and $(\sM,\Sigma_{\sE}^n,\S(\sM))$.
  \end{enumerate}
\end{theorem}
\begin{proof}
  \eqref{it:is-frobenius} By \th\ref{recognition-thm-exact} we have
  that $(\M,\Y)$ is an $n$-exact category; thus we only need to show
  that it is Frobenius. Indeed, note that the definition of
  $n$-cluster-tilting subcategory implies that $\M$ contains all
  $\X$-injective objects. Moreover, since $\X$-admissible
  monomorphisms with terms in $\M$ are precisely the $\Y$-admissible
  monomorphisms, all $\X$-injectives are also
  $\Y$-injectives. Finally, the condition $\mho(\M)\subseteq\M$
  implies that $(\M,\Y)$ has enough $\Y$-injectives. By duality,
  $(\M,\Y)$ has enough $\Y$-projectives (and they are the
  $\X$-projectives). Since $\X$-projectives and $\X$-projectives
  coincide, this shows that $(\M,\Y)$ is a Frobenius $n$-exact
  category.

  \eqref{it:sM-is-CT} This statement follows readily from the
  definitions.

  \eqref{it:structures-coincide} For simplicity, we assume that
  $\Sigma_{\sM}$ and $\Sigma_{\sE}^n$ are equal and not only
  isomorphic (note that $\Sigma_{\sM}$ is induced by the $n$-th power
  of the cosyzygy in $\E$).  By \th\ref{props-pre-n} it is enough to
  show that $\S(\sM)\subseteq\S(\Y)$. For this, recall that a standard
  triangle
  $A\xto{\overline{u}} B\xto{\overline{v}}
  C\xto{\overline{w}}\Sigma_{\sE}A$
  in $\S(\X)$ is given by a morphism of admissible $\X$-exact
  sequences
  \[
  \begin{tikzcd}
    A\rar[tail]{u'}\dar{u}&I(A)\rar[two heads]{u}\dar{v'}&SA\dar[equals]\\
    B\rar{v}&C\rar[two heads]{w}&SA
  \end{tikzcd}
  \]
  where $I(A)$ is an $\X$-injective object. By
  \th\ref{props-of-n-pushout-diagrams} this gives rise to an
  $\X$-admissible exact sequence
  \[
  \begin{tikzcd}[column sep=large,ampersand replacement=\&]
    A\rar[tail]{
      \begin{bmatrix}
	u'\\u
      \end{bmatrix}}\&I(A)\oplus B\rar[two heads]{
      \begin{bmatrix}
	v'&v
      \end{bmatrix}}\&C.
  \end{tikzcd}
  \]
  Consider a $(n+2)$-angle $X\in\S(\sM)$
  \[
  X\colon\begin{tikzcd}[column sep=tiny, row sep=small]
    {}&X^1\drar\ar{rr}{\overline{\alpha^1}}&&X^2\drar&&\cdots&&X^n\drar{\overline{\alpha^n}}\\
    X^0\urar{\overline{\alpha^0}}&&X^{1.5}\ar{ll}[description]{|}\urar&&X^{2.5}\ar{ll}[description]{|}&\cdots&X^{n-1.5}\urar&&X^{n+1}\ar{ll}[description]{|}
  \end{tikzcd}
\]
such that each of the involved triangles is a standard triangle in
$(\sE,\S(\X))$ with $\overline{\alpha^{n+1}}$ the composition along
the lower edge of the diagram. By gluing the associated
$\X$-admissible exact sequences associated to each of these triangles
we obtain a $\Y$-admissible $n$-exact sequence
  \[
  \begin{tikzcd}[column sep=small, row sep=small]
    X^0\ar[tail]{r}&I(X^0)\oplus
    X^1\ar{r}&\cdots\ar{r}&I(X^{n-1.5})\oplus X^n\ar[two
    heads]{r}&X^{n+1}.
  \end{tikzcd}
  \]
  \th\ref{n-exact-sequence-induce-n+2-angle} implies that this
  $n$-exact sequence induces a standard $(n+2)$-angle $X'\in\S(\Y)$
  \[
  X'\colon\begin{tikzcd}
    X^0\rar{\overline{\alpha^0}}&X^1\rar{\overline{\alpha^1}}&\cdots\rar{\overline{\alpha^n}}&X^{n+1}\rar&\Sigma_{\sM}
    X^0.
  \end{tikzcd}
  \]
  Finally, a straightforward verification shows that one can take
  $X^{n+1}\to\Sigma_{\sM}X^0$ in $X'$ equal to
  $\overline{\alpha^{n+1}}$ showing that $X=X'$ and the result
  follows. For example, for $n=2$ the last claim follows from
  \th\ref{n-exact-sequence-induce-n+2-angle} and the existence of a
  commutative diagram
  \[
  \begin{tikzcd}[column sep=small, row sep=small]
    X^0\ar[tail]{rr}\ar[equals]{dd}&&I(X^0)\oplus X^1\ar{rr}\ar[two heads]{dr}\ar[equals]{dd}&&I(X^{1.5})\oplus X^2\ar[two heads]{rr}\ar{dd}&&X^3\ar{dd}{-\beta^{3}}\\
    &&&X^{1.5}\ar[tail]{ur}\\
    X^0\ar[tail]{rr}\ar[equals]{dd}&&I(X^0)\oplus X^1\ar{rr}\ar{dd}\ar[two heads]{dr}&&I(X^{1.5})\ar[two heads]{rr}\ar{dd}&&\Sigma X^{1.5}\ar{dd}{-\Sigma_{\sM}\beta^{1.5}}\\
    &&&X^{1.5}\ar[tail]{ur}\ar[crossing over, equals]{uu}\\
    X^0\ar[tail]{rr}&&I(X^0)\ar{rr}\drar[two heads]&&I(\Sigma_{\sM} X^{0})\ar[two heads]{rr}&&\Sigma_{\sM}^2X^0\\
    &&&\Sigma_{\sM}X^0\ar[tail]{ur}\ar[leftarrow, crossing
    over]{uu}[swap, near end]{-\beta^{1.5}}
  \end{tikzcd}
  \]
  The diagram needed for the general case can be easily inferred from
  the diagram above.
\end{proof}

\section{Examples}
\label{sec:examples}

We conclude this article with a collection of examples of $n$-abelian,
$n$-exact categories and algebraic $(n+2)$-angulated categories. Most
of the examples we present are known, and all of them arise as
$n$-cluster-tilting subcategories in different contexts. Our main
tools in this section are
\th\ref{recognition-thm-abelian,recognition-thm-exact,standard-construction-n-exact}.
In the remainder, $K$ denotes an algebraically closed field and all
algebras are finite dimensional over $K$.

\subsection{$n$-representation finite algebras}

The class of $n$-representation finite algebras was introduced by
Iyama-Oppermann in \cite{iyama_n-representation-finite_2011} in the
context of higher Auslander-Reiten theory as higher analogs of
representation-finite algebras.

\begin{definition}
  \cite[Def. 2.2]{iyama_n-representation-finite_2011} Let $\Lambda$ be
  a finite dimensional algebra over a field $K$.
  \begin{enumerate}
  \item A $\Lambda$-module $M\in\mod \Lambda$ is an
    \emph{$n$-cluster-tilting module} if $\add M$ is an $n$-cluster-tilting
    subcategory of $\mod\Lambda$. Note that
    \th\ref{recognition-thm-abelian} implies that $\add M$ is an
    $n$-abelian category.
  \item We say that $\Lambda$ is \emph{$n$-representation-finite} if
    $\gldim\Lambda=n$ and there exists an $n$-cluster-tilting
    $\Lambda$-module.
  \end{enumerate}
\end{definition}

The following result, observed jointly with Martin Herschend, gives examples of $n$-abelian categories for
every positive integer $n$.

\begin{proposition}
  Let $n\geq1$ and $m\geq0$. Also, let $\vec{A}_{nm+1}$ be the linearly oriented
  quiver of Dynkin type $A$ with $nm+1$ vertices, $J$ be the
  Jacobson radical of the path algebra $K\vec{A}_{nm+1}$, and
  $\Lambda:=K\vec{A}_{nm+1}/J^2$. Then, the following statements hold:
  \begin{enumerate}
  \item There exists a unique basic $n$-cluster-tilting
    $\Lambda$-module $M$.
  \item The category $\add M\subseteq\mod\Lambda$ is $n$-abelian.
  \end{enumerate}
\end{proposition}
\begin{proof}
  We assume that $\vec{A}_{nm+1}$ has vertices $0,1,\dots,nm$. Let $S_i$
  be the simple module concentrated at the vertex $i$, and $P_i$ the
  indecomposable projective $\Lambda$-module with top $S_i$. The
  Auslander-Reiten quiver of $\mod\Lambda$ is given by 
  \begin{center}
	\includegraphics[width=0.75\textwidth]{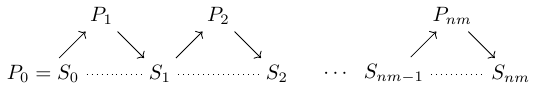}
  \end{center}
  It is straightforward to verify that the module
  \[
  M:=\Lambda\oplus S_{n}\oplus
  S_{2n}\oplus\cdots\oplus S_{(m-1)n}\oplus S_{nm}
  \]
  is the unique basic $n$-cluster-tilting $\Lambda$-module.
\end{proof}

Let $\Lambda$ be a representation-finite algebra, \ie such that the
set of isomorphism classes of indecomposable $\Lambda$-modules is
finite.  We remind the reader that the \emph{Auslander algebra
  associated to $\Lambda$} is the endomorphism algebra of a basic
$\Lambda$-module $M$ such that $\add M=\mod\Lambda$.

\begin{example}
  Typical examples of 2-representation-finite algebras, hence sources
  of 2-abelian categories, are the Auslander algebras associated to
  $KQ$ where $Q$ is a Dynkin quiver $\vec{A}_m$, see
  \cite[Sec. 9.2]{herschend_selfinjective_2011}.
\end{example}

\subsection{$n$-representation infinite algebras}

The class of $n$-representation-infinite algebras was introduced by
Herschend-Iyama-Oppermann in \cite{herschend_n-representation_2014} as
a higher analog of representation-infinite hereditary algebras from
the viewpoint of higher Auslander-Reiten theory. This class of
algebras complements that of $n$-representation-finite algebras.

Let $\Lambda$ be a finite dimensional algebra with finite global
dimension. Then, $\Db(\mod\Lambda)$, the bounded derived category of
$\mod\Lambda$, has a Serre functor
\[
\begin{tikzcd}
  \nu:=-\otimes_\Lambda^L
  D\Lambda\colon\Db(\mod\Lambda)\rar&\Db(\mod\Lambda).
\end{tikzcd}
\]
We define $\nu_n:=\nu[-n]$.

\begin{definition}
  \cite[Def. 2.7]{herschend_n-representation_2014} Let $\Lambda$ be a
  finite dimensional algebra such that $\gldim\Lambda=n$. We say that
  $\Lambda$ is \emph{$n$-representation-infinite} if for all $i\geq0$
  we have $\nu_n^{-i}(\Lambda)\in\mod\Lambda$.
\end{definition}
Let $\T$ be a triangulated category. Following
\cite{beilinson_faisceaux_1982}, a \emph{$t$-structure on $\T$} is a
pair $(\T^{\leq0},\T^{\geq0})$ of strictly full (\ie full and closed
under isomorphisms) subcategories of $\T$ which satisfies the
following properties:
\begin{enumerate}
\item We have $\Sigma\T^{\leq0}\subseteq\T^{\leq0}$ and
  $\Sigma^{-1}\T^{\geq0}\subseteq\T^{\geq0}$.
\item For all $X\in\T^{\leq0}$ and for all $Y\in\T^{\geq0}$ we have
  $\T(X,\Sigma^{-1}Y)=0$.
\item For each $X\in\T$ there exists a triangle $X'\to X\to X''\to
  \Sigma X'$ with $X'\in\T^{\leq0}$ and $X''\in\Sigma^{-1}\T^{\geq0}$.
\end{enumerate}
The \emph{heart} of $(\T^{\leq0},\T^{\geq0})$ is by definition
$\T^{\leq0}\cap\T^{\geq0}$. The heart of a $t$-structure is always an
abelian category \cite{beilinson_faisceaux_1982}.

Note that $\Db(\mod\Lambda)$ has a standard $t$-structure
$(\D^{\leq0},\D^{\geq0})$ defined by
\[
\D^{\leq0}:=\setP{X\in\Db(\mod\Lambda)}{\text{for all } i>0\ H^i(X)=0},
\]
\[
\D^{\geq0}:=\setP{X\in\Db(\mod\Lambda)}{\text{for all } i<0\ H^i(X)=0}.
\]
The heart of $(\D^{\leq0},\D^{\geq0})$ is precisely $\mod\Lambda$.

\begin{theorem}
  \th\label{non-standard_t-structure_n-RI}
  \cite[Thm. 3.7]{minamoto_ampleness_2012} Let $\Lambda$ be an
  $n$-representation infinite algebra such that $\Pi_{n+1}(\Lambda)$
  is noetherian. Let $(\D^{\leq0},\D^{\geq0})$ be the standard
  $t$-structure of $\Db(\mod\Lambda)$ and define
  \[
  \X^{\leq0}:=\setP{X\in\Db(\mod\Lambda)}{\nu_n^{-i}(X)\in \D^{\leq0}\,
    \forall i\gg0}
  \]
  \[
  \X^{\geq0}:=\setP{X\in\Db(\mod\Lambda)}{\nu_n^{-i}(X)\in \D^{\geq0}\
    \forall i\gg0}.
  \]
  Then, the pair $(\X^{\leq0},\X^{\geq0})$ is a $t$-structure in
  $\Db(\mod\Lambda)$. Moreover, the heart of this $t$-structure is
  equivalent to the non-commutative projective scheme
  $\qgr\Pi_{n+1}(\Lambda)$, see \cite{artin_noncommutative_1994} for
  the definition.
\end{theorem}

The following result gives examples of $n$-exact categories.

\begin{theorem}
  \th\label{n-exact-cat-U} \cite[Thm. 2.11]{herschend_representation_2014} Let $\Lambda$ be
  an $n$-representation infinite algebra such that
  $\Pi_{n+1}(\Lambda)$ is noetherian. Let $(\X^{\leq0},\X^{\geq0})$ be
  the $t$-structure defined in \th\ref{non-standard_t-structure_n-RI}
  and $\H$ be its heart.  Then, the following statements hold:
  \begin{enumerate}
  \item The category \[
    \E:=\setP{X\in\H}{\nu_n^i(X)\in(\mod\Lambda)[-n]\ \forall i\gg0}.
    \]
    is an extension closed subcategory of $\H$. 
  \item The category
    \[
    \U:=\add\setP{\nu_n^{-i}(\Lambda)}{i\in\ZZ}
    \]
    is an $n$-cluster-tilting subcategory of $\E$.
  \end{enumerate}
\end{theorem}

With the notation of \th\ref{n-exact-cat-U}, note that $\E$ is an
exact category hence
\th\ref{recognition-thm-exact} implies that $\U$ is an $n$-exact
category.

We now give a concrete example of an $n$-exact category constructed
using \th\ref{n-exact-cat-U}.

\begin{example}
  Let $\coh\PP_K^n$ be the category of coherent sheaves over the
  projective $n$-space over $K$, and let $\Lambda$ be the endomorphism
  algebra of the tilting bundle
  $\OO\oplus\OO(1)\oplus\cdots\oplus\OO(n)$, see \cite{beilinson_coherent_1978}.  It is known that
  $\Lambda$ is an $n$-representation-infinite algebra and that
  $\Pi_{n+1}(\Lambda)$ is noetherian, see
  \cite[Ex. 2.15]{herschend_n-representation_2014}.  Moreover, there
  is an equivalence of triangulated categories
  \[
  \Db(\mod\Lambda) \cong \Db(\coh\PP_K^n).
  \]
  This equivalence induces an equivalence of exact categories between
  the category $\E$ given in \th\ref{n-exact-cat-U} and
  $\vect\PP_K^n$, the category of vector bundles over $\PP_K^n$. Also,
  it induces an equivalence of additive categories
  \[
  \U\cong\add\setP{\OO(i)}{i\in\ZZ}.
  \]
  Finally, \th\ref{recognition-thm-exact} implies that $\U$ is an
  $n$-exact category with respect to the class of all exact sequences
  \[
  \begin{tikzcd}
    0\rar&X^0\rar&X^1\rar&\cdots\rar&X^n\rar&X^{n+1}\rar&0
  \end{tikzcd}
  \]
  with all terms in $\U$.
\end{example}

\subsection{Relative $n$-cluster-tilting subcategories}

Let $\Lambda$ be a finite-dimensional algebra and $T$ a
$\Lambda$-module. It is easy to see that the perpendicular category
\[ {T}^{\perp_{>0}}:=\setP{M\in\mod\Lambda}{\forall k>0,\
  \Ext_\Lambda^k(T,M)=0}
\]
is exact for it is an extension closed subcategory of $\mod\Lambda$.
We have the following result.

\begin{proposition}
  \th\label{relative-2-ct} \cite[Cor. 1.15]{iyama_cluster_2011} Let
  $Q$ be a Dynkin quiver. Then, there exists a tilting $KQ$-module $T$ of
  projective dimension 1 such that ${T}^{\perp_{>0}}$ contains a
  2-cluster-tilting subcategory $\M$.
\end{proposition}

\begin{remark}
  With the notation of \th\ref{relative-2-ct}, the category $\M$ is
  2-exact by \th\ref{recognition-thm-exact}.
\end{remark}

More generally, in \cite[Cor. 1.16]{iyama_cluster_2011} for each $n$
an algebra $\Lambda$ of global dimension at most $n$ such that there
exists a tilting $\Lambda$-module $T$ of finite projective dimension
and ${T}^{\perp_{>0}}$ has an $n$-cluster-titling subcategory was
constructed.

\subsection{Isolated singularities}
\label{sec:isolated-singularities}

Let $R$ be a commutative complete Gorenstein ring of dimension $n$
with residue field $K$. The category of \emph{Cohen-Macaulay
  $R$-modules} is by definition
\[
\CM R:=\setP{M\in\mod R}{\operatorname{depth}M=n}.
\]
Note that $\CM R$ is a Frobenius exact category, see
\cite[Rmk. 4.8]{buchweitz_maximal_1986}.

We remind the reader that $R$ is an \emph{isolated singularity} if $R$
is not regular and for all non-maximal prime ideals
$\mathfrak{p}\subset R$ we have that $R_{\mathfrak{p}}$ is a regular
ring. In this case $\CM R$ has almost-split sequences, see \cite[page
200]{auslander_isolated_1986} and
\cite[Thm. 3.2]{auslander_isolated_1986,yoshino_cohen-macaulay_1990}.

\begin{theorem}
  \th\label{CM-SG} \cite[Thm. 2.5]{iyama_higher-dimensional_2007} and
  \cite[Cor. 8.2]{iyama_mutation_2008} Let $K$ be an algebraically
  closed field of characteristic 0 and set
  $S:=K\llbracket x_0,x_1,\dots,x_n\rrbracket$. Also, let $G$ be a
  finite subgroup of $\SL_{n+1}(K)$ such that no element $\sigma\neq1$
  of $G$ has eigenvalue 1. Then $G$ acts on $S$ in a natural way and
  we define
  \[
  S^G:=\setP{s\in S}{\forall g\in G,\ g\cdot s = s}.
  \]
  Then, the following statements hold:
  \begin{enumerate}
  \item The ring $S^G$ is an isolated singularity.
  \item We have $S\in\CM S^G$.
  \item The category $\add S$ is an $n$-cluster-tilting subcategory of
    $\CM S^G$.
  \end{enumerate}
\end{theorem}

With the notation of \th\ref{CM-SG}, note that
\th\ref{recognition-thm-exact} implies that $\add S$ is an $n$-exact
category with respect to the class of all exact sequences
\[
\begin{tikzcd}
  0\rar&X^0\rar&X^1\rar&\cdots\rar&X^n\rar&X^{n+1}\rar&0
\end{tikzcd}
\]
in $\mod S^G$ with terms in $\add S$.

\subsection{Algebraic $(n+2)$-angulated categories}

In this subsection we revisit the examples of
\cite[Sec. 6.3]{geiss_n-angulated_2013} from the viewpoint of
algebraic $(n+2)$-angulated categories.  We remind the reader that we say
that an object $T$ in a triangulated category $\C$ is
\emph{$n$-cluster-tilting} if $\add T$ is an $n$-cluster-tilting
subcategory of $\C$.

Let $\C$ be a algebraic triangulated category. Hence, there exists a
Frobenius exact category $(\E,\X)$ such that $\sE$ is equivalent to
$\C$ as triangulated categories. It is easy to see that each
$n$-cluster-tilting subcategory of $\C$ lifts to an
$n$-cluster-tilting subcategory of $\E$ by including all the
$\X$-injective objects in $\E$. By
\th\ref{standard-construction-n-exact}, every $(n+2)$-angulated
category constructed using \th\ref{standard-construction-n-angulated}
from an algebraic triangulated category is in turn an algebraic
$(n+2)$-angulated category. Known examples of algebraic
$(n+2)$-angulated categories arising in this way are the following:

\begin{itemize}
\item Let $\C=\C_Q$ be the cluster category associated with an acyclic
  quiver $Q$, see \cite{buan_tilting_2006} for details. It is known
  that $\C$ is an algebraic triangulated category
  \cite{keller_acyclic_2008}. Moreover, a basic 2-cluster-tilting object
  $T\in\C$ satisfies $\Sigma^2 T\cong T$ if and only if $\C(T,T)$ is a
  selfinjective algebra, see \cite[Cor. 3.8]{iyama_stable_2013}. All
  such algebras were classified by Ringel in
  \cite{ringel_self-injective_2008}. In particular, such algebras
  exist only if $Q$ is a Dynkin quiver of type $D$ (including
  $D_3=A_3$). Hence, if $T$ is a 2-cluster-tilting object in $\C$ such
  that $\Sigma^2 T\cong T$, then $\add T\subseteq\C$ is an algebraic
  4-angulated category.
\item Let $\C=\C_\XX$ be the cluster category associated with a
  weighted projective line, see \cite{barot_cluster_2010} for
  details. As in the previous case, a 2-cluster-tilting object
  $T\in\C$ satisfies $\Sigma^2 T\cong T$ if and only if $\C(T,T)$ is a
  selfinjective algebra. All such algebras are classified in
  \cite[Thm. 1.3]{jasso_tau2-stable_2014}. Such algebras exist if and
  only if $\XX$ has tubular weight type $(2,2,2,2)$, $(2,4,4)$, or
  $(2,3,6)$. If $T$ is a 2-cluster-tilting object in $\C$ such that
  $\Sigma^2 T\cong T$, then $\add T\subseteq\C$ is an algebraic
  4-angulated category.
\item Let $\Lambda$ be the preprojective algebra of type $A_n$. Recall
  that $\mod\Lambda$ is a Frobenius abelian category hence the stable
  category $\underline{\mod}\Lambda$ is triangulated. It is known that
  the standard 2-cluster-tilting $\Lambda$-module $T$ corresponding to
  the linear orientation of $A_n$ satisfies $\mho^2(T)\cong T$, see
  \cite{geiss_auslander_2007}. It follows that $\add
  T\subseteq\mod\Lambda$ is a Frobenius 2-exact category and thus
  $\underline{\add}\,T\subseteq\underline{\mod}\Lambda$ is an
  algebraic 4-angulated category.
\item Let $\Lambda$ be an $n$-representation-finite algebra. Then,
  \cite[Cor. 3.7]{iyama_stable_2013} implies that the canonical
  $n$-cluster-tilting object $\pi\Lambda$ in the associated Amiot
  $n$-cluster category $\C$ is stable under the Serre functor
  $\Sigma^n$. It is known that $\C$ is an algebraic triangulated
  category, see \cite[Thm. 4.15]{iyama_stable_2013}, hence
  $\add\pi\Lambda\subseteq\C$ is an algebraic $(n+2)$-angulated
  category.
\end{itemize}
We refer the reader to \cite[Sec. 6]{geiss_n-angulated_2013} for more
details.


\bibliographystyle{abbrv} \bibliography{zotero}

\end{document}